\documentclass[12pt]{article}
\usepackage {epsfig,amsmath,euscript}
\usepackage[latin1]{inputenc}
\usepackage[english]{babel}
\usepackage{color}
\usepackage{amsmath}
\usepackage{amssymb}
\usepackage[scr=boondoxo]{mathalfa}
\usepackage[amsmath,thmmarks,standard,thref]{ntheorem}
\usepackage{ae}
\usepackage{subfigure}
\usepackage[T1]{fontenc}
\usepackage{graphicx}
\usepackage{epstopdf}
\usepackage{soul}
\usepackage{color}
\sethlcolor{yellow}
\definecolor{rred}{rgb}{0.7,0.0,0.2}
\definecolor{bblue}{rgb}{0.2,0.0,0.7}
\newcommand\ed[1]{{\color{black} #1}}

\newcommand{\secref}[1]{\ref{sec:#1}}

\newcommand{\seclab}[1]{\label{sec:#1}}
\newcommand{\eqlab}[1]{\label{eq:#1}}
\renewcommand{\eqref}[1]{(\ref{eq:#1})}
\newcommand{\eqsref}[2]{(\ref{eq:#1}) and~(\ref{eq:#2})}

\newcommand{\figref}[1]{Fig.~\ref{fig:#1}}
\newcommand{\figlab}[1]{\label{fig:#1}}
\newcommand{\propref}[1]{Proposition~\ref{proposition:#1}}
\newcommand{\proplab}[1]{\label{proposition:#1}}
\newcommand{\lemmaref}[1]{Lemma~\ref{lemma:#1}}
\newcommand{\lemmalab}[1]{\label{lemma:#1}}
\newcommand{\remref}[1]{Remark~\ref{remark:#1}}
\newcommand{\remlab}[1]{\label{remark:#1}}
\newcommand{\thmref}[1]{Theorem~\ref{theorem:#1}}
\newcommand{\thmlab}[1]{\label{theorem:#1}}

\newcommand{\appref}[1]{Appendix~\ref{app:#1}}

\newcommand{\applab}[1]{\label{app:#1}}
\usepackage{tikz}

\title{Le canard de Painlev\'e}

\author{K. Uldall Kristiansen and S. J. Hogan\thanks{K. Uldall Kristiansen: Department of Applied Mathematics and Computer Science, Technical University of Denmark, 2800 Kgs. Lyngby, DK. S. J. Hogan: Department of Engineering Mathematics, University of Bristol, Bristol BS8 1UB, United Kingdom.  }}
\begin{document}
\maketitle

\begin{abstract}
We consider the problem of a slender rod slipping along a rough surface. Painlev\'e \cite{Painleve1895, Painleve1905a,Painleve1905b} showed that the governing rigid body equations for this problem can exhibit multiple solutions (the {\it indeterminate} case) or no solutions at all (the {\it inconsistent} case), provided the coefficient of friction $\mu$ exceeds a certain critical value $\mu_P$. Subsequently G\'enot and Brogliato \cite{GenotBrogliato1999} proved that, from a consistent state, the rod cannot reach an inconsistent state through slipping. Instead there is a special solution for $\mu>\mu_C>\mu_P$, with $\mu_C$ a new critical value of the coefficient of friction, where the rod continues to slip until it reaches a singular ``$0/0$'' point $P$. Even though the rigid body equations can not describe what happens to the rod beyond the singular point $P$, it is possible to extend the special solution into the region of indeterminacy. This extended solution is very reminiscent of a {\it canard} \cite{Benoit81}. To overcome the inadequacy of the rigid body equations beyond $P$, the rigid body assumption is relaxed in the neighbourhood of the point of contact of the rod with the rough surface. Physically this corresponds to assuming a small compliance there. It is natural to ask what happens to both the point $P$ and the special solution under this regularization, in the limit of vanishing compliance.

In this paper, we prove the existence of a canard orbit in a reduced $4D$ slow-fast phase space, connecting a $2D$ focus-type slow manifold with the stable manifold of a $2D$ saddle-type slow manifold. The proof combines several methods from local dynamical system theory, including blowup. The analysis is not standard, since we only gain ellipticity rather than hyperbolicity with our initial blowup. 
\end{abstract}

\pagestyle{myheadings}
\thispagestyle{plain}

\section{Introduction}\seclab{intro}
In a series of classical papers, Painlev\'e \cite{Painleve1895, Painleve1905a,Painleve1905b} showed that the governing equations for a slender rod slipping along a rough surface (see  \figref{fig:rod}) can exhibit multiple solutions (the {\it indeterminate} case) or no solutions at all (the {\it inconsistent} case), provided the coefficient of friction $\mu$ exceeds a certain critical value $\mu_P$. In the intervening years, a large number of authors \cite{BlumenthalsBrogliatoBertails2016, Brogliato1999, ChampneysVarkonyi2016, Stewart2000} have considered different aspects of these {\it Painlev\'e paradoxes}, which have been shown to occur in many important engineering systems \cite{LeineBrogliatoNijmeijer2002, LiuZhaoChen2007, NeimarkFufayev1995, Or2014, OrRimon2012, WilmsCohen1981, WilmsCohen1997, ZhaoLiuMaChen2008}. 

The theoretical study of Painlev\'e paradoxes received a great boost with the work by G\'enot and Brogliato \cite{GenotBrogliato1999}, who discovered a new critical value of the coefficient of friction $\mu_C>\mu_P$. They proved that, from a consistent state, the rod cannot reach an inconsistent state through slipping. Instead, the rod will either stop slipping and stick or it will lift-off from the surface. For $\mu>\mu_C$, these cases are separated by a special solution where the rod slips until it reaches a singular point $P$ corresponding a ``$0/0$''-singularity in the equations of motion. Beyond $P$, the rigid body equations are unable to predict what happens. Nevertheless, it is possible to extend the special solution beyond the singular point $P$ into the region of indeterminacy. Therefore this extended solution is very reminiscent of a {\it canard} \cite{Benoit81} that occurs at folded equilibria in $(2+1)$-slow-fast systems\footnote{An $(n+m)$-slow-fast system \cite{kuehn2015} is a dynamical system with $n$ slow variables and $m$ fast variables.} \cite{szmolyan_canards_2001,wechselberger_existence_2005} and in the two-fold of piecewise smooth (PWS) systems \cite{desroches_canards_2011, krihog, krihog2, krihog3}. 

Ever since the time of Painlev\'e, there have been attempts to resolve the paradoxes by including more physics into the rigid body formalism. Lecornu \cite{Lecornu1905} proposed that a jump in vertical velocity would allow for an escape from an inconsistent, horizontal velocity, state.  This jump has been called {\it impact without collision} (IWC) \cite{GenotBrogliato1999}, {\it tangential impact} \cite{Ivanov1986} or {\it dynamic jamming} \cite{OrRimon2012}.  During (the necessarily instantaneous) IWC, the governing equations of motion must be expressed in terms of the normal impulse, rather than time \cite{Darboux1880, Keller1986}. But this approach can produce contradictions, such as an apparent energy gain in the presence of friction \cite{Brach1997, Stronge2015}. 

Another possible way to resolve the Painlev\'e paradox is to relax the rigid body assumption in the neighbourhood of the contact point. Physically this corresponds to assuming a small compliance, usually modelled as a spring, with large stiffness and (possibly) damping. Dupont and Yamajako \cite{DupontYamajako1997} appear to be the first to show that the classical Painlev\'e problem with compliance could then be written as a slow-fast system. They showed that the fast subsystem is unstable in the Painlev\'e paradox. Song {\it et al.} \cite{SongKrausKumarDupont2001} extended this work and established conditions under which the fast solution can be stabilized. Zhao {\it et al.} \cite{ZhaoLiuChenBrogliato2015} considered the example in \figref{fig:rod} and regularized the equations by assuming a compliance that consisted of an {\it undamped} spring. They gave estimates for the time taken in the resulting stages of the dynamics. Neimark and Smirnova \cite{NeimarkSmirnova2000, NeimarkSmirnova2001} considered a different type of regularization in which the normal and tangential reactions take (different) finite times to adjust. Their results showed a strong dependence on the ratio of these times. More recently, the current authors presented \cite{hogkri} the first {\it rigorous} analysis of compliant IWC in both the inconsistent and indeterminate cases and gave explicit asymptotic expressions in the limiting cases of small and large damping. For the indeterminate case, we presented a formula for conditions that separate compliant IWC and lift-off. 

In this paper, we consider the dynamics of the special solution (canard) around $P$ in the presence of compliance. This will give rise to a (2+2)-slow-fast system with small parameter $\varepsilon$ being the inverse square root of the stiffness associated with the compliance. Slow-fast systems receive an enormous amount of attention, since they occur naturally in many biological and engineering systems. As the recent book by Kuehn \cite{kuehn2015}, and others, have made clear, a major boost to the subject came about following the seminal work of Fenichel \cite{fen1, fen2, fen3} and the development of geometric singular perturbation theory (GSPT) \cite{jones_1995}. Fenichel theory and GSPT work away from critical points, such as folds and singularities (specifically any point where hyperbolicity is lost). At such points, GSPT has to be extended. Such an extension was made possible by the pioneering work of Dumortier and Roussarie \cite{dumortier_1991, dumortier_1993, dumortier_1996}. Their approach, known as blowup, was further developed by Krupa and Szmolyan \cite{krupa_extending_2001, krupa_extending2_2001, krupa_relaxation_2001} to a form where it became popular and widely applicable to many different and challenging problems\footnote{The present authors have successfully applied GSPT \cite{krihog, krihog2, krihog3} to piecewise smooth (PWS) problems \cite{Bernardo08}, where the underlying vector fields have jumps or discontinuities that are then regularized.}.


It is also possible to study canards using blowup. Originally discovered by Beno\^it {\it et al.} \cite{Benoit81}, these are solutions to singularly perturbed problems that initially follow a stable manifold, then pass through a critical point, before following an unstable manifold for a non-vanishing period of time. Their study was significantly aided by the development of blowup, where the critical point had, until then, proved a barrier to the use of GSPT. Canards are important since they are crucial to the so-called {\it canard explosion} \cite{bro2,krupa_relaxation_2001}, in which limit cycles are transformed, under parameter variation, into relaxation oscillations. The change happens over an exponentially small parameter range\footnote{Canards are known to occur in PWS systems and their fate under regularization has been studied \cite{desroches_canards_2011, krihog, krihog2, krihog3}.}. 

We will apply blowup to the compliant (2+2)-slow-fast system and rigorously show the existence of a canard that connects, in the $4D$ phase space, a $2D$ attracting Fenichel slow manifold of focus-type with the stable manifold of a $2D$ saddle-type slow manifold (\thmref{thm:main}). The singular point $P$ of the rigid body system becomes a line of Bogdanov-Takens (BT) points \cite{perko2001a} of the layer problem associated with the regularization with a nilpotent $2\times 2$ Jordan block. The mathematical difficulties in proving \thmref{thm:main} are as follows. In the scaling chart associated with the blowup, we obtain the following equation
 \begin{align}
  {\tilde y}'''(\theta_2) = \theta_2 {\tilde y}'(\theta_2)+(1-\xi){\tilde y}(\theta_2),\quad \theta_2\in \mathbb R,\eqlab{3rdChart2intro}
 \end{align}
for $\xi\in (0,1)$. The third order linear ODE \eqref{3rdChart2intro} appears to have been first considered by Langer \cite{Langer1995a,Langer1995b}, as an example of an ODE in which the characteristic equation can have three coincident roots. See also \cite{Vallee1999,Vallee2004}. We will therefore refer to this equation as \textit{Langer}'s equation\footnote{We are aware of a different {\it Langer's equation} in the theory of spinodal decomposition (J. S. Langer Theory of spinodal decomposition in alloys. {\it Ann. Phys.} 65:53-85, 1971). However this other equation post-dates \eqref{3rdChart2intro}.} henceforth.  {Langer}'s equation also appears in \cite{NordVarChamp17}, and so the Painlev\'e paradox would seem to be its first physically important application. 

We will show (\lemmaref{x21x22x23} in Section \secref{sec:kappa3}) that Langer's equation has a distinguished solution 
$$\textnormal{La}_{\xi}(\theta_2)= \int_0^\infty e^{-\tau^3/3+\theta_2 \tau} \tau^{-\xi}d\tau,$$ which spans all solutions that are non-oscillatory for $\theta_2\rightarrow -\infty$. All other solutions, spanned by special functions $\textnormal{Lb}_\xi(\theta_2)$, $\textnormal{Lc}_\xi(\theta_2)$, introduced in \lemmaref{x21x22x23}, are oscillatory as $\theta_2\rightarrow -\infty$. Therefore, as a consequence, we only gain ellipticity (rather than hyperbolicity) of the focus-type slow manifold of the blowup of $P$ (upon desingularization). So we apply normal form transformations - to eliminate fast oscillations - that will subsequently allow for an additional application of a (polar) blowup transformation. We gain hyperbolicity by this second transformation and are therefore able to extend Fenichel's slow manifold as a center-like manifold up close to the point $P$ (see \propref{Ma1k1} in \appref{appNF}). But interestingly, this manifold does not extend all the way to the scaling chart. There is a gap which we can only cover by estimation of the forward flow. This brings us up close to the distinguished non-oscillatory solutions in the scaling chart for $0<\varepsilon \ll 1$. We then complete our proof by using properties of $\textnormal{La}_{\xi}(\theta_2)$ for $\theta_2\rightarrow \infty$.  

The paper is organized as follows. In Section \secref{classic}, we introduce the classical Painlev\'e problem, outline some of the results due to G\'enot and Brogliato \cite{GenotBrogliato1999}, show that $\mu_C=\frac{2}{\sqrt{3}}\mu_P$ for a large class of rigid bodies, introduce compliance and, in \eqref{eq.slowfast2}, present our (2+2)-slow-fast system. In Section \secref{mainResults}, we summarise our main result, \thmref{thm:main}. The rest of the paper is devoted to the mathematical proof of \thmref{thm:main}, using blowup \cite{krupa_extending_2001, krupa_extending2_2001, krupa_relaxation_2001}. Section \secref{blowup} sets up the initial blowup. The exit chart is considered in Section \secref{sec:kappa3}, the scaling chart in Section \secref{sec:kappa2} and the entry chart in Section \secref{sec:kappa1}. Each of Sections \secref{sec:kappa3} to \secref{sec:kappa1} contains a number of technical Propositions whose details are confined to the Appendices. We discuss our results and outline our conclusions in Section \secref{conclusions}.

\section{The classical Painlev\'e problem}\seclab{classic}

The governing equations\footnote{Painlev\'e \cite{Painleve1895} originally studied a planar box sliding down an inclined plane. Nevertheless, as noted in \cite[p. 539]{ChampneysVarkonyi2016}, the problem most closely associated with Painlev\'e is the one in \figref{fig:rod}. We will therefore also refer to this as the {\it classical Painlev\'e problem}.} of the rigid rod $AB$ of length $2l$ that slips on a rough horizontal surface, as shown in \figref{fig:rod}, are given by
\begin{eqnarray}\eqlab{eq:dynrod}
m\ddot{X} & = & -F_T, \\ 
m\ddot{Y} & = & -mg + F_N, \nonumber \\
I\ddot{\theta} & = & -l(\cos\theta F_N-\sin\theta F_T). \nonumber
\end{eqnarray}

\begin{figure}[h!] 
\begin{center}
{\includegraphics[width=.5\textwidth]{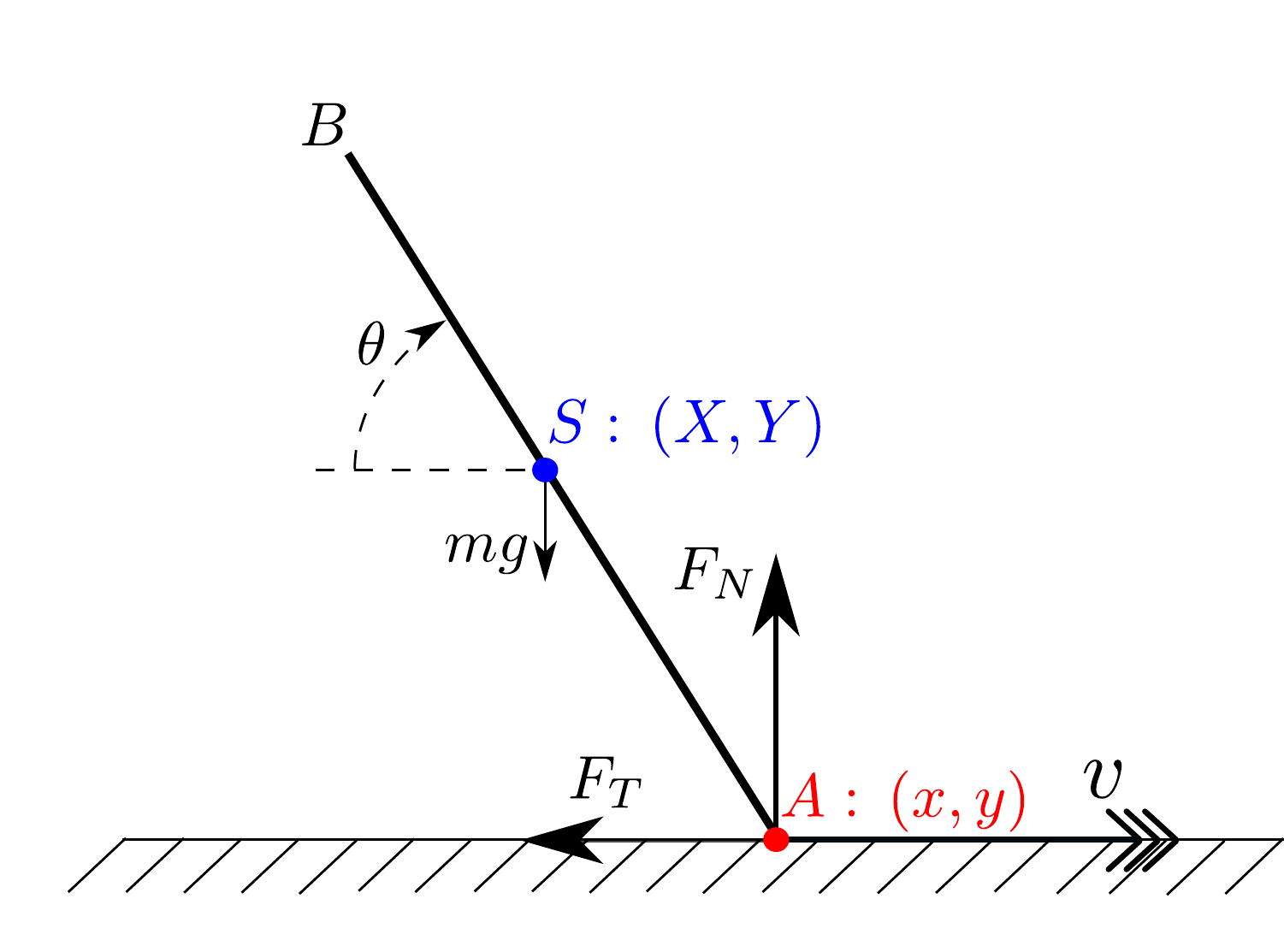}}
\end{center}
 \caption{
 The classical Painlev\'e problem: $g$ is the acceleration due to gravity; the rod has mass $m$, length $2l$, the moment of inertia of the rod about its center of mass $S$ is given by $I$ and its center of mass coincides with its center of gravity. The point $S$ has coordinates $(X,Y)$ relative to an inertial frame of reference $(x,y)$ fixed in the rough surface. The rod makes an angle $\theta$ with respect to the horizontal, with $\theta$ increasing in a clockwise direction. At $A$, the rod experiences a contact force $(F_T, F_N)$, which opposes the motion. 
}
 \figlab{fig:rod}
\end{figure}
From geometry
\begin{eqnarray}\eqlab{eq:coord}
x & = & X + l\cos\theta, \\
y & = & Y - l\sin\theta. \nonumber
\end{eqnarray}
We now define dimensionless variables and parameter $\alpha$ as follows $$l(\tilde X,\tilde Y)=(X,Y), \quad l(\tilde x,\tilde y)=(x,y), \quad mg(\tilde F_T, \tilde F_N)=(F_T, F_N), \quad \tilde t= \omega t, \quad \alpha=\frac{ml^2}{I},$$ where $\omega^2=\frac{g}{l}$. For a uniform rod, $I=\frac{1}{3}ml^2$, and so $\alpha=3$ in this case. 


So for general $\alpha$, by combining \eqsref{eq:dynrod}{eq:coord} and writing everything in terms of the dimensionless variables, and then dropping the tildes, we find  
\begin{eqnarray}\eqlab{eq:dynrodsc}
\ddot{x} & = & -\dot{\theta}^2\cos \theta +\alpha\sin\theta\cos\theta F_N-(1+\alpha\sin^2\theta)F_T, \\ 
\ddot{y} & = & -1 +\dot{\theta}^2\sin\theta + (1+\alpha\cos^2\theta) F_N -\alpha\sin\theta\cos\theta F_T,  \nonumber \\
\ddot{\theta} & = & -\alpha(\cos\theta F_N-\sin\theta F_T). \nonumber
\end{eqnarray}
We assume Coulomb friction between the rod and the surface. So, when $\dot{x} =v \ne 0$, we set
\begin{equation}\eqlab{eq:Coulomb}
F_T = \mu \textnormal{sign}(\dot{x})F_N,
\end{equation}
where $\mu$ is the coefficient of friction. We introduce $\phi=\dot \theta$, $w=\dot y$, $v=\dot x$ and substitute \eqref{eq:Coulomb} into \eqref{eq:dynrodsc} to get
\begin{eqnarray}\eqlab{eq:Painleve}
\dot{x} & = & v, \\
\dot{v} & = & a(\theta,\phi) + q_{\pm}(\theta)F_N, \nonumber \\
\dot{y} & = & w, \nonumber \\
\dot{w} & = & b(\theta,\phi) + p_{\pm}(\theta)F_N, \nonumber \\
\dot{\theta} & = & \phi, \nonumber \\
\dot{\phi} & = & c_{\pm}(\theta)F_N, \nonumber 
\end{eqnarray}
where
\begin{eqnarray}\eqlab{eq:coeffs}
a(\theta,\phi) & = & -\phi^2\cos \theta , \\
b(\theta,\phi) & = & -1 +\phi^2\sin\theta, \nonumber \\
q_{\pm}(\theta) & = & \alpha\sin\theta\cos\theta \mp \mu(1+\alpha\sin^2\theta), \nonumber \\
p_{\pm}(\theta) & = & 1+\alpha\cos^2\theta \mp \mu\alpha\sin\theta\cos\theta, \nonumber \\
c_{\pm}(\theta) & = & -\alpha(\cos\theta \mp \mu\sin\theta) \nonumber 
\end{eqnarray}
for the configuration in \figref{fig:rod}. The suffix $\pm$ corresponds to $\dot{x} = v \gtrless 0$ respectively. 

We will suppose that the rod is initially moving to the right at time $t=0$: $$\dot x(0) = v(0)=v_0>0.$$  Then if, at some later time $t=T$, $\dot{x}(T) = v(T) = 0$ and $\dot v$ for $v\gtrless 0$ both oppose the discontinuity set $v=0$: \ed{$\dot v<0$ for $v=0^+$ and $\dot v>0$ for $v=0^-$}, the required vector-field is obtained by Filippov's method \cite{filippov1988differential}, see \cite{hogkri}. We call this dynamics {\it sticking}.  Note that by \eqref{eq:coeffs} it follows that $q_+<0$ whenever $p_+\approx 0$.

We now need to determine $F_N$, using either the constraint-based method, which leads to a Painlev\'e paradox, or the compliance-based method, which is used in this paper. 

\subsection{Constraint-based method}\seclab{sec:constraintbasedmethod}
In order to maintain the constraint $y=0$, at most one of $F_N$ and $y$ can be positive \cite{hogkri} and so $F_N$ and $y$ must satisfy  
\begin{equation}\eqlab{eq:comprel}
0 \le F_N \perp y \ge 0.
\end{equation}

Hence, from \eqref{eq:Painleve}, if $\dot w=0$, then
 \begin{align}
  F_N = -\frac{b}{p_+},\eqlab{eq:kk}
 \end{align}
since $v>0$. Then we have a reduced, decoupled system in the $(\theta,\phi)$-plane:
\begin{eqnarray}\eqlab{eq:PainleveReduced}
\dot{\theta} & = & \phi, \\
\dot{\phi} & = & -\frac{c_{\pm}(\theta) b(\theta,\phi)}{p_+(\theta)}, \nonumber 
\end{eqnarray}
and the variables $x$ and $v$ satisfy
\begin{eqnarray}\nonumber
\dot{x} & = & v, \\
\dot{v} & = & a(\theta,\phi) -\frac{q_{\pm}(\theta) b(\theta,\phi))}{p_+(\theta)}, \nonumber
\end{eqnarray}
which can be directly integrated once $\theta$ and $\phi$ are known. 

For the system in \figref{fig:rod}, Painlev\'e paradoxes occur when $v>0$ and $\theta \in (0,\frac{\pi}{2})$, provided $p_+(\theta)<0$ \cite{hogkri}. 
From \eqref{eq:coeffs}, it is straightforward to show that $p_+(\theta)<0$ requires
\begin{equation}\eqlab{eq:mucrit}
\mu > \mu_P (\alpha) \equiv \frac{2}{\alpha}\sqrt{1+\alpha}.
\end{equation}
Then a Painlev\'e paradox occurs for $\theta \in (\theta_{1},\theta_{2})$ where
\begin{eqnarray}\eqlab{eq:thetacrit}
\theta_{1}(\mu,\alpha) & = & \arctan \frac{1}{2}\left ( \mu\alpha - \sqrt{\mu^2\alpha^2-4(1+\alpha)} \right ), \\
\theta_{2}(\mu,\alpha) & = & \arctan \frac{1}{2}\left ( \mu\alpha + \sqrt{\mu^2\alpha^2-4(1+\alpha)} \right ). \nonumber
\end{eqnarray}
For a uniform rod, $\mu_P(3) = \frac{4}{3}$.  The dynamics in the $(\theta, \phi)$-plane\footnote{G\'enot and Brogliato \cite{GenotBrogliato1999} plot in Figure 2 the {\it unscaled} angular velocity $\omega\phi$ {\it vs.} $\theta$ where $\omega=\sqrt{\frac{g}{l}}$, for the case $g=9.8$ $\textnormal{ms}^{-2}$, $l=1$ m.} are shown in \figref{fig:GB} for $\alpha=3$ and $\mu=1.4$. The region $\theta\in (\theta_1,\theta_2)$ where $p_+(\theta)<0$ is coloured green and purple. In the green region, $b<0$ hence $F_N$ in \eqref{eq:kk} is negative. This is the {\it inconsistent} (or {\it non-existent}) mode of the Painlev\'e paradox. In the purple region, $b>0$. From \eqref{eq:Painleve}, $b$ is the free acceleration of the end of the rod. Lift-off into $y>0$ is therefore always possible within this region. At the same time $F_N$ in \eqref{eq:kk} is positive. Hence, the purple region is the {\it indeterminate} (or {\it non-unique}) mode of the Painlev\'e paradox.
\begin{figure}[h!] 
\begin{center}
{\includegraphics[width=.7\textwidth]{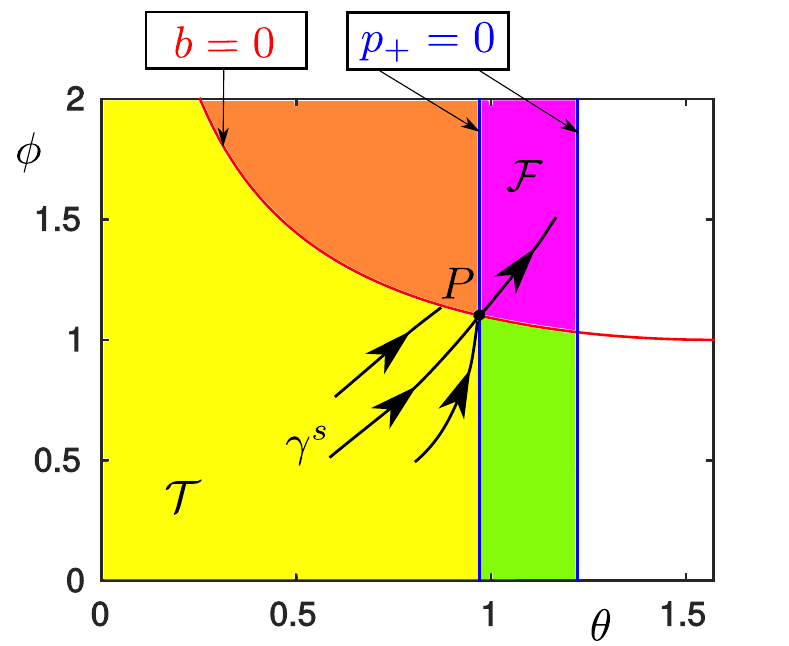}}
\end{center}
 \caption{The $(\theta, \phi)$-plane for the classical Painlev\'e problem of \figref{fig:rod}, for $\alpha=3$ and $\mu=1.4$. The point $P$ has coordinates $(\theta_1, \sqrt{\csc\theta_1})$, where $\theta_1$ is given in \eqref{eq:thetacrit}. $\gamma^{s}$ is defined in \eqref{gammas}. In the purple region, $b >0, \enskip p_+<0$ and the dynamics is indeterminate (non-unique). In the orange region, $b >0, \enskip p_+>0$ and the rod lifts off the rough surface. In the yellow region, $b <0, \enskip p_+>0$ and the rod moves (slips) along the surface. Finally, in the green region, $b <0, \enskip p_+<0$ and the dynamics is inconsistent; there exists no positive value of $F_N$, even though the constraint $y=0$ is satisfied, contrary to \eqref{eq:comprel}. } 
 \figlab{fig:GB}
\end{figure}
The lines $p_+(\theta_{1,2})=0$ intersect $b(\theta,\phi)=0$ at four points: $\phi^{\pm}_{1,2} = \pm \sqrt{\csc\theta_{1,2}}$. The point 
\begin{equation}
P: (\theta,\phi) = (\theta_1, \sqrt{\csc\theta_1})
\eqlab{eq:P}
\end{equation} 
is the most important \cite{GenotBrogliato1999}. Then we have the following:
\begin{proposition}\proplab{Wss}
 Consider
 \begin{align}
 \mathcal F = \{(\theta,\phi)\vert b(\theta,\phi)>0, \,p_+(\theta)<0\}\cap \mathcal U,  \eqlab{eq:FTregions}\\
  \mathcal T = \{(\theta,\phi)\vert b(\theta,\phi)<0, \,p_+(\theta)>0\}\cap \mathcal U, \nonumber  
 \end{align} 
 where $\mathcal U$ is a small neighborhood of $P\in \{(\theta,\phi)\vert b(\theta,\phi) = 0,\,p_+(\theta)=0\}$. 
Then the point $P$ is a stable node of \eqref{eq:PainleveReduced} within $\mathcal U$ with respect to a new time $\tau$, satisfying
\begin{align}
 \frac{d\tau }{dt} = \frac{1}{p_+(\theta)}.
 \eqlab{eq:tau}
\end{align}
In particular, if
\begin{equation}
\mu> \mu_C(\alpha)=\frac{4}{\alpha}\sqrt{\frac{\alpha+1}{3}}.
\eqlab{eq:muCgen}
\end{equation}
then there exists a constant $c>0$ sufficiently small and a smooth $1D$ invariant, strong stable manifold $\gamma^s$ (of \eqref{eq:PainleveReducedDS} below) within ${\mathcal T}\cup P\cup \mathcal F$:
\begin{align}
 \gamma^{s}:\quad \phi = m_{ss}(\theta)\equiv (1-\xi)^{-1} s\theta+\mathcal O(\theta^2),\quad \theta\in [\theta_1-c,\theta_1+c],\eqlab{gammas}
\end{align}
tangent to 
\begin{align}
  (1-\xi,s)^{T},\eqlab{strongDirection}
\end{align}
at $P$, where 
\begin{align}
\xi &= \lambda_1^{-1} \lambda_2\in (0,1),\eqlab{eq.xi}\\
\quad  s &= -\lambda_1^{-1} (\phi_1^+)^2>0,\eqlab{eq.s}
\end{align}
and $\lambda_{1,2}<0$ are defined in \eqref{lambda12} below. Every point in the subset
\begin{align}
 \mathcal L=\{(\theta,\phi)\in \mathcal T\vert \phi > m_{ss}(\theta)\},
 \eqlab{eq:L}
\end{align}
leaves $\mathcal T$, under the forward flow of \eqref{eq:PainleveReduced},  through the boundary defined by $b(\theta,\phi)=0$ while every point in the subset
\begin{align}
 \mathcal S=\{(\theta,\phi)\in \mathcal T\vert \phi <m_{ss}(\theta)\},
  \eqlab{eq:S}
\end{align}
leaves $\mathcal T$ through $P$ tangent to the vertical boundary $p(\theta,\phi)=0$.
\end{proposition}

\begin{proof}
We include a simple proof of this proposition. In terms of $\tau$ we obtain from \eqsref{eq:PainleveReduced}{eq:tau}
\begin{eqnarray}\eqlab{eq:PainleveReducedDS}
\frac{d\theta}{d\tau} & = & p_+(\theta) \phi, \\
\frac{d\phi}{d\tau} & = & -{c_{+}(\theta) b(\theta,\phi)}. \nonumber 
\end{eqnarray}
The point $P=(\theta_1,\phi_1^+)$ given in \eqref{eq:P} is a fixed point of these equations. Linearization about $P$ gives the Jacobian
\begin{align*}
\begin{pmatrix}
 p_+'(\theta_1) \phi_1^+ &0 \\
 -c_{+}(\theta_1) \partial_\theta b(\theta_1,\phi_1^+) & -c_{+}(\theta_1) \partial_\phi b(\theta_1,\phi_1^+)
\end{pmatrix} = \begin{pmatrix}
 p_+'(\theta_1) \phi_1^+ &0 \\
 -(\phi_1^+)^2 & -2\tan (\theta_1) \phi_1^+
\end{pmatrix},
\end{align*}
since $c_+(\theta_1)  = \sec \theta_1$, which has eigenvalues
\begin{align}
 \lambda_1 = p_+'(\theta_1) \phi_1^+,\quad \lambda_2 = -c_{+}(\theta_1) \partial_\phi b(\theta_1,\phi_1^+) = -2\tan (\theta_1) \phi_1^+,\eqlab{lambda12}
\end{align}
that are both negative in the range of $\theta_1$ that contains the Painlev\'e paradox. 

Simple algebraic manipulations show that $$\xi=\lambda_1^{-1} \lambda_2 < 1$$ if and only if 
\begin{align*}
 \text{arctan}\,\left({\frac {2+\sqrt {3\,{\mu}^{2}+4}}{3\mu}}\right) < \theta_1.
\end{align*}
By combining this expression with \eqref{eq:thetacrit} for $\theta_1$, a lengthy calculation then shows, for general $\alpha$, that $\lambda_1^{-1} \lambda_2 < 1$ if and only if 
\eqref{eq:muCgen} holds. 
The eigenvectors associated with $\lambda_1$ and $\lambda_2$ are
\begin{align*}
 (1-\lambda_1^{-1}\lambda_2,-\lambda_1^{-1} (\phi_1^+)^2)^{T}=(1-\xi,s)^{T} \quad \textnormal{and} \quad (0,1)^{T},
\end{align*}
using \eqsref{eq.xi}{eq.s},
respectively.
%
\end{proof}

\begin{remark}
The main results in \propref{Wss} were given in G\'enot and Brogliato \cite{GenotBrogliato1999}, except for the inequality \eqref{eq:muCgen}, which does not seem to have appeared in the literature before\footnote{The result $\mu_C(3)=\frac{8}{3\sqrt{3}}$ {\it does} appear in \cite{GenotBrogliato1999}.}. When $\mu= \mu_C(\alpha)$, it can be shown that $\tan \theta_1 = \sqrt{\frac{\alpha+1}{3}}$. From \eqsref{eq:mucrit}{eq:muCgen}, we have,
\begin{align}\eqlab{muCmuP}
\mu_C = \frac{2}{\sqrt{3}}\mu_P, \quad \forall \, \alpha,
\end{align}
{\it independent} of $\alpha$. This remarkable result also appears to be new.
\end{remark}

\begin{remark}
 Since $p_+<0$ within $\mathcal F$, the new time $\tau$ reverses direction there. Therefore the manifold $\gamma^{s}$ gives a solution of \eqref{eq:PainleveReduced} with respect to the original time having a smooth continuation through the singularity $P$ (as indicated in \figref{fig:GB}). We shall refer to this as a {\it strong singular canard}. 
\end{remark}

\begin{remark}\remlab{weak} For $\mu_P<\mu< \mu_C$ so that $\xi>1$, then the direction $\theta=\theta_1$ is strong while \eqref{strongDirection} is weak. However, by evaluating the slope of the curve $b(\theta,\phi)=0$ at the point $P$ and comparing the result with $s/(1-\xi)$, it is straightforward to show that the weak eigendirection \eqref{strongDirection} is not contained within $\mathcal F\cup P\cup \mathcal T$. The reduced problem is only defined within $\mathcal F\cup \mathcal T$ and hence the classical Painlev\'e problem does therefore not support \textit{weak singular canards}. 
\end{remark}

The implications of \propref{Wss} are as follows. The dynamics cannot cross $p_+=0$ unless also $b=0$. Furthermore, initial conditions within $\mathcal L$, as defined in \eqref{eq:L}, lift off at $b(\theta,\phi)=0$. On the other hand, orbits within $\mathcal S$, defined in \eqref{eq:S}, are tangent to $p_+(\theta,\phi)=0$ at $P$. Therefore the equilibrium value of the normal component of the contact force $F_N$, given in \eqref{eq:kk}, becomes singular as $(\theta,\phi)$ approaches $P$. 
In general, points reach $P$ in finite (original) time $t$ \cite{ChampneysVarkonyi2016,GenotBrogliato1999,NordVarChamp17}. But it can happen that the rod sticks before reaching $P$, for sufficiently small $v(0)$, as follows: close to $P \equiv (\theta_1,\phi_1^+)$ we have
\begin{align*}
 F_N \approx \frac{\partial_\theta b(\theta_1,\phi_1^+) (\theta-\theta_1) + \partial_\phi b(\theta_1,\phi_1^+) (\phi-\phi_1^+)}{\partial_\theta p_+(\theta_1) (\theta-\theta_1)} = \frac{\partial_\theta b(\theta_1,\phi_1^+)}{\partial_\theta p_+(\theta_1)} + \frac{\partial_\phi b(\theta_1,\phi_1^+)}{\partial_\theta p_+(\theta_1)}\frac{(\phi-\phi_1^+)}{(\theta-\theta_1)}\rightarrow \infty,
\end{align*}
as $(\theta,\phi)\rightarrow P$ through the forward flow of \eqref{eq:PainleveReducedDS}. 
But then, since $q_+<0$ in \eqref{eq:Painleve} near $p_+=0$, $\dot v\approx q_+F_N=-q_+b/p_+<0$ to leading order, or alternatively $dv/d\tau \approx -q_+ b$ with respect to the time in \eqref{eq:PainleveReducedDS}. Now $b=b(\theta(\tau),\phi(\tau))$ decays exponentially since $\theta$ and $\phi$ converge exponentially to the stable node $P$. Furthermore, $q_+=q_+(\theta(\tau))$ is bounded. Therefore the improper integral $\lim_{\tau\rightarrow \infty} v(\tau)=v(0)+\int_0^\infty (dv/d\tau) d\tau\approx v(0)-\int_0^\infty q_+ bd\tau $ converges. If this integral is negative, then $v(T)=0$ for some $T>0$ and sticking ($\dot v=0$) occurs, described by the Filippov vector-field \cite{filippov1988differential}.

%

%

As mentioned in the introduction, the rigid body equations \eqref{eq:dynrod} are unable to address what happens beyond $P$. Therefore we will now relax the rigid body assumption by adding compliance. 

\subsection{Compliance-based method}\seclab{sec:compliancebasedmethod}
Following \cite{DupontYamajako1997,McClamroch1989}, we assume that there are small excursions (compliance) into $y<0$ in the neighbourhood of the point $A$ between the rod and the surface, when they are in contact (see \figref{fig:rod}). Then we assume that the non-negative normal force $F_N$ takes the form
\begin{eqnarray}\eqlab{eq:compliance}
F_N(y,w)  =\left[\varepsilon^{-1}F(\varepsilon^{-1}y,w)\right]=\left\{ \begin{array}{cc}
                    0 & \text{for} \quad  y>0\\
                    \max \{\varepsilon^{-1} F(\varepsilon^{-1} y,w), 0 \} & \text{for} \quad y \le 0,
                    \end{array}\right.
\end{eqnarray}
for $\varepsilon>0$, where the operation $\left[\cdot\right]$ is defined by the last equality and $F$ is assumed to be smooth with
\begin{align}
 F(\hat y,w) = -\hat y-\delta w+\mathcal O ((\hat y+w)^2).
 \eqlab{eq:F}
\end{align}
The motivation for \eqref{eq:compliance} is as follows. $F_N=0$ for $y>0$ because the rod is not in contact with the surface. The quantities $\lim_{y\rightarrow 0^-}(-\partial_y F_N(y,0))=\varepsilon^{-2}$ and $\lim_{y\rightarrow 0^-}(-\partial_w F_N(y,0)) = \varepsilon^{-1} \delta $ represent a (scaled) spring constant and damping coefficient. This choice of scaling ensures \cite{DupontYamajako1997,McClamroch1989} that the critical damping coefficient is independent of $\varepsilon$. We are interested in the case when the compliance is very small, so we consider $0<\varepsilon\ll 1$.

The first two equations in \eqref{eq:Painleve} play no role in what follows, so we drop them. Then we combine the remaining four equations in \eqref{eq:Painleve} with \eqref{eq:compliance} to give the following set of governing equations, valid while $v>0$, that we will use in the sequel 
\begin{eqnarray}\eqlab{eq:PainleveEqs}
\dot{y} & = & w, \\
\dot{w} & = & b(\theta,\phi) + p_{+}(\theta)[\varepsilon^{-1} F(\varepsilon^{-1} y,w)], \nonumber \\
\dot{\theta} & = & \phi, \nonumber \\
\dot{\phi} & = & c_{+}(\theta)[\varepsilon^{-1} F(\varepsilon^{-1} y,w)]. \nonumber 
\end{eqnarray}
In our previous paper \cite{hogkri}, we studied this singularly perturbed system in the regions corresponding to the first (purple) $\mathcal F$ and fourth (green) ``quadrants" of \figref{fig:GB} and showed the appearance of IWC. \ed{This provides further evidence that the scaling of the damping in \eqref{eq:F} is the right one: as $\delta \to \infty$ IWC vanishes.} In this paper, we consider the third (yellow) ``quadrant" $\mathcal T$ and focus on the fate of $P$ and $\gamma^{s}$ under regularization. 

\subsection{Slow-fast analysis}
To analyse \eqref{eq:PainleveEqs} we consider the following scaling
\begin{align}\eqlab{slowfastscaling}
 y = \varepsilon^2 y_2,\quad w=\varepsilon w_2,
\end{align}
also used in \cite{hogkri}. Inserting this into \eqref{eq:PainleveEqs}, with \eqref{eq:F}, gives a (2+2)-slow-fast system
\begin{align}
 \dot y_2 &= w_2,\eqlab{eq.slowfast2}\\
 \dot w_2 &= b(\theta,\phi) + p_{+}(\theta) \left[ -y_2-\delta w_2+\varepsilon N(y_2,w_2,\varepsilon)\right],\nonumber\\
 \dot \theta &= \varepsilon \phi,\nonumber\\
 \dot \phi &= \varepsilon c_+(\theta) \left[ -y_2-\delta w_2+\varepsilon N(y_2,w_2,\varepsilon)\right],\nonumber
\end{align}
upon scaling time by $\varepsilon$, where $N(y_2,w_2,\varepsilon)=\mathcal O((y_2+w_2)^2+\varepsilon(y_2+w_2)^3)$ represents higher order terms. Setting $\varepsilon=0$ gives the {\it layer problem}
\begin{align}
 \dot y_2 &= w_2,\eqlab{eq:layerproblem}\\
 \dot w_2 &= b(\theta,\phi) + p_{+}(\theta) \left[ -y_2-\delta w_2\right],\nonumber\\
 \dot \theta &= 0,\nonumber\\
 \dot \phi &= 0,\nonumber
\end{align}
in which both $\theta$, $\phi$ are constant. Undoing the scaling of time by $\varepsilon$ and then setting $\varepsilon=0$ gives the {\it reduced problem}:
\begin{align}
 0 &= w_2,\eqlab{reduced}\\
 0 &= b(\theta,\phi) + p_{+}(\theta) \left[ -y_2-\delta w_2\right],\nonumber\\
 \dot \theta &= \phi,\nonumber\\
 \dot \phi &= c_+(\theta) \left[ -y_2-\delta w_2\right].\nonumber
\end{align}
The reduced problem is only defined on the critical points of the layer problem. 

We now discuss some dynamics of both problems, summarised in \propref{critS} below. Let
 \begin{align}
  g(\theta,\phi) = \frac{b(\theta,\phi)}{p_+(\theta)},\quad (\theta,\phi) \in \mathcal T\cup \mathcal F.\eqlab{gg}
 \end{align}
 where $\mathcal F$, $\mathcal T$ are defined in \eqref{eq:FTregions}. Then we have the following
\begin{proposition}\proplab{critS}
The critical set $S$ of the layer problem \eqref{eq.slowfast2}$_{\varepsilon=0}$ is given by
 \begin{align*}
  S = S_a\cup S_r\cup \widehat P,
 \end{align*}
with
\begin{align}
 S_a&:\quad w_2=0,\,y_2 = g(\theta,\phi),\quad (\theta,\phi)\in \mathcal T,\eqlab{Sar}\\
 S_r&:\quad w_2=0,\,y_2 = g(\theta,\phi),\quad (\theta,\phi)\in \mathcal F,\nonumber\\
 \widehat P&:\quad w_2 = 0,\,y_2\in \mathbb R,\,(\theta,\phi) = P.\nonumber
\end{align}
Here $S_a$ is normally attracting (focus type), $S_r$ is repelling (saddle type) while $\widehat P$ is a line of nonhyperbolic Bogdanov-Takens (BT) fixed points. The reduced flow on $S_{a,r}$ coincides with \eqref{eq:PainleveReduced}. In particular, 
if \eqref{eq:muCgen} holds, then 
$\gamma^s$, given in the $(\theta,\phi)$-plane by \eqref{gammas}, is a solution to the reduced problem \eqref{reduced}
having a smooth continuation through $P$. 
\end{proposition}
\begin{proof}
Straightforward. Linearization of the layer problem \eqref{eq:layerproblem} about $w_2=0,\,y_2 = g(\theta,\phi)$, $(\theta,\phi)\in \mathcal T\cup \mathcal F$, from \eqref{Sar}, gives
\begin{align*}
 \begin{pmatrix}
  0 & 1\\
  -p_{+}(\theta) & -p_{+}(\theta) \delta
 \end{pmatrix}.
\end{align*}
Here we have used that $\left[F\right]=F$ when $F>0$, see \eqref{eq:compliance}. The PWS system \eqref{eq:layerproblem} is therefore smooth in a neighborhood of any point on $S_{a,r}$. 
The eigenvalues are
\begin{align*}
 \lambda_\pm = - \frac12 p_+(\theta)\delta \pm \frac12 \sqrt{-4p_+(\theta)+\delta^2 p_+(\theta)}.
\end{align*}
Expansion about $\theta=\theta_1$, defined in \eqref{eq:thetacrit}, gives
\begin{align}
 \lambda_\pm = -\frac12 p_+'(\theta_1) \delta \Delta \theta(1+\mathcal O(\Delta \theta))\pm \frac12  \sqrt{-4p_+'(\theta) \Delta \theta (1+\mathcal O(\Delta \theta))},\eqlab{eigenvalues}
\end{align}
where $\Delta \theta = \theta-\theta_1$ and then, since $p_+'(\theta_1)<0$, the claims concerning $S_{a,r}$ therefore follow. Similarly, the linearization about any point in $\widehat P$ gives a nilpotent $2\times 2$ Jordan block. 

Inserting $w_2=0,\,y_2 = g(\theta,\phi)$ into \eqref{reduced} gives \eqref{eq:PainleveReduced}. The result therefore follows from \propref{Wss}.
\end{proof}

The strong singular canard $\gamma^s$ connects $S_a$ with $S_r$ through $\widehat P$. It intersects $\widehat P$ in 
\begin{align*}
 \gamma^s \cap \widehat P:\quad y_2 = \frac{\partial_\theta b(\theta_1,\phi_1^+) +\partial_\phi b(\theta_1,\phi_1^+) (1-\xi)^{-1} s }{\partial_\theta p_+(\theta_1)},\, w_2=0,\,\theta= \theta_1,\,\phi=\phi_1^+,
\end{align*}
using \eqref{gammas}, \eqsref{gg}{Sar}. Note again (recall \remref{weak}) that, as opposed to the folded node in classical $(2+1)$-slow-fast systems \cite{szmolyan_canards_2001,wechselberger_existence_2005}, there is no equivalent weak canard in this particular setting. Here the weak direction, defined by $\theta=\theta_1$, is an invariant of the reduced problem \eqref{eq:PainleveReducedDS} but corresponds to $y=-\infty$ by \eqsref{gg}{Sar}. 


By Fenichel's theory \cite{fen1,fen2,fen3}, compact subsets of $S_{a,r}$ perturb to invariant slow manifolds $S_{a,\varepsilon}$ and $S_{r,\varepsilon}$, respectively. These objects are non-unique but $\mathcal O(e^{-c/\epsilon})$-close.

\section{Main Result}\seclab{mainResults}
Since the rigid body equations \eqref{eq:dynrod} are unable to address what happens beyond $P$, we introduced compliance in Section \secref{sec:compliancebasedmethod}, leading to a regularized set of governing equations \eqref{eq.slowfast2}. We have already seen in \propref{critS} that the point $P$ becomes the line $ \widehat P$ of nonhyperbolic BT points under regularization. Now we focus on the fate of the strong singular canard $\gamma^{s}$, also described in \propref{critS}, under this regularization. For convenience, we summarise our main result here:
\begin{theorem}\thmlab{thm:main}
Suppose $\mu> \mu_C$ with $\mu_C$ as in \eqref{eq:muCgen} and consider a small neighborhood $\mathcal U\subset \{(\theta,\phi)\in\mathbb R^2\}$ of the point $P=(\theta_1,\phi_1^+)$ where $b=p_+=0$. Then for $0<\varepsilon\le \varepsilon_0$ sufficiently small there exists a \textnormal{canard} orbit $\gamma_\epsilon^s$ of \eqref{eq.slowfast2} connecting the attracting Fenichel slow manifold $S_{a,\varepsilon}$ with the stable manifold of the repelling Fenichel slow manifold $S_{r,\varepsilon}$. $\gamma_\varepsilon^s$ is $o(1)$-close to $\gamma^{s}$ within $\mathcal U$ and it divides $S_{a,\varepsilon}$ into orbits that lift off from those that eventually stick. 
\end{theorem}
\begin{remark}
 We can estimate the $o(1)$ in \thmref{thm:main} to be $\mathcal O(\varepsilon^{\eta(7-2\xi)/24})$, for any $\eta\in (0,1)$, using Gronwall's inequality. This is a corollary of \propref{M2Est} in \appref{appNF}. The estimate could probably be improved but we did not pursue this. 
\end{remark}
\begin{remark}\remlab{newremark}
 \ed{
 For the last statement of the theorem, we add the following. 
 Fix $C>0$ large and consider the following box in the $(y_2,w_2,\theta,\phi)$-space:
 \begin{align*}
  U = \{(y_2,w_2,\theta,\phi) \in \mathbb R^4&\vert y_2 \in [-C,0],\,w_2 \in [-\varpi^{-1},\varpi^{-1}],\\
  &(\theta,\phi)\in [\theta_1-\chi,\theta_1+\chi]\times [\phi_1^+-\chi,\phi_1^++\chi]\},
 \end{align*}
with $\varpi>0$ and $\chi>0$ both sufficiently small. Fenichel's manifold $S_{a,\varepsilon}$ is a graph over a compact subset $\mathcal C\subset \mathcal T$. Within $\mathcal C$ we can write $\gamma_\varepsilon^s$ as
 \begin{align*}
  \phi = m_{ss,\varepsilon}(\theta ) = m_{ss}(\theta)+o(1),
 \end{align*}
  recall \eqref{gammas}. Now, consider initial conditions on the intersection of $S_{a,\varepsilon}$ with the subset of the $\{\theta=\theta_1-\chi\}$-face of the box $U$ where $\phi$ is sufficiently close but greater than $m_{ss,\varepsilon}(\theta_1-\chi)$. 
Under the forward flow, points within this set will then either leave the box $U$ through its $\{\theta=\theta_1+\chi\}$-face, if $\phi$ is $\mathcal O(e^{-c/\varepsilon})$-close to $m_{ss,\varepsilon}(\theta_1-\chi)$, or leave the box $U$ through the $\{y_2=0\}$-face with $w_2>0$ such that lift-off occurs (like $\mathcal L$ in \propref{Wss}). Similarly, consider initial conditions on the intersection of $S_{a,\varepsilon}$ with the subset of the $\{\theta=\theta_1-\chi\}$-face of the box $U$ where $\phi$ is sufficiently close but less than $m_{ss,\varepsilon}(\theta_1-\chi)$. Under the foward flow, points within this set
will then either leave the box $U$ through its $\{\theta=\theta_1+\chi\}$-face, if $\phi$ is $\mathcal O(e^{-c/\varepsilon})$-close to $m_{ss,\varepsilon}(\theta_1-\chi)$, or leave the box $U$ through the $\{y_2=-C\}$-face with $w_2<0$ such that sticking and IWC occurs (like $\mathcal S$ in \propref{Wss}), as described in \cite{hogkri}. 

These results are corollaries of Theorem 1 and Fenichel's theory and they generalise \propref{Wss} to the compliant version. 
Note that orbits initially on the Fenichel slow manifold $S_{a,\varepsilon}$ do not twist upon passage near $\widehat P$. In particular, the projection of orbits on $S_{a,\varepsilon}$ near $\gamma^s_{\varepsilon}$ onto the $(y_2,w_2)$-plane do not oscillate. This is part of our main result. It is clearly different when we go backwards from $S_{r,\varepsilon}$ because $S_{a}$ is of focus type. Consider initial conditions on the intersection of $S_{r,\epsilon}$ with the subset of the $\{\theta=\theta_1+\chi\}$-face of $U$ where $\phi$ is sufficiently close to $\phi=m_{ss,\varepsilon}(\theta_1+\chi)$.
 These points will under the backward flow have projections onto the $(y_2,w_2)$-plane that oscillate around $\gamma^s_\varepsilon$ when reaching $\mathcal T$.

%
 
 }
\end{remark}

In the next section, to begin the proof of \thmref{thm:main}, in \eqref{painleveFinal} we present a rescaled version of \eqref{eq.slowfast2} in order to simplify the subsequent blowup 
of $\widehat P$ in Section \secref{sec:phatblowup}. Our approach naturally leads to three different changes of variable, known as {\it charts}, which are analysed in Sections \secref{sec:kappa3} to \secref{sec:kappa1}. The technical details are presented in a series of Appendices.
\section{Proof of \thmref{thm:main}}\seclab{blowup}
Starting with  \eqref{eq.slowfast2} we proceed by: (a) dropping the subscripts on $(y,w)$, (b) moving $P=(\theta_1,\phi_1^+)$ to the origin $(\theta,\phi)=0$, (c) straightening out the zero level set of $b$ to $\phi=0$, (d) eliminating time, and finally (e) applying appropriate scalings. Omitting the details, we obtain the following system
\begin{align}
 \varepsilon {y}' &= (1+ f(\theta, \phi)) w,\eqlab{painleveFinal}\\
 \varepsilon { w}'&=  \phi (1+ \tilde b(  \theta, \phi))- \theta (1+  \tilde p( \theta,   \phi))\left[- y- \delta w+\varepsilon N(y,w)\right],\nonumber\\
 {\theta}' &= 1,\nonumber\\
 \phi'&=\xi (1+ \tilde c( \theta,  \phi))\left[- y- \delta w+\varepsilon N(y,w)\right]+s(1+ h(\theta)).\nonumber
\end{align}
for $(\theta,\phi)\in \mathcal U$, a small neighborhood of $(0,0)$, where $\xi\in (0,1),s>0$ are defined in \eqsref{eq.xi}{eq.s}, where 
\begin{align*}
f(\theta,\phi) &=\mathcal O(\theta+\phi),\\
 h(\theta) &= \mathcal O(\theta),
\end{align*}
and
\begin{align*}
 \tilde b(\theta,\phi)&= \mathcal O(\theta+\phi),\\
 \tilde p(\theta,\phi)&=\mathcal O(\theta+\phi),\\
 \tilde c(\theta,\phi) &=\mathcal O(\theta+\phi),
\end{align*}
are all smooth functions. For simplicity, we suppress any dependency on $\varepsilon$, since this will play no role in the following. Also, since we will be working near $\gamma^s$ on $S_{a,r}$ the nonsmoothness of $F_N$ will play no role in the following. We will therefore replace $\left[\cdot \right]$ in \eqref{painleveFinal} by parentheses, recall \eqref{eq:compliance}. 

We will now prove the existence of a strong canard for \eqref{painleveFinal} for $0<\epsilon\ll 1$, which then proves \thmref{thm:main}.
We begin by redefining the sets $\mathcal F$ and $\mathcal T$ from \eqref{eq:FTregions} as 
\begin{align*}
\mathcal F &= \mathcal U\cap \{\theta>0,\,\phi>0\},\\
 \mathcal T &=\mathcal U\cap \{\theta<0,\,\phi<0\},
\end{align*}
so that they are now precisely the first and third quadrant, respectively, of the $(\theta,\phi)$-plane. We also redefine $g(\theta,\phi)$ from \eqref{gg} as
\begin{align}
 g(\theta,\phi)= -\frac{\phi(1+\tilde b(\theta,\phi))}{\theta(1+\tilde p(\theta,\phi))},\quad (\theta,\phi)\in \mathcal F\cup \mathcal T.\eqlab{eq.gFunc}
\end{align}
then the critical set $S$ of the layer problem \eqref{painleveFinal}$_{\epsilon=0}$ is a union of
\begin{align}
 S_a&:\quad w=0,\,y = g(\theta,\phi),\quad (\theta,\phi)\in \mathcal T,\eqlab{newSar}\\
 S_r&:\quad w=0,\,y = g(\theta,\phi),\quad (\theta,\phi)\in \mathcal F,\nonumber\\
 \widehat P&:\quad y\in \mathbb R,\,w=\theta = \phi = 0,\nonumber
\end{align}
identical to \eqref{Sar}. Following arguments identical to those used in \propref{critS}, $S_a$ is normally attracting (focus-type), $S_r$ is repelling (saddle type) and $\widehat P$ is a line of nonhyperbolic BT points, as before. Finally, there exists a strong singular canard $\gamma^{s}$ for the slow flow on $S_a\cup S_r$ that is tangent to
\begin{align}
(1-\xi,s)^{T},\eqlab{vStrong}
\end{align}
at $(\theta,\phi)=0$,  in the $(\theta,\phi)$-plane. Using \eqref{eq.gFunc} it follows that $\gamma^s$ on $S$ intersects $\widehat P$ in
\begin{align}
 y=-(1-\xi)^{-1} s,\,w=\theta=\phi = 0.\eqlab{vstrong2}
\end{align}


Compact subsets of $S_a$ and $S_r$ perturb by Fenichel's theory \cite{fen1,fen2,fen3} to attracting and repelling invariant manifolds $S_{a,\varepsilon}$ and $S_{r,\varepsilon}$, respectively, for $\epsilon$ sufficiently small. 

We now blowup the line $\widehat P$, defined in \eqref{newSar}, using the formalism of Krupa and Szmolyan \cite{krupa_extending_2001}.


\subsection{Blowup of $\widehat P$}\seclab{sec:phatblowup}
To study system \eqref{painleveFinal} near $\widehat P$ we consider the extended system $(\eqref{painleveFinal},\varepsilon'=0)$ written in terms of the fast time scale:
\begin{align}
 \dot {y} &= (1+ f(\theta, \phi)) w,\eqlab{painleveFinalNew}\\
 { w}'&=  \phi (1+ \tilde b(  \theta, \phi))- \theta (1+  \tilde p( \theta,   \phi))\left[- y- \delta w+\varepsilon N(y,w)\right],\nonumber\\
 \dot {\theta} &= \varepsilon,\nonumber\\
 \dot{\phi}&=\varepsilon \left(\xi (1+ \tilde c( \theta,  \phi))\left[- y- \delta w+\varepsilon N(y,w)\right]+s(1+ h(\theta))\right),\\
 \dot \varepsilon &=0.
\end{align}
In this way, $S_{a}$, $S_r$ and $\widehat P$ become subsets of $\{\epsilon=0\}$. Similarly, the Fenichel $2D$ slow manifolds $S_{a,\varepsilon}$ and $S_{r,\varepsilon}$ are now  $\{\varepsilon=\textnormal{const.}\}$-sections of $3D$ center manifolds $M_{a}$ and $M_r$ of \eqref{painleveFinalNew}. We will continue to denote these obvious embeddings by the same symbols. 

We will work in a small neighborhood of $\widehat P:\,y\in (-\infty,0],(w,\theta,\phi,\varepsilon) = 0$.
We then apply the following blowup of $(w,\theta,\phi,\varepsilon)=0$:
\begin{align*}
 \Phi:\quad (y,r,(\bar w,\bar \theta,\bar \phi,\bar \epsilon))\in \mathcal B\rightarrow (y,w,\theta,\phi,\varepsilon)\in \mathbb R^4\times \mathbb R_+, \quad
 &\mathcal B \equiv \mathbb R\times R_+\times S^3,
\end{align*}
defined by
\begin{align}
w =r\bar w,\quad \theta=r^2\bar \theta,\quad \phi=r^2 \bar \phi,\quad \varepsilon=r^3\bar \epsilon,\eqlab{blowupMap}
\end{align}
with 
\begin{align*}
 S^3:\quad \bar w^2+\bar \theta^2+\bar \phi^2+\bar \epsilon^2=1.
\end{align*}
The blowup map $\Phi$ does not change $y$ so we retain this symbol. Notice that $r=0$ in \eqref{blowupMap} corresponds to $\widehat P$ and  $\Phi$ therefore blows up $\widehat P$ to a cylinder of $3$-spheres $(y,(\bar w,\bar \theta,\bar \phi,\bar \epsilon))\in (-\infty,0]\times S^3$.

\ed{The mapping $\Phi$ gives rise to a vector-field $\overline X$ on $\mathcal B$ by pull-back of \eqref{painleveFinalNew}. Here $\overline{X}\vert_{r=0}=0$. The exponents (or weights) of $r$ in \eqref{blowupMap} are, however, chosen so that the desingularized vector-field
\begin{align*}
 \widehat X = r^{-1} \overline{X},
\end{align*}
is well-defined and non-trivial for $r=0$. It is $\widehat X$ that we shall study in the sequel. As usual, the orbits of $\widehat X$ agree with those of $\overline X$ for $r>0$ but the fact that $\widehat X$ is non-trival for $r=0$ allow us to use regular perturbation techniques to describe $\overline X$ for $r$ small.  

\subsection{Charts}
Clearly, we can describe a small neighborhood of $(w,\theta,\phi,\varepsilon)=0$ with $\varepsilon\ge 0$ by studying each value of $(r,(\bar w,\bar \theta,\bar \phi,\bar \epsilon))$ with $r\sim 0$ and with $(\bar w,\bar \theta,\bar \phi,\bar \epsilon)\in S^3\cap \{\bar \epsilon\ge 0\}$. Instead of working with spherical coordinates, it is more convenient to work with the directional charts
\begin{align}
 \kappa_1:\quad &\bar \theta = -1:\quad w=r_1w_1,\quad \theta = -r_1^2,\quad \phi =r_1^2\phi_1,\quad \varepsilon=r_1^3\epsilon_1,\quad (w_1,\phi_1,\epsilon_1) \in K_1,\eqlab{kappa1}\\
 \kappa_2:\quad &\bar \epsilon = 1:\quad \,\,\,\,\, w=r_2w_2,\quad \theta= r_2^2\theta_2,\quad \phi=r_2^2\phi_2,\quad \varepsilon=r_2^3,\quad \,\,\,\,\,(w_2,\theta_2,\phi_2) \in K_2, \eqlab{kappa2}\\
 \kappa_3:\quad &\bar \theta = 1:\quad \,\,\,\,\,w=r_3w_3,\quad \theta= r_3^2,\quad \,\,\,\,\,\phi=r_3^2\phi_3,\quad \varepsilon=r_3^3\epsilon_3,\quad (w_3,\phi_3,\epsilon_3)\in K_3,\eqlab{kappa3}
\end{align}
that correspond to setting $\bar \theta=-1$, $\bar \epsilon=1$ and $\bar \theta=1$, respectively, in \eqref{blowupMap}. The sets $K_1$, $K_2$ and $K_3$ are sufficiently large open sets in $\mathbb R^3$ so that the three charts cover our neighborhood of $(w,\theta,\phi,\varepsilon)=0$ with $\varepsilon\ge 0$. In the proof of \thmref{thm:main}, we will actually fix $K_1$ and $K_3$ to be such that the boxes
\begin{align}
U_1&= \{(w_1,\phi_1,\epsilon_1) \in \mathbb R^3 \vert w_1 \in [-\sigma,\sigma],\, \epsilon_1 \in [0,\nu]\, \phi_1 \in [-\varpi^{-1},-\varpi]\},\eqlab{U1set}\\
U_3&=\{(w_3,\phi_3,\epsilon_3) \in \mathbb R^3 \vert w_3 \in [-\sigma,\sigma],\, \epsilon_3 \in [0,\nu]\, \phi_3 \in [\varpi,\varpi^{-1}]\},\eqlab{U3set}
\end{align}
for $\sigma>0$, $\nu>0$ and $\varpi>0$ sufficiently small, are subsets of $K_1$ and $K_3$, respectively, and then adjust $K_2$ accordingly. In particular, we will take $K_2$ so large that the box
\begin{align}
U_2= \{(w_2,\theta_2,\phi_2) \in \mathbb R^3 \vert &w_2 \in [-\nu^{-1/3}\sigma,\nu^{-1/3}\sigma],\, \theta_2 \in [-\nu^{-2/3},\nu^{-2/3}],\\
&\phi_2 \in [-\nu^{-2/3}\varpi^{-1},\nu^{-2/3} \varpi^{-1}]\},\eqlab{U2set}
\end{align}
is a subset of $K_2$.

The chart $\kappa_1$ is called the \textit{entry chart}, $\kappa_2$ is called the \textit{scaling chart}, and finally $\kappa_3$ is called the \textit{exit chart}. Geometrically \eqref{kappa1} can be interpreted as a stereographic-like projection from the plane $\{(w_1,-1,\phi_1,\epsilon_1)\vert (w_1,\phi_1,\epsilon_1)\in K_1\}$, tangent to $S^3$ at $\bar \theta=-1$, to the hemisphere $S^3\cap \{\bar \theta<0\}$:
\begin{align*}
 w_1 = (-\bar \theta)^{-1/2} \bar w,\quad \phi_1 = (-\bar \theta)^{-1} \bar \phi,\quad \epsilon_1 = (-\bar \theta)^{-3/2} \bar \epsilon.
\end{align*}
Similar interpretations apply to \eqsref{kappa2}{kappa3}:
\begin{align*}
 w_2 &= \bar \epsilon^{-1/3} \bar w,\quad \theta_2 = \bar \epsilon^{-2/3} \bar \theta,\quad \phi_2 = \bar \epsilon^{-2/3} \bar \phi,\\
 w_3 &= \bar \theta^{-1/2} \bar w,\quad \phi_3 = \bar \theta^{-1} \bar \phi,\quad\,\,\, \epsilon_3 = \bar \theta^{-3/2} \bar \epsilon.
\end{align*}
for $\bar \epsilon>0$ and $\bar \theta>0$, respectively. 
We follow the convention that variables, manifolds, and other dynamical objects will be given a subscript $i$ in chart $\kappa_i$. Similarly, objects in the blowup variables $(y,r,(\bar w,\bar \theta,\bar \phi,\bar \epsilon))$ are given an overline. 

\subsection{Coordinate changes}
When the charts $\kappa_1$ and $\kappa_2$ or $\kappa_2$ and $\kappa_3$ overlap we can change coordinates. ($\kappa_1$ and $\kappa_3$ cannot overlap.) We will denote the smooth change of coordinates from $\kappa_i$ to $\kappa_j$ by $\kappa_{ji}$. Straightforward calculations show that
\begin{align}
 \kappa_{12}: \quad (r_2,w_2,\theta_2,\phi_2) &\mapsto (r_1,w_1,\phi_1,\epsilon_1),\eqlab{kappa12New}\\
 r_1 &= r_2 (-\theta_2)^{1/2},\nonumber\\
 w_1 &= (-\theta_2)^{-1/2} w_2,\nonumber\\
 \phi_1 &= (-\theta_2)^{-1} \phi_2,\nonumber\\
 \epsilon_1 &= (-\theta_2)^{-3/2},\nonumber
 \end{align}
 for $r_2\ge 0$ and all $(w_2,\theta_2,\phi_2) \in K_2$ so that $(w_1,\phi_1,\epsilon_1)\in K_1$. Furthermore,
 \begin{align*}
 \kappa_{23}:\quad (r_3,w_3,\phi_3,\epsilon_3) &\mapsto (r_2,w_2,\theta_2,\phi_2),\nonumber\\
  r_2 &= r_3 \epsilon_3^{1/3},\nonumber\\
  w_2 &= \epsilon_3^{-1/3} w_3,\nonumber\\
  \theta_2 &= \epsilon_3^{-2/3},\nonumber\\
  \phi_2 &= \epsilon_3^{-2/3} \phi_3,\nonumber
\end{align*}
for $r_3\ge 0$ and all $(w_3,\phi_3,\epsilon_3)\in K_3$ so that $(w_1,\phi_1,\epsilon_1)\in K_1$. The expressions for $\kappa_{21}$ and $\kappa_{32}$ follow easily from these results. Notice, in particular, that 
\begin{align*}
 \kappa_{23}(r_3,w_3,\phi_3,\nu) = (r_3 \nu^{1/3},\nu^{-1/3} w_3,\nu^{-2/3},\nu^{-2/3}\phi_3).
\end{align*}
and hence the $\{\epsilon_3=\nu\}$-face of the box $U_3$ \eqref{U3set} gets mapped by the diffeomorphism $\kappa_{23}$ to a subset of the $\{\theta_2=\nu^{-2/3}\}$-face of the box $U_2$ \eqref{U2set} . Similarly, the $\{\epsilon_1=\nu\}$-face of $U_1$ \eqref{U1set} gets mapped by the diffeomorphism $\kappa_{12}$ to a subset of the $\{\theta_2=-\nu^{-2/3}\}$-face of $U_2$. We collect this result in the following lemma.  
 \begin{lemma}\lemmalab{boxes}
  \begin{align*}
   \kappa_{23}\left(\{\epsilon_3=\nu\}\cap U_3\right)  &= \{\theta_2=\nu^{-2/3},\,\phi_2 \in [\nu^{-2/3}\varpi,\nu^{-2/3} \varpi^{-1}]\}\cap U_2,\\
\kappa_{21}\left(\{\epsilon_1=\nu\}\cap U_1\right) &=\{\theta_2=-\nu^{-2/3},\,\phi_2 \in [-\nu^{-2/3}\varpi^{-1},-\nu^{-2/3} \varpi]\}\cap U_2.
\end{align*}
 \end{lemma}


}
%

In chart $\kappa_1$ we will encounter a line of normally {\it elliptic} critical points. A true unfolding of $\widehat P$ as a line of co-dimension two BT-bifurcation points, similar to the approach in \cite{chiba2011a}, would enable some hyperbolicity in this chart (without the need for additional blowup). However, for our problem such an unfolding is unphysical. Instead we will apply a sequence of normal form transformations that accurately eliminates the fast oscillations, and then subsequently apply an additional blowup that captures the contraction in the entry chart, enabling an accurate continuation of the slow manifold into the scaling chart $\kappa_2$ and the $\{\theta_2=-\nu^{-2/3}\}$-face of $U_2$. 

Due to the technical difficulties in chart $\kappa_1$ in this paper we will  work our way backwards, starting from the exit chart $\kappa_3$ in Section \secref{sec:kappa3}, then move onto the scaling chart $\kappa_2$ in Section \secref{sec:kappa2} and then finally attack the difficulties in the entry chart $\kappa_1$ in Section \secref{sec:kappa1}. Lengthy proofs are consigned to a series of Appendices. Then, in Section \secref{combi} we combine these results to prove \thmref{thm:main}. 


\subsection{Exit chart $\kappa_3$}\seclab{sec:kappa3}
Substituting \eqref{kappa3} into \eqref{painleveFinal} gives 
\begin{align*}
\dot y&=(1+r_3^2f_3(r_3,\phi_3))w_3,\\
\dot w_3 &=\phi_3(1+r_3^2b_3(r_3,\phi_3))-(1+r_3^2p_3(r_3,\phi_3)) \left(-y-\delta r_3 w_3+\varepsilon N(y,r_3w_3,\varepsilon)\right)-\frac12 \epsilon_3 w_3,\\
 \dot r_3 &=\frac12 r_3\epsilon_3,\\
 \dot \epsilon_3 &= -\frac32 \epsilon_3^2,\\
 \dot \phi_3 &=\epsilon_3\left(\xi(1+r_3^2 g_3(r_3,\phi_3))\left(-y-\delta r_3 w_3+\varepsilon N(y,r_3w_3,\varepsilon)\right)+s(1+r_3^2 h_3(r_3,\phi_3))-\phi_3\right),
\end{align*}
after division of the right hand side by $r_3$. We keep the use of $\varepsilon \, \, (= r_3^3\epsilon_3)$ 
for brevity. 

The subspaces $\{r_3=0\}$ and $\{\epsilon_3=0\}$ are invariant. Along their intersection $\{r_3=\epsilon_3=0\}$ we find
\begin{align*}
 L_3:\quad y = \phi_3,\,w_3=r_3=\epsilon_3=0,
\end{align*}
as a line of critical points. Linearizing about a point in $L_3$ gives
the following generalized eigensolutions $(\lambda_i,w_i)$,
\begin{align*}
 \lambda_1 &= 1,\quad w_1 =(1,1,0,0,0)^T, 
\quad 
                           \lambda_2 = -1,\quad w_2 = (1,-1,0,0,0)^T,
                           \end{align*}
                           and
                           \begin{align*}                       
                           \lambda_{3,4,5} &=0,\quad w_{3,4,5} = (0,0,1,0,0)^T,(-1,0,1,0,0)^T,(0,0,0,0,1)^T.
\end{align*}
Hence we have gained hyperbolicity of $S_r$ at the blowup of $P$. Now, consider the set
\begin{align*}
V_3  = \{(y,r_3,(w_3,\phi_3,\epsilon_3))\in (-\infty,0]\times [0,\nu]\times U_3\},
\end{align*}
with $U_3$ as in \eqref{U3set}. In particular, we shall henceforth fix $\varpi>0$ small enough so that $\phi_3 = (1-\xi)^{-1}s\in [\varpi,\varpi^{-1}]$. 
Then we have the following Proposition.
\begin{proposition}\proplab{hallo}
 For $\nu$ sufficiently small, there exists a smooth saddle-type center manifold within $V_3$:
 \begin{align*}
  M_{r,3}:\quad &y = -\phi_3+\epsilon_3^2(s-(1-\xi)\phi_3)(1+\epsilon_3 \psi_1^{(y)}(\phi_3,\epsilon_3))+r_3\psi_2^{(y)}(r_3,\phi_3,\epsilon_3),\\
  &w_3 = \epsilon_3(s-(1-\xi)\phi_3)(1+\epsilon_3\psi_1^{(w)}(\phi_3,\epsilon_3)) +r_3\psi_2^{(w)}(r_3,\phi_3,\epsilon_3),
 \end{align*}
 where 
 \begin{align*}
\psi_2^{(w)}(r_3,\phi_3,\epsilon_3),\, \psi_2^{(y)}(r_3,\phi_3,\epsilon_3) = \mathcal O(r_3+\epsilon_3).
 \end{align*}
Also locally, $M_{r,3}$ has smooth foliations by stable and unstable fibers:
\begin{align*}
 W^s(M_{r,3})&:\quad y = m_s(w_3,r_3,\phi_3,\epsilon_3) = -\phi_3-w_3 + \mathcal O(r_3+\epsilon_3),\\
 W^u(M_{r,3})&:\quad y = m_u(w_3,r_3,\phi_3,\epsilon_3) = -\phi_3+w_3 + \mathcal O(r_3+\epsilon_3),
\end{align*}
with $m_{s,u}$ both smooth. 
The manifold $M_{r,3}$ contains $S_{r,3}\equiv \kappa_3(S_r)$ within $\epsilon_3=0$ as a set of critical points and 
\begin{align*}
 C_{r,3}:\quad &y = -\phi_3+\epsilon_3^2(s-(1-\xi)\phi_3)(1+\epsilon_3 \psi_1^{(y)}(\phi_3,\epsilon_3)),\\
  &w_3 = \epsilon_3(s-(1-\xi)\phi_3)(1+\epsilon_3\psi_1^{(w)}(\phi_3,\epsilon_3)),
\end{align*}
within $r_3=0$, as a center saddle-type sub-manifold. The sub-manifold $C_{r,3}$ contains the invariant line:
\begin{align}
l_3:\quad y=-\frac{s}{1-\xi},\,w_3=0,\,\phi_3=\frac{s}{1-\xi},\,\epsilon_3\ge 0,\,r_3=0. \eqlab{eq:l3} 
\end{align}
\end{proposition}
\begin{proof}
 Follows from center manifold theory and simple calculations.
\end{proof}
The manifold $M_{r,3}$ is foliated by invariant hyperbolas $r_3^3 \epsilon_3=\varepsilon\ge 0$. We let
\begin{align*}
 M_{r,3}(\varepsilon) \equiv M_{r,3}\cap \{\varepsilon=r_3^3\epsilon_3\},
\end{align*}
with $0<\varepsilon \ll 1$ fixed.
It is an extension of the Fenichel slow manifold $S_{r,\varepsilon}$ up to $\{\epsilon_3=\nu\}$-face of $V_3$. Here it is a smooth graph over $\phi_1$ and $r_1=(\varepsilon\nu^{-1})^{1/3}$:
\begin{align}
 M_{r,3}(\varepsilon)\cap \{\epsilon_3=\nu\}:\quad y &= -\phi_3+\nu^2(s-(1-\xi)\phi_3)(1+\nu \psi_1^{(y)}(\phi_3,\nu))\eqlab{SrepsInto2}\\
 &+(\varepsilon\nu^{-1})^{1/3}\psi_2^{(y)}((\varepsilon\nu^{-1})^{1/3},\phi_3,\nu),\nonumber\\
  w_3 &= \nu(s-(1-\xi)\phi_3)(1+\nu\psi_1^{(w)}(\phi_3,\nu)) \nonumber\\
  &+(\varepsilon\nu^{-1})^{1/3}\psi_2^{(w)}((\varepsilon\nu^{-1})^{1/3},\phi_3,\nu),\nonumber
\end{align}
where it is $\mathcal O(\varepsilon^{1/3})$-close to $C_{r,3}\cap \{\epsilon_3=\nu\}$. 
The reduced problem on $M_{r,3}$ is
\begin{align}
 \dot r_3 &= \frac12 r_3,\eqlab{reducedMr3}\\
  \dot \epsilon_3 &=-\frac32 \epsilon_3,\nonumber\\
 \dot \phi_3 &= (s-(1-\xi)\phi_3)(1+\mathcal O(\epsilon_3^2))+r_3\mathcal O(\epsilon_3+r_3).\nonumber
\end{align}
after division by $\epsilon_3$. This division desingularizes the dynamics within $S_{r,3}\subset \{\epsilon_3=0\}$. The point
\begin{align*}
 p_3:\quad r_3=\epsilon_3=0,\,\phi_3=(1-\xi)^{-1} s>0,
\end{align*}
is a hyperbolic equilibrium of the reduced problem \eqref{reducedMr3}. It is the intersection of $\bar \gamma^s$ with the blowup cylinder\footnote{$\bar \gamma^s$ is simply $\gamma^s$ in terms of the blowup variables $(y,(\bar w,\bar \theta,\bar \phi,\bar \epsilon))\in \mathbb R\times S^3$.}. See \eqsref{vStrong}{vstrong2}. The linearization of \eqref{reducedMr3} about $p_3$ gives eigenvalues $\frac12,\,-(1-\xi),\,-\frac32$, respectively. The invariant line $l_3:\,r_3=0,\,\phi_3=(1-\xi)^{-1} s,\,\epsilon_3\ge 0$ is therefore the strong stable manifold within $\{r_3=0\}$ of $p_3$ for the reduced problem \eqref{reducedMr3}, coinciding with the strong eigenvector associated with the strong eigenvalue $-\frac32<-(1-\xi)$, since $\xi \in (0,1)$ from \eqref{eq.xi}. The unique unstable manifold contained within the $(r_3,\phi_3)$-plane corresponds to the singular strong canard $\kappa_3(\gamma^s)\subset \{\epsilon_3=0\}$.
We will continue $l_3$ backwards into chart $\kappa_2$ in the following section.


\subsection{Scaling chart $\kappa_2$}\seclab{sec:kappa2}
Substituting \eqref{kappa2} into \eqref{painleveFinal} gives 
\begin{align}
 \dot y &=w_2,\eqlab{chart2EqnNew}\\
 \dot w_2 &=\phi_2(1+r_2^2 b_2(\theta_2,\phi_2,r_2))-\theta_2(1+r_2^2p_2(\theta_2,\phi_2,r_2))\left(-y-\delta r_2 w_2+\varepsilon N(y,r_2w_2,\varepsilon)\right),\nonumber\\
 \dot \theta_2 &=1,\nonumber\\
 \dot \phi_2 &=\xi(1+ r_2^2g_2( \theta_2,  \phi_2,r_2))\left(- y- \delta r_2 w_2+\varepsilon N(y,r_2w_2,\varepsilon)\right)+s(1+ r_2^2 h_2(\theta_2,r_2)),\nonumber
\end{align}
after division of the right hand side by $r_2$. Also $\dot r_2=0$ since $r_2=\varepsilon^{1/3}$. In this chart, we consider the set
\begin{align*}
V_2  = \{(y,r_2,(w_2,\theta_2,\phi_2))\in (-\infty,0]\times [0,\nu]\times U_2\},
\end{align*}
with $U_2$ as in \eqref{U2set}.
Setting $r_2=0$ gives the following linear system 
\begin{align}
 y'(\theta_2) &=w_2,\eqlab{chartK2r2Eq0}\\
 w_2'(\theta_2) &=\phi_2+\theta_2y,\nonumber\\
 \phi_2'(\theta_2) &=-\xi y+s,\nonumber
\end{align}
after the elimination of time. 
\begin{lemma}
 Recall $\xi\in (0,1)$. The line 
 \begin{align}
  l_2:\quad \phi_2 = \frac{s}{1-\xi }\theta_2,\quad y=-\frac{s}{1-\xi},\quad w_2=0,\,\theta_2\in \mathbb R,\,r_2=0,\eqlab{linel2}
 \end{align}
is an invariant of \eqref{chart2EqnNew}$_{r_2=0}$. It coincides with $\kappa_{23}(l_3)$, where $l_3$ is the invariant line in chart $\kappa_3$, given in \eqref{eq:l3}.
\end{lemma}
\begin{proof}
 Straightforward calculation.
\end{proof}
\begin{remark}
 Note that the projection of the line $l_2$ onto the $(\theta_2,\phi_2)$ plane coincides with the span of the eigenvector in \eqref{vStrong}. In terms of the blowup variables $(y,(\bar w,\bar \theta,\bar \phi,\bar \epsilon))\in \mathbb R\times S^3$ it becomes the great circle
 \begin{align*}
  \overline  l:\quad y = -\frac{s}{1-\xi},\quad r=0,\quad \bar w=0,\quad \bar \theta^{-1} \bar \phi = \frac{s}{1-\xi},\quad (\bar \theta,\bar \phi,\bar \epsilon)\in S^2.
 \end{align*}
\end{remark}

We now show that \eqref{chartK2r2Eq0} can be rewritten as Langer's \cite{Langer1995a,Langer1995b} equation \eqref{3rdChart2intro}. Let 
\begin{align}
\phi_2=\frac{s}{1-\xi }\theta_2+\tilde \phi_2, \quad y=-\frac{s}{1-\xi}+\tilde y,\eqlab{original}
\end{align}
 which centers $l_2$ along $\tilde y=w_2=\tilde \phi_2=0$. Then \eqref{chartK2r2Eq0} becomes
\begin{align}
 \tilde y'(\theta_2) &=w_2,\eqlab{x2Eqn}\\
 w_2'(\theta_2) &= \tilde \phi_2+\theta_2\tilde y,\nonumber\\
 \tilde \phi_2'(\theta_2) &= -\xi \tilde y.\nonumber
\end{align}
We now eliminate $w_2$ and $\tilde \phi_2$ as follows:
\begin{align}
 w_2(\theta_2) = \tilde y'(\theta_2),\quad \tilde \phi_2(\theta_2) = w_2'(\theta_2)-\theta_2\tilde y(\theta_2)=\tilde y''(\theta_2)-\theta_2\tilde y(\theta_2),\eqlab{y2v2Sol}
\end{align}
to finally obtain Langer's equation
\begin{align}
 \tilde y'''(\theta_2) =\theta_2\tilde y'(\theta_2)+(1-\xi) \tilde y(\theta_2).\eqlab{3rdChart2}
\end{align}

%
The general solution of \eqref{3rdChart2} can be expressed in terms of hyper-geometric functions. But we do not find this presentation useful. Instead we investigate those asymptotic properties of the solutions of \eqref{3rdChart2} that are important for our analysis. 
\begin{lemma}\lemmalab{x21x22x23}
 The solution space of \eqref{3rdChart2}  is spanned by the following linearly independent solutions
 \begin{align}
  \textnormal{La}_{\xi}(\theta_2) &= \int_0^\infty \exp\left({-\tau^3/3+\theta_2\tau}\right) \tau^{-\xi} d\tau,\eqlab{Langy21}\\
  \textnormal{Lb}_{\xi}(\theta_2)&= \int_0^\infty\cos(\tau^3/3+\theta_2\tau+{\xi \pi}/{2})\tau^{-\xi} d\tau,\eqlab{Langy22}\\
  \textnormal{Lc}_{\xi}(\theta_2)&=\int_0^\infty \cos(\tau^3/3+\theta_2\tau-{\xi \pi}/{2})\tau^{-\xi} d\tau.\eqlab{Langy23}
 \end{align}
 Here $\text{span}\, \{\textnormal{La}_{\xi}\}$ contains all \textnormal{non-oscillatory solutions} of \eqref{3rdChart2} for $\theta_2\rightarrow -\infty$. The solution $\textnormal{La}_{\xi}$ can be written in the following form for $\theta_2<0$:
  \begin{align*}
   \textnormal{La}_{\xi} = (-\theta_2)^{-(1-\xi)}E_2((-\theta_2)^{-3}),
  \end{align*}
with $E_2$ real analytic and satisfying:
\begin{align*}
 E_2(0) = {\Gamma(1-\xi)},
\end{align*}
where  $\Gamma(z) = \int_0^\infty t^{z-1} e^{-t} dt$ is the $\Gamma$-function:. 
The following asymptotics hold
 \begin{align}
\textnormal{La}_{\xi}(\theta_2) &= \sqrt{\pi} \theta_2^{(1-\xi)/2-3/4} \exp\left({2\theta_2^{3/2}/3}\right)(1+o(1)),\eqlab{x21ASuPos}\\
\textnormal{Lb}_{\xi}(\theta_2) &= \frac{\sqrt{\pi}}{2} \theta_2^{(1-\xi)/2-3/4} \exp\left({-2\theta_2^{3/2}/3}\right)(1+o(1)),\eqlab{x22ASuPos}\\
\textnormal{Lc}_{\xi}(\theta_2) &= \sin(\xi\pi) \theta_2^{(1-\xi)/2-3/2}\Gamma(1-\xi)(1+o(1)),\eqlab{x23ASuPos}
 \end{align}
 for $\theta_2\rightarrow \infty$,
and
\begin{align}
\textnormal{La}_{\xi}(\theta_2) &= (-\theta_2)^{-(1-\xi)}\Gamma(1-\xi)(1+\mathcal O(\theta_2^{-3})),\eqlab{x21ASuNeg}\\
\textnormal{Lb}_{\xi}(\theta_2) &= \sqrt{2\pi}(-\theta_2)^{(1-\xi)/2-3/4} \sin\left(\frac23 (-\theta_2)^{3/2}(1+o(1)) +\frac{\pi}{4}-\frac{\xi \pi}{2}\right)(1+o(1))\nonumber\\
 &+\sin(\xi \pi)\textnormal{La}_{\xi}(\theta_2),\eqlab{x22ASuNeg}\\
\textnormal{Lc}_{\xi}(\theta_2) &= \sqrt{2\pi}(-\theta_2)^{(1-\xi)/2-3/4} \sin\left(\frac23 (-\theta_2)^{3/2}(1+o(1)) +\frac{\pi}{4}+\frac{\xi \pi}{2}\right)(1+o(1)),\eqlab{x23ASuNeg}
 \end{align}
 for $\theta_2\rightarrow -\infty$.
\end{lemma}
\begin{proof}
We apply the Laplace transform $\tilde y(\theta)= \int_\Upsilon \hat{\tilde y}(z) e^{\theta z} dz$ and solve for $\hat{\tilde y}(z)$, $z\in D\subset \mathbb C$, and the unbounded contour $\Upsilon\subset D$. This representation simplifies the asymptotics for $\theta\rightarrow \pm \infty$. Full details are given in \appref{appx21x22x23}. 
 \end{proof}
\begin{remark}\remlab{LaLbLc}
Note that $\textnormal{Lb}_{0} = \textnormal{Lc}_{0}= \pi\textnormal{Ai}$ where $\textnormal{Ai}$ is the standard Airy function. Also, from \lemmaref{x21x22x23}:
 \begin{itemize}
  \item $\textnormal{La}_{\xi}$ has (a) algebraic decay and is non-oscillatory for $\theta_2\rightarrow -\infty$ and (b) exponential growth for $\theta_2\rightarrow \infty$.
  \item $\textnormal{Lb}_{\xi}$ has (a) oscillatory behaviour for $\theta_2\rightarrow -\infty$ and (b) exponential decay for $\theta_2\rightarrow \infty$.
  \item $\textnormal{Lc}_{\xi}$ has (a) oscillatory behaviour for $\theta_2\rightarrow -\infty$ and (b) algebraic decay for $\theta_2\rightarrow \infty$.
 \end{itemize}
 \ed{The oscillatory behaviour for $\textnormal{Lb}_{\xi}$ and $\textnormal{Lc}_{\xi}$ for $\theta_2\rightarrow -\infty$ decays in amplitude, since\\ $(-\theta_2)^{(1-\xi)/2-3/4} \rightarrow 0$ for $\xi\in (0,1)$.}
\end{remark}

Given $\tilde y(\theta_2)$, the values of $(w_2,\tilde \phi_2)$, can be determined from \eqref{y2v2Sol}. In particular, $\tilde y=\textnormal{La}_{\xi}(\theta_2)$ and $\tilde y=\textnormal{Lc}_{\xi}(\theta_2)$
give
$(w_2,\tilde \phi_2)=(\textnormal{La}_{\xi}'(\theta_2),\textnormal{La}_{\xi}''(\theta_2)-\theta_2\textnormal{La}_{\xi}(\theta_2))$ and $(w_2,\tilde \phi_2)=(\textnormal{Lc}_{\xi}'(\theta_2),\textnormal{Lc}_{\xi}''(\theta_2)-\theta_2\textnormal{Lc}_{\xi}(\theta_2))$, respectively. 
We therefore introduce the following 1D solution spaces of \eqref{x2Eqn}:
\begin{align}
 C_{e,2}&\equiv \text{span}\,\{(\tilde y,w_2,\tilde \phi_2) = (\textnormal{La}_{\xi}(\theta_2),\textnormal{La}_{\xi}'(\theta_2),\textnormal{La}_{\xi}''(\theta_2)-\theta_2\textnormal{La}_{\xi}(\theta_2))\},\eqlab{Ce2}\\
 C_{r,2}&\equiv \text{span}\,\{(\tilde y,w_2,\tilde \phi_2) =(\textnormal{Lc}_{\xi}(\theta_2),\textnormal{Lc}_{\xi}'(\theta_2),\textnormal{Lc}_{\xi}''(\theta_2)-\theta_2\textnormal{Lc}_{\xi}(\theta_2))\}.\nonumber
\end{align}
We can take $C_{r,2}=\kappa_{23}(C_{r,3})$ for $\theta_2\gg 1$ due to the algebraic decay \eqref{x23ASuPos} of $\textnormal{Lc}_{\xi}(\theta_2)$ for $\theta_2\rightarrow \infty$.  Recall that $C_{r,3}$ is nonunique as a saddle-type center manifold. 
We then have the following:
\begin{proposition}\proplab{Ce2Cr2}
 The space $C_{e,2}$ is transverse to $W^s(C_{r,2})$ along $l_2:\,(\tilde y,w_2,\tilde \phi_2)(\theta_2)\equiv 0$.
\end{proposition}
\begin{proof}
By the exponential growth of $\textnormal{La}_{\xi}$ (see \eqref{x21ASuPos}) for $\theta_2\rightarrow \infty$ it follows that the tangent space of $C_{e,2}$ is not a subspace of the tangent space of $W^s(C_{r,2})$ along $l_2$ for $\theta_2\gg 1$. Hence the intersection is transverse.
\end{proof}

Returning to the variables $(y,w_2,\theta_2,\phi_2)$, then $C_{e,2}$ and $C_{r,2}$ become $2D$ invariant manifolds of \eqref{chart2EqnNew}$_{r_2=0}$:
\begin{align}
 C_{e,2} = \bigg\{(y,w,\theta_2,\phi_2)\vert &\theta_2 \in \mathbb R,\,(y,w_2,\phi_2)=\left(-\frac{s}{1-\xi}+\tilde y,w_2,\frac{s}{1-\xi }\theta_2+\tilde \phi_2\right),\nonumber\\
 &(\tilde y,w_2,\tilde \phi_2)\in \text{span}\,\{(\textnormal{La}_{\xi}(\theta_2),\textnormal{La}_{\xi}'(\theta_2),\textnormal{La}_{\xi}''(\theta_2)-\theta_2\textnormal{La}_{\xi}(\theta_2))\}\bigg\},\eqlab{Ce2New}\\
 C_{r,2} = \bigg\{(y,w,\theta_2,\phi_2)\vert &\theta_2 \in \mathbb R,\,(y,w_2,\phi_2)=\left(-\frac{s}{1-\xi}+\tilde y,w_2,\frac{s}{1-\xi }\theta_2+\tilde \phi_2\right),\nonumber\\
 &(\tilde y,w_2,\tilde \phi_2)\in \text{span}\,\{(\textnormal{Lc}_{\xi}(\theta_2),\textnormal{Lc}_{\xi}'(\theta_2),\textnormal{La}_{\xi}''(\theta_2)-\theta_2\textnormal{Lc}_{\xi}(\theta_2))^T\}\bigg\},\nonumber
\end{align}
using the same symbols for the new objects in the new variables.   

\begin{lemma}\lemmalab{Ce2}
 Consider $\theta_2\le -\nu^{-2/3}$. Then for $\nu$ sufficiently small, the invariant manifold $C_{e,2}$ of \eqref{chart2EqnNew}$_{r_2=0}$ can be written as a graph over $(\theta_2,\phi_2)$,
\begin{align*}
 y &= (-\theta_2)^{-1}\phi_2+(-\theta_2)^{-3} \left((-\theta_2)^{-1}\phi_2+\frac{s}{1-\xi}\right) F_2((-\theta_2)^{-3}),\\
  w_2&=(-\theta_2)^{-1} (1-\xi)  \left((-\theta_2)^{-1}\phi_2+\frac{s}{1-\xi}\right) \left(1+(-\theta_2)^{-3} G_2((-\theta_2)^{-3})\right),
\end{align*}
where $F_2$ and $G_2$ are real analytic functions.
\end{lemma}
\begin{proof}
 By \eqref{y2v2Sol} we find that the linear space $C_{e,2}$ in \eqref{Ce2} is spanned by 
  \begin{align*}
  \tilde y &= (-\theta_2)^{-1+\xi} E_2((-\theta_2)^{-3}),\\
  w_2&=(-\theta_2)^{-2+\xi} \left((1-\xi)E_2((-\theta_2)^{-3})+3(-\theta_2)^{-3}E_2'((-\theta_2)^{-3})\right)\\
  &\equiv (-\theta_2)^{-2+\xi} (1-\xi)\tilde G_2((-\theta_2)^{-3}),\\
  \tilde \phi_2 &=(-\theta_2)^{\xi} \bigg(E_2((-\theta_2)^{-3})+(-\theta_2)^{-3}\bigg((2-\xi)\left((1-\xi)E_2((-\theta_2)^{-3})+3(-\theta_2)^{-3}E_2'((-\theta_2)^{-3})\right)\\
  &+(-\theta_2)^{-3}\left((12-\xi)E_2'((-\theta_2)^{-3})+9(-\theta_2)^{-3}E_2''((-\theta_2)^{-3})\right)\bigg)\bigg)\\
  &\equiv (-\theta_2)^{\xi} \tilde H_2((-\theta_2)^{-3}),
 \end{align*}
 where the functions $\tilde G_2$ and $\tilde H_2$ defined by these equations are real analytic. 
 Notice that $\tilde H_2(0)=E_2(0)=\Gamma(1-\xi)$. Therefore for $\theta_2\ll 0$ we find from the last equation
 \begin{align*}
  (-\theta_2)^{\xi} = \frac{\tilde \phi_2}{\tilde H_2((-\theta_2)^{-3})}.
 \end{align*}
Then we substitute this expression for  $(-\theta_2)^{\xi}$ into the first two equations to give 
\begin{align*}
 \tilde y &= (-\theta_2)^{-1} \tilde \phi_2 \tilde H_2((-\theta_2)^{-3})^{-1} E_2((-\theta_2)^{-3}),\\
 w_2 &= (-\theta_2)^{-2} \tilde \phi_2 (1-\xi)\tilde H_2((-\theta_2)^{-3})^{-1} \tilde G_2((-\theta_2)^{-3}).
\end{align*}
Then the desired results follow upon returning to the original variables (using \eqref{original}) and setting 
\begin{align*}
F_2(u)&\equiv \int_0^1 \frac{d}{dv}\left(\tilde H_2(v)^{-1} E_2(v)\right)\bigg|_{v=su}ds,\\
G_2(u)&\equiv \int_0^1 \frac{d}{dv}\left(\tilde H_2(v)^{-1}\tilde G_2(v)\right)\bigg|_{v=su}ds. 
\end{align*}
\end{proof}

We will continue $l_2$ backwards into chart $\kappa_1$ in the following section.
%



\subsection{Entry chart $\kappa_1$}\seclab{sec:kappa1}

Substituting \eqref{kappa1} into \eqref{painleveFinal} gives
\begin{align}
 \dot y&=(1+r_1^2f_1(r_1,\phi_1))w_1,\eqlab{eq:entrychart}\\
 \dot w_1&=\phi_1(1+r_1^2 b_1(r_1,\phi_1))+(1+r_1^2 q_1(r_1,\phi_1))\left(-y-\delta r_1w_1+\varepsilon N(y,r_1w_1,\varepsilon)\right),\nonumber\\
 \dot r_1 &=-\frac12 r_1\epsilon_1,\nonumber\\
 \dot \epsilon_1&=\frac32 \epsilon_1^2,\nonumber\\
 \dot \phi_1 &=\epsilon_1\left(\xi(1+r_1^2 g_1(r_1,\phi_1))\left(-y-\delta r_1w_1+\varepsilon N(y,r_1w_1,\varepsilon)\right)+s(1+r_1^2h_1(r_1,\phi_1))+\phi_1\right),\nonumber
\end{align}
after division of the right hand side by $r_1$, where $\dot{()}=\frac{d}{dt_1}$ and $\varepsilon=r_1^3\epsilon_1$. 
Then $\{r_1=0\}$ and $\{\epsilon_1=0\}$ are invariant subspaces. Within $\{r_1=\epsilon_1=0\}$ we obtain
\begin{align}
 \dot y&=w_1,\eqlab{eps1r1Eq0}\\
 \dot w_1&=\phi_1-y,\nonumber\\
 \dot \phi_1 &=0.\nonumber
\end{align}
Therefore the space $\{r_1=\epsilon_1=0\}$ is foliated by invariant cylinders:
\begin{align*}
 (y-\phi_1)^2+w_1^2= R_0^2,\,\phi_1<0,\,R_0\ge 0.
\end{align*}
Each slice $\phi_1=\text{const}.<0$ of such a cylinder is a periodic orbit:
\begin{align*}
y(t_1)=\phi_1+R_0\cos(t_1),\quad w_1(t_1)=-R_0\sin(t_1),\quad \phi_1(t_1)=\text{const}.,
\end{align*}
for \eqref{eps1r1Eq0}.
For $R_0=0$ we have
\begin{align} \eqlab{eq:lineL1}
 L_1:\,y=\phi_1,\,w_1=0,\,\epsilon_1=0,\,r_1=0,\,\phi_1<0,
\end{align}
as a line of equilibria.


The line $L_1$ is normally elliptic rather than hyperbolic. This is not a surprise. In fact, we can deduce this directly from the expression \eqref{eigenvalues} for the eigenvalues $\lambda_{\pm}$ of the layer problem \eqref{eq:layerproblem}, as follows. From \eqref{eigenvalues}, we have
 \begin{align*}
  \lambda_\pm = \frac12 \theta \delta \pm \sqrt{\theta} = \frac12 \theta \delta \pm i\sqrt{-\theta},
 \end{align*}
 in terms of our new variables, ignoring for simplicity the higher order terms in $\theta \sim 0$. Setting $\theta=-r_1^2$, from \eqref{kappa1}, gives
\begin{align*}
 \lambda_\pm = -\frac12 r_1^2\delta \pm i r_1.
\end{align*}
The desingularization amplifies this eigenvalue to $\mathcal O(1)$ through the division of $r_1$ such that 
\begin{align}
 \lambda_{1,\pm} =-\frac12 r_1\delta \pm i,\eqlab{lambda1pm}
\end{align}
which for $r_1=0$ collapse to the eigenvalues $\lambda_{1,\pm}=\pm i$ we obtain by linearizing \eqref{eps1r1Eq0} about $L_1$.  

Within $\{\epsilon_1=0\}$ we re-discover 
\begin{align}
S_{a,1}=\kappa_1(S_a):\,y=\frac{1+r_1^2b_1(r_1,\phi_1)}{1+r_1^2q_1(r_1,\phi_1)}\phi_1,\quad w_1=0,\quad \epsilon_1=0,\eqlab{Sa1yw1}
\end{align}
as a manifold of equilibria. The linearization of a point in $S_{a,1}$ gives \eqref{lambda1pm} (to first order in $r_1$) as nontrivial eigenvalues. It is attracting for $r_1>0$ but for $r_1=0$ (where $S_{a,1}\subset \{\epsilon_1=0\}$ collapses to $L_1)$ it is only normally elliptic. 

Similarly, by the analysis in chart $\kappa_2$, we have, within $\{r_1=0\}$, an invariant manifold $C_{e,1}=\kappa_{1,2}(C_{e,2})$ (recall  \eqref{Ce2New}). 
Let
\begin{align*}
V_1  = \{(y,r_1,(w_1,\phi_1,\epsilon_1))\in (-\infty,0]\times [0,\nu]\times U_1\},
\end{align*}
with $U_1$ as in \eqref{U1set}. 


\begin{lemma}\lemmalab{Ce1yw1}
For $\nu>0$ sufficiently small, $C_{e,1}$ becomes:
\begin{align*}
 y &= \phi_1+\epsilon_1^2 \left(\phi_1+\frac{s}{1-\xi}\right) F_2(\epsilon_1^2),\\
  w_1&=\epsilon_1 (1-\xi)  \left(\phi_1+\frac{s}{1-\xi}\right)  \left(1+\epsilon_1^2G_2(\epsilon_1^2)\right),\\
  r_1 &=0,
\end{align*}
within $V_1$ where $\epsilon_1 \in [0,\nu]$, $\phi_1 \in [-\varpi^{-1},-\varpi]$.

\end{lemma}
\begin{proof}
 This follows from \lemmaref{Ce2} and the coordinate changes $\kappa_{21}=\kappa_{12}^{-1}$ and $\kappa_{12}$, see \eqref{kappa12New}. 
 
\end{proof}

Consider $r_1=0$ and the reduced system on $C_{e,1}$ within $V_1$. This gives
\begin{align}
 \dot \epsilon_1&=\frac32 \epsilon_1,\eqlab{reducedCe1}\\
 \dot \phi_1 &=\left((1-\xi) \phi_1 +s\right)(1+\mathcal O(\epsilon_1)),\nonumber
\end{align}
after division of the right hand side by $\epsilon_1$. Here $l_1:\, \epsilon_1\ge 0,\,\phi_1 = -(1-\xi)^{-1}s$ becomes a strong unstable manifold of the unstable node $(\epsilon_1,\phi_1) = (0,-(1-\xi)^{-1}s)$ within $C_{e,1}\subset \{r_1=0\}$. 

\begin{lemma}\lemmalab{tildeyw0}
Let $(\tilde y_0,\tilde w_0)$ be defined as
\begin{align}
 y &= \frac{1+r_1^2 b_1(r_1,\phi_1)}{1+r_1^2 q_1(r_1,\phi_1)}\phi_1+\epsilon_1^2 \left(\phi_1+\frac{s}{1-\xi}\right) F_2(\epsilon_1^2)+\tilde y_0,\eqlab{tildeywTransformation}\\
  w_1&=\epsilon_1 (1-\xi)  \left(\phi_1+\frac{s}{1-\xi}\right)  \left(1+\epsilon_1^2G_2(\epsilon_1^2)\right)+\tilde w_0,\nonumber
\end{align}
for $(y,r_1,w_1,\phi_1,\epsilon_1)\in V_1$. 
Then $S_{a,1}\subset \{\epsilon_1=0\}$ and $C_{e,1}\subset\{r_1=0\}$ become
\begin{align*}
 S_{a,1}:\quad \tilde y_0=0,\quad \tilde w_0=0,\quad \epsilon_1=0,\\
 C_{e,1}:\quad  \tilde y_0=0,\quad \tilde w_0=0,\quad r_1=0,
\end{align*}
within $\tilde V_1$: the image of $V_1$ under the diffeomorphism $(y,r_1,w_1,\phi_1,\epsilon_1)\mapsto (\tilde y_0,r_1,\tilde w_0,\phi_1,\epsilon_1)$ defined by \eqref{tildeywTransformation}. 
\end{lemma}
\begin{proof}
Follows from \eqref{Sa1yw1} and \lemmaref{Ce1yw1}.
\end{proof}
%
The main result of this section is then:
\begin{proposition}\proplab{SaEpsEst}
  Let $\nu>0$ be sufficiently small. Then for $0<\varepsilon\ll 1$ the forward flow of $S_{a,\varepsilon}$ intersects the $\{\epsilon_1=\nu\}$-face of $\tilde V_1$: 
  in a $C^1$-graph over $\phi_1\in [-\varpi^{-1},-\varpi]$:
   \begin{align*}
  (\tilde y_0,\tilde w_0)  = m_{\varepsilon}(\phi_1),\quad \phi_1\in [-\varpi^{-1},-\varpi],
 \end{align*}
  with $r_1=(\varepsilon \nu^{-1})^{1/3}$ and
 \begin{align*}
  m_{\varepsilon}(\phi_1)=o(1),\quad \partial_{\phi_1} m_{\varepsilon}(\phi_1) = o(1).
 \end{align*}
\end{proposition}
\begin{proof}
Full details of the proof are given in \appref{appNF}. We present an outline here. We work with the coordinates $(\tilde y_0,\tilde w_0)$ and amplify the dissipation in \eqref{lambda1pm} by a further (polar) blowup transformation (of $r_1=\epsilon_1=0$) and further desingularization. 
However, we cannot apply blowup and desingularization directly due to the fast oscillatory part (recall e.g. \eqref{lambda1pm}$_{r_1=0}$). 
Therefore we first apply normal form transformations (like higher order averaging) in Section \secref{sec.nf} to factor out this oscillatory part. The result is described in \propref{NF}. Then in Section \secref{sec.vanderPol} we apply a van der Pol transformation (like moving into a rotating coordinate frame), given by \eqref{vanderPol}. This gives rise to system \eqref{zFinal} for which the transformed normally elliptic line $L_1$ \eqref{eq:lineL1} can be studied using a second blowup and subsequent desingularization; see \eqref{blowupFinal} and Section \secref{sec.blowupFinal}. Within the chart \eqref{barr11}, the desingularization corresponds to division by $r_1$ of the real part of the eigenvalues of \eqref{lambda1pm} to $-\frac12\delta$ (to leading order). 
Hereby we gain hyperbolicity for $r_1=0$ which allow us (with some technical difficulties due to the oscillatory remainder of the normal form) to extend the slow manifold $S_{a,\varepsilon}$ as a perturbation of $S_a$ up until
\begin{align}
\theta=-(\varepsilon \nu^{-1})^{1/2},
\end{align}
for $\nu$ sufficiently small; see \propref{Ma1k1}, proved in \appref{appMa1k1}. We extend this further up until 
\begin{align}
\theta = -(\varepsilon \nu^{-1})^{2/3},\eqlab{thetaEps1_3}
\end{align} 
where $\theta_2 = -\nu^{-2/3}$ and therefore $\epsilon_1=\nu$ cf. \eqref{kappa12New}, by applying the forward flow near a hyperbolic saddle in a subsequent chart \eqref{bareps11} in Section \secref{sec.bareps11}. The result then shows that $S_{a,\varepsilon}$ is $o(1)$-close to the invariant manifold $C_{e,1}=\kappa_{12}(C_{e,2})$ \eqref{Ce2} of non-oscillatory solutions at the section defined by \eqref{thetaEps1_3}; see \propref{M2Est}, which working backwards then implies \propref{SaEpsEst}. 
\end{proof}


\subsection{Combining the results to prove \thmref{thm:main}}\seclab{combi}
\ed{The existence of a maximal canard $\gamma^s_\varepsilon$, connecting the Fenichel slow manifold $S_{a,\varepsilon}$ with the stable manifold of $M_{r,3}(\varepsilon)$, the extension of $S_{r,\varepsilon}$ into chart $\kappa_3$, follows from \propref{SaEpsEst} and \propref{Ce2Cr2}. Indeed, \propref{SaEpsEst} implies, by \lemmaref{tildeyw0}, the $o(1)$-closeness of the forward flow of $S_{a,\varepsilon}$ to $C_{e,2}$ along the $\{\theta_2=-\nu^{-2/3}\}$-face of the box $U_2$ (recall \lemmaref{boxes}). To finish the proof, we can therefore work in $U_2$ in chart $\kappa_2$ only and follow $S_{a,\varepsilon}$ from $\theta_2=-\nu^{-2/3}$ up to $\theta_2=\nu^{-2/3}$ using $C_{e,2}$ as a guide. By regular perturbation theory, $S_{a,\varepsilon}$ is $o(1)$-close to $C_{e,2}$ along the $\{\theta_2=\nu^{-2/3}\}$-face of $U_2$. Here we also know from \propref{hallo}, in particular \eqref{SrepsInto2}, that $M_{r,2}(\varepsilon) = \kappa_{23}(M_{r,3}(\varepsilon))$, is $\mathcal O(\varepsilon^{1/3})$-close to $C_{r,2}$. Now, combining this with \propref{Ce2Cr2}, which states that $C_{e,2}$ intersects $W^s(C_{r,2})$ transversally along $l_2$, we finally conclude that the forward flow of $S_{a,\varepsilon}$ intersects $W^s(M_{r,2}(\varepsilon))$ transversally at $\theta_2=\nu^{-2/3}$ for all $0<\varepsilon\ll 1$. The intersection of these objects defines $\gamma_\varepsilon^s$ and it follows that it is $o(1)$-close to $l_2$ in chart $\kappa_2$. Therefore also $\gamma^s_\varepsilon\rightarrow \gamma^s$ as $\varepsilon\rightarrow 0$. }

\section{Discussion and Conclusions}\seclab{conclusions}

We have considered the problem of a slender rod slipping along a rough surface, as shown in \figref{fig:rod}. In a series of classical papers, Painlev\'e \cite{Painleve1895, Painleve1905a,Painleve1905b} showed that the governing rigid body equations for this problem can exhibit multiple solutions (the {\it indeterminate} case) or no solutions at all (the {\it inconsistent} case), provided the coefficient of friction $\mu$ exceeds a certain critical value $\mu_P$, given by \eqref{eq:mucrit}. Subsequently G\'enot and Brogliato \cite{GenotBrogliato1999} proved that, from a consistent state, the rod cannot reach an inconsistent state through slipping. Instead the rod will either stop slipping and stick or it will lift-off from the surface. Between these two cases is a special solution for $\mu>\mu_C>\mu_P$, where $\mu_C$ a new critical value of the coefficient of friction, given by \eqref{eq:muCgen}. Physically, the special solution corresponds to the rod slipping until it reaches a singular ``$0/0$'' point $P$, shown in \figref{fig:GB}. Even though the rigid body equations can not describe what happens to the rod beyond the singular point $P$, it is possible to extend the special solution into the region of indeterminacy. Hence this extended solution is very reminiscent of a {\it canard} \cite{Benoit81}. To overcome the inadequacy of the rigid body equations beyond $P$, the rigid body assumption can be relaxed in the neighbourhood of the point of contact of the rod with the rough surface. Physically this corresponds to assuming a small compliance there. So it is natural to ask what happens to both the point $P$ and the special solution under this regularization.

In this paper, we have rigorously proved the existence of a strong canard in the {\it regularization by compliance} of the classical Painlev\'e problem. The canard is called strong because it is tangent to a strong eigendirection that appears in the rigid body formulation of Painlev\'e's problem. Our analysis is based on the blowup method, in the formalism developed and popularised by Krupa and Szmolyan \cite{krupa_extending_2001, krupa_extending2_2001, krupa_relaxation_2001}. Initially blowup gains us ellipticity  only (rather than hyperbolicity) in the entry chart $\kappa_1$, as shown in Section \secref{sec:kappa1}. As a consequence we cannot extend Fenichel's slow manifold into the scaling chart, where $\theta =\mathcal O(\varepsilon^{2/3})$, as a perturbation of the critical one, as it is done in $(2+1)$-slow-fast systems, for example. Instead we apply a sequence of normal form transformations, followed by an additional blowup that captures the contraction in the entry chart, enabling an accurate continuation of the slow manifold up until $\theta =\mathcal O(\varepsilon^{1/2})$. Recall proof of \propref{SaEpsEst}. From there we extend the slow manifold up until $\theta =\mathcal O(\varepsilon^{2/3})$ in the scaling chart $\kappa_2$ by careful estimation of the forward flow near a hyperbolic saddle. Key to the dynamics in the scaling chart $\kappa_2$ is Langer's equation \eqref{3rdChart2} and its asymptotic properties. In addition to our results on the regularized problem, we show in $\eqref{muCmuP}$ the surprising result that $\mu_C = \frac{2}{\sqrt{3}}\mu_P$ for a very large class of rigid body.\

This work was stimulated by a seminar given to the Applied Nonlinear Mathematics group in Bristol by Alan Champneys in December 2015 and attended by SJH, who immediately saw the potential for the use of blowup in this field. The main work in this paper was carried out during the Spring and Autumn of 2016. Subsequently the current authors were made aware of the paper by Nordmark {\it et al.} \cite{NordVarChamp17}. That paper addresses a wider class of rigid body problems than we do here. Canards are also studied and Langer's equation also appears. These authors use formal asymptotic methods and numerical computations, rather than our GSPT and blowup approach. 

One important difference between our two approaches lies in the number of different cases that are covered. We consider the class of rigid body problem where, for $\mu>\mu_C$, the weak direction is $\theta=\theta_1$ and the strong direction lies between the first and third quadrants of \figref{fig:GB}. When $\mu_P<\mu<\mu_C$, the strong direction is $\theta=\theta_1$ and the weak direction lies between the second and fourth quadrants (but does not correspond to a weak canard). Thus our case corresponds to Case II of Figure 3 in Nordmark {\it et al.} \cite{NordVarChamp17}. 

So the question naturally arises as to whether we could extend our approach to prove the existence of weak canards in more general settings (Case III of \cite{NordVarChamp17}). Weak canards in $(2+1)$-slow-fast systems are obtained as the intersection of an extension of Fenichel's slow manifold as a perturbation of the critical one into the scaling chart. The weak canards do not necessarily intersect the original Fenichel slow manifold. But then, as we are unable to extend the slow manifold into the scaling chart as a perturbation, it is therefore at this stage questionable, given the contraction towards the weak singular canard, whether one can really obtain a sensible notion of these canards for the compliant version when $0<\varepsilon\ll 1$. It seems that the result may depend upon the contraction rate of the slow-flow towards the weak singular canard. 

\bibliography{refs}
\bibliographystyle{plain}
\appendix
\section{Proof of \lemmaref{x21x22x23}: Properties of the solutions of Langer's equation}\applab{appx21x22x23}
In Section \secref{sec:kappa2}, we considered \textit{Langer}'s \cite{Langer1995a,Langer1995b} third order linear ODE:
\begin{align}
 y'''(\theta) = \theta y'(\theta)+(1-\xi)y(\theta),\eqlab{3rdChart2App}
\end{align}
with $\xi\in (0,1)$ where, in this Appendix, we 
drop both the subscripts and tildes in comparison with \eqref{3rdChart2}. 
The computations and analysis we perform in this section follow similar arguments used for studying the solutions of the Airy equation \begin{align}
 A''(\theta) = \theta A(\theta).\eqlab{AiryEqn}
\end{align} 
In fact, \eqref{3rdChart2App} is related to the Airy equation, see \cite{Vallee2004}. For $\xi=0$ we obtain \eqref{3rdChart2App} from \eqref{AiryEqn}$_{A=y}$ by differentiating with respect to $\theta$. For $\xi=1$, we obtain \eqref{AiryEqn} by setting $A=y'$. For the case when $\xi$ is a relative integer $1/n,\,n\in \mathbb N$, the solution involves algebraic combinations of Airy functions $\textnormal{Ai}$ and $\textnormal{Bi}$, their integrals and their derivatives (see \cite{Vallee1999, Vallee2004}). In particular, for $\xi=\frac12$, it is a straigthforward calculation to show that $y=A(2^{-2/3}\theta)^2$ solves \eqref{3rdChart2App} when $A$ solves \eqref{AiryEqn}. The solution for $\xi=\frac12$ is therefore a linear combination of $\textnormal{Ai}(2^{-2/3} \theta)^2$, $\textnormal{Bi}(2^{-2/3} \theta)^2$ and $\textnormal{Ai}(2^{-2/3} \theta)\textnormal{Bi}(2^{-2/3} \theta)$.

To proceed for general $\xi\in (0,1)$, we consider the solution ansatz
\begin{align*}
 y(\theta)  = \int_\Upsilon \hat y(z) e^{\theta z} dz, 
\end{align*}
following Laplace, 
where both the complex analytic function $\hat y(z)$, $z=u+iv\in D\subset \mathbb C$, and the unbounded contour $\Upsilon\subset D$ are to be determined. Suppose that $\hat y\vert_{\partial \Upsilon}=0$ and that the integral and its first three derivatives with respect to $u$ converge absolutely. Then insertion into \eqref{3rdChart2App} gives
\begin{align*}
 \int_\Upsilon \left(z^3 \hat y(z) +z\hat y'(z)+\xi \hat y(z)\right)e^{\theta z}dz = 0,
\end{align*}
upon using integration by parts. Therefore we set
\begin{align}
 \hat y(z) = e^{-z^3/3-\xi \log z},\eqlab{hatX}
\end{align}
with $z\ne 0$ as a solution of
\begin{align*}
 z^3 \hat y(z) +z\hat y'(z)+\xi \hat y(z)=0.
\end{align*}
This gives rise to the following solution of \eqref{3rdChart2App}
\begin{align}
 y(\theta)  = \int_\Upsilon e^{-z^3/3+\theta z} e^{-\xi \log z} dz, \eqlab{xuSolutionApp}
\end{align}
for appropriately chosen contours $\Upsilon$. 
Given that $\hat y\vert_{\partial \Upsilon}=0$ we restrict attention to those $z$ that asymptotically satisfy $\text{Re}\,(z^3)<0$, or equivalently
\begin{align}
 \text{Arg}_{\pi}(z) \in (-\pi/6,\pi/6)\cup (\pi/2,5\pi/6)\cup (-5\pi/6,-\pi/2),\eqlab{ArgSet}
\end{align}
for $\vert z\vert\gg 0$. Here $\text{Arg}_{\pi}\in (-\pi,\pi)$ is the principal value argument of $z$. We will later also need the separate argument
\begin{align}
 \text{Arg}_{0}(z)\in (0,2\pi),\eqlab{Arg0}
\end{align}
of $z\in \mathbb C$. The ``ends'' of the contour $\Upsilon$ should asymptotically be confined to the set in \eqref{ArgSet}. Furthermore, $0\notin \Upsilon$. 
But note that, since $\xi\in (0,1)$, the function $z^{-\xi}$ is integrable over $z\in (0,a)$ with $a>0$. 

To obtain the three different linearly independent solutions $\textnormal{La}_{\xi}$, $\textnormal{Lb}_{\xi}$ and $\textnormal{Lc}_{\xi}$ in \lemmaref{x21x22x23} we consider three different paths $\Upsilon_1$, $\Upsilon_2$ and $\Upsilon_3$ together with two different branch cuts for the complex logarithm appearing in  \eqref{hatX}. 

\appref{appx21x22x23} is organised as follows. The three solutions $\textnormal{La}_{\xi}$, $\textnormal{Lb}_{\xi}$ and $\textnormal{Lc}_{\xi}$ are considered in Sections \secref{app:A}, \secref{app:B} and \secref{app:C}, respectively. Then their asymptotics for $\theta\rightarrow \infty$ are considered in Sections \secref{app:Apinf}, \secref{app:Bpinf} and \secref{app:Cpinf}, respectively, and for $\theta\rightarrow -\infty$, in  Sections \secref{app:Aminf}, \secref{app:Bminf} and \secref{app:Cminf}, respectively

\subsection{Solution $\textnormal{La}_{\xi}$}\seclab{app:A}
 We will obtain the solution $\textnormal{La}_{\xi}$ by considering the integral \eqref{xuSolutionApp} over the contour $\Upsilon_{1,\nu}$, shown in \figref{contourGamma12}(a) and defined as:
 \begin{align*}
  \Upsilon_{1,\nu}&=\Upsilon_{1,\nu}^-\cup \Upsilon_{1,\nu}^+ \cup \Upsilon_{1,\nu}^0,
  \end{align*}
  where
  \begin{align*}
  \Upsilon_{1,\nu}^- &=\{z\in \mathbb C\vert \text{Im}(z) = - \nu,\,\text{Re}(z)\ge 0\},\\
    \Upsilon_{1,\nu}^0 &= \{z\in \mathbb C\vert  \vert z\vert =\nu,\,\text{Re}(z)< 0\},\\
    \Upsilon_{1,\nu}^+ &= \{z\in \mathbb C\vert \text{Im}(z) = + \nu,\,\text{Re}(z)\ge 0\},
 \end{align*}
for $\nu>0$. The path of integration is clockwise. We take a branch cut along $\text{arg}(z)=0$ and define the complex logarithm  in  \eqref{hatX} as
 \begin{align}
  \log_0(z)\equiv \ln \vert z\vert+i\text{Arg}_0(z),\eqlab{log0}
 \end{align}
where $\text{Arg}_0$ is the argument in \eqref{Arg0}. To ensure that $\textnormal{La}_{\xi}$ is real we multiply \eqref{xuSolutionApp} by $(1-e^{-\xi 2\pi i})^{-1}$ and therefore set
\begin{align*}
 \textnormal{La}_{\xi}(\theta) = (1-e^{-\xi 2\pi i})^{-1}\int_{\Upsilon_{1,\nu}} e^{-z^3/3+\theta z} e^{-\xi \log_0(z)} dz
\end{align*}
Since $\xi\in (0,1)$ and the integrand is analytic away from $\text{arg}(z)=0$, we easily conclude that
\begin{align}
 \textnormal{La}_{\xi}(\theta) &=(1-e^{-\xi 2\pi i})^{-1} \left(\int_{\infty}^0 e^{-\tau^3/3+\theta \tau} e^{-\xi (\ln \tau+i2\pi)}d\tau + \int_0^\infty e^{-\tau^3/3+\theta \tau} e^{-\xi \ln \tau}d\tau\right)\nonumber\\
 &=\int_0^\infty e^{-\tau^3/3+\theta \tau} \tau^{-\xi}d\tau,\eqlab{xu1SolApp}
\end{align}
upon sending $\nu\rightarrow 0^+$, in agreement with \eqref{Langy21}.

\begin{figure}
\begin{center}
 \subfigure[]{\includegraphics[width=.495\textwidth]{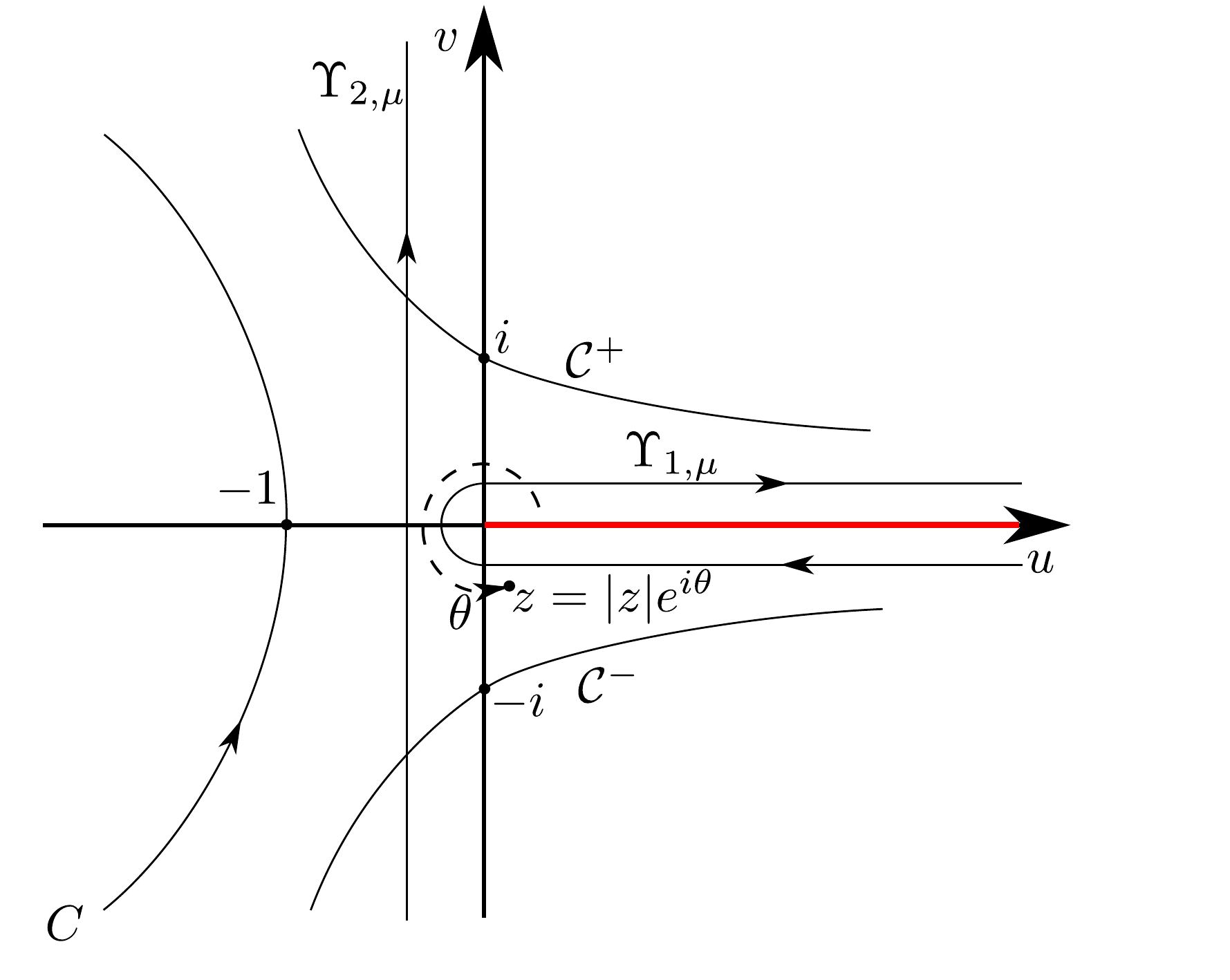}}
  \subfigure[]{\includegraphics[width=.495\textwidth]{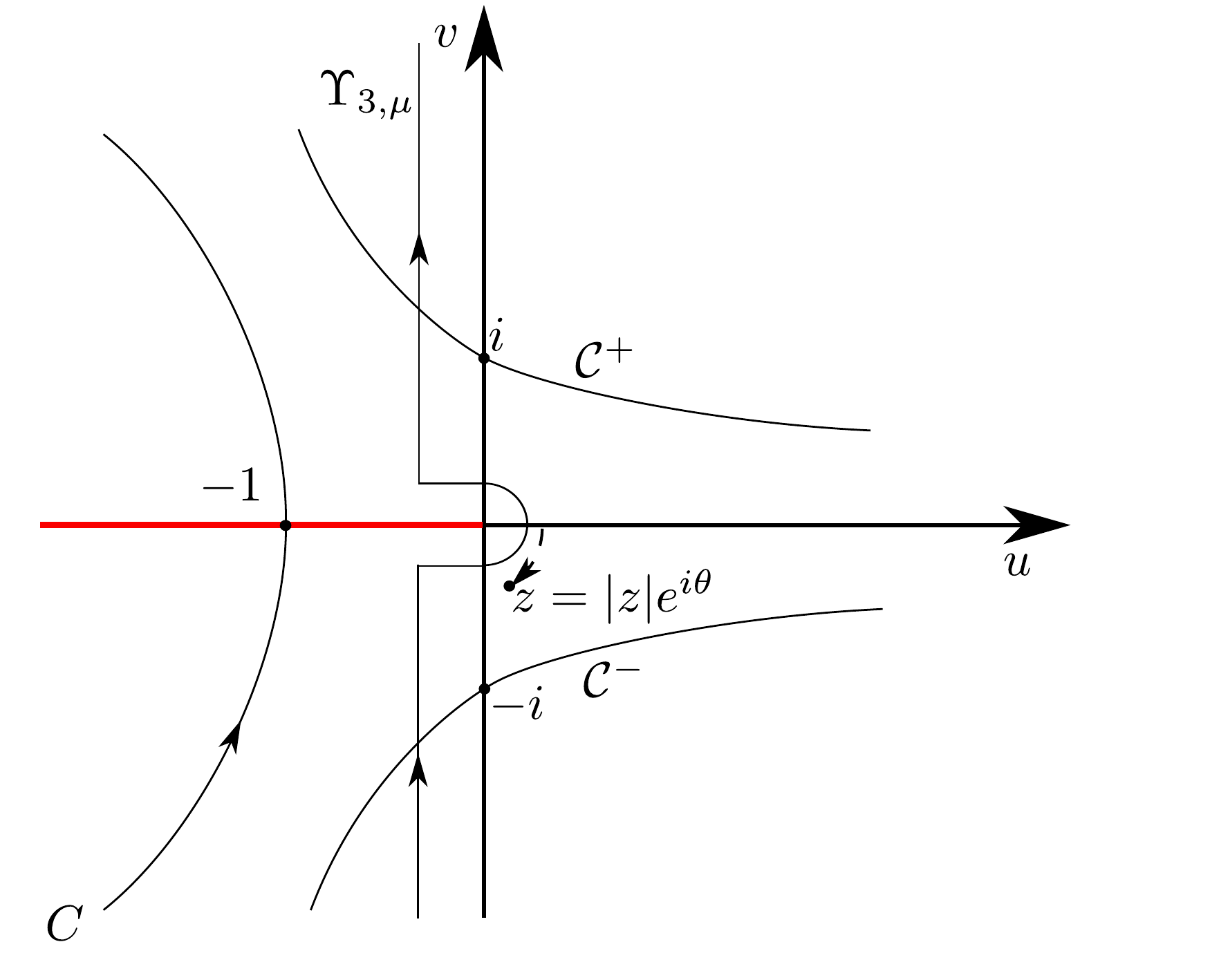}}
 \end{center}
 \caption{(a) Illustration of the contours $\Upsilon_{1,\nu}$ and $\Upsilon_{2,\nu}$. To determine the asymptotics of $\textnormal{Lb}_{\xi}$ we will use the contours $C$ and $\mathcal C^{\pm}$. (b) Illustration of the contour $\Upsilon_{3,\nu}$. To determine the asymptotics of $\textnormal{Lc}_{\xi}$ we will use the contours $C$ and $\mathcal C^{\pm}$.}
 \figlab{contourGamma12}
\end{figure}
\subsection{Solution $\textnormal{Lb}_{\xi}$}\seclab{app:B}
The solution $\textnormal{Lb}_{\xi}$ is obtained by considering the integral \eqref{xuSolutionApp} over a contour $\Upsilon_{2,\nu}$, shown in \figref{contourGamma12}(a) and defined as:
 \begin{align*}
  \Upsilon_{2,\nu}&=\{z\in \mathbb C\vert \text{Re}(z) = - \nu\}.
  \end{align*}
The direction of integration is along positive $\text{Im}(z)$.  Furthermore, we take a branch cut along $\text{arg}(z)=0$ and define the complex logarithm  in  \eqref{hatX} as in \eqref{log0}.
 To ensure that $\textnormal{Lb}_{\xi}$ is real, we multiply \eqref{xuSolutionApp} by $\frac{1}{2i}e^{i\xi \pi}$ and therefore set
\begin{align}
 \textnormal{Lb}_{\xi}(\theta) = \frac{1}{2i}e^{i\xi \pi}\int_{\Upsilon_{2,\nu}} e^{-z^3/3+\theta z} e^{-\xi \log_0(z)} dz.\eqlab{xu2solApp}
\end{align}
Again, since $\xi\in (0,1)$ and the integrand is analytic away from $\text{arg}(z)=\pi$, we easily conclude that
\begin{align*}
 \textnormal{Lb}_{\xi}(\theta) &=\int_0^\infty \cos(\tau^3/3+\theta \tau+\xi \pi/2)\tau^{-\xi}d\tau,
\end{align*}
upon sending $\nu\rightarrow 0^+$, in agreement with \eqref{Langy22}. 
\begin{remark}
 Note that the expressions for $\textnormal{Lb}_{\xi}$ with $\xi=0$ and $\xi=1$ collapse to 
 \begin{align*}
  \textnormal{Lb}_{\xi}(\theta) &=\int_0^\infty \cos(\tau^3/3+\theta \tau)d\tau=\pi \textnormal{Ai}\,(\theta).
 \end{align*}
 and
  \begin{align*}
  \textnormal{Lb}_{\xi}(\theta) &=\int_0^\infty \cos(\tau^3/3+\theta \tau+\pi/2)\tau^{-1} d\tau=\int_0^\infty -\sin(\tau^3/3+\theta \tau)\tau^{-1} d\tau = -\pi \int_0^\theta \textnormal{Ai}\,(u)du-\frac{\pi}{6},
 \end{align*}
 respectively. For $\xi=1$ we have used that $\int_0^\infty \sin(\tau^3/3)\tau^{-1} d\tau = \frac{\pi}{6}$.
\end{remark}

\subsection{Solution $\textnormal{Lc}_{\xi}$}\seclab{app:C}
For the solution $\textnormal{Lc}_{\xi}$ we select the contour $\Upsilon_{3,\nu}$, shown in \figref{contourGamma12}(b) and defined as:
 \begin{align*}
  \Upsilon_{3,\nu}&=\Upsilon_{3,\nu}^-\cup \Upsilon_{3,\nu}^+ \cup \Upsilon_{3,\nu}^0,
  \end{align*}
  where
  \begin{align*}
  \Upsilon_{3,\nu}^- &=\{z\in \mathbb C\vert \text{Re}(z) = - \nu,\,\text{Im}(z)\ge \nu\},\\
    \Upsilon_{3,\nu}^0 &= \{z\in \mathbb C \vert \text{Im}(z) = \pm \nu,\,\text{Re}(z)\in (-\nu,0)\}\cup \{z\in \mathbb C\vert  \vert z\vert =\nu,\,\text{Re}(z)\ge 0\},\\
    \Upsilon_{3,\nu}^+ &= \{z\in \mathbb C\vert \text{Re}(z) = - \nu,\,\text{Im}(z)\le -\nu\},
 \end{align*}
for $\nu>0$. The integration is anti-clockwise. Furthermore, we take a branch cut along $\text{arg}(z)=\pi$ and define the complex logarithm  in  \eqref{hatX} as
 \begin{align*}
  \log_\pi(z)\equiv \ln \vert z\vert+i\text{Arg}_\pi(z).
 \end{align*}
 To ensure that $\textnormal{Lc}_{\xi}$ is real we multiply \eqref{xuSolutionApp} by $\frac{1}{2i}$ and therefore set
\begin{align*}
 \textnormal{Lc}_{\xi}(\theta) = \frac{1}{2i}\int_{\Upsilon_{3,\nu}} e^{-z^3/3+\theta z} e^{-\xi \log_\pi(z)} dz,
\end{align*}
which as above gives
\begin{align*}
 \textnormal{Lc}_{\xi}(\theta) &=\int_0^\infty \cos(\tau^3/3+\theta \tau-\xi \pi/2)\tau^{-\xi}d\tau,
\end{align*}
upon sending $\nu\rightarrow 0^+$, in agreement with \eqref{Langy23}. 

Now we describe the asymptotics of each of the functions $\textnormal{La}_{\xi}$, $\textnormal{Lb}_{\xi}$ and $\textnormal{Lc}_{\xi}$ as $\theta\rightarrow \pm \infty$.

\subsection{Asymptotics of $\textnormal{La}_{\xi}$ for $\theta\rightarrow \infty$}\seclab{app:Apinf}
To study the behaviour of the function $\textnormal{La}_{\xi}$ in the limit $\theta\rightarrow \infty$, we consider \eqref{xu1SolApp} and set $\tau = \sqrt{\theta}s$ so that 
\begin{align*}
 \textnormal{La}_{\xi}(\theta) &= \theta^{(1-\xi)/2} \int_0^\infty e^{\theta^{3/2}(-s^3/3+s)} s^{-\xi} ds\\
 &=\theta^{(1-\xi)/2} e^{2\theta^{3/2}/3} \int_{-1}^\infty e^{\theta^{3/2}(-s^2(1+s/3))} (1+s)^{-\xi} ds
\end{align*}
and then using the fact that $\frac{1}{\sqrt{\pi} a}e^{-1/a^2 s^2}$ is a $\delta$-sequence as $a \equiv \theta^{-3/4}\rightarrow 0$ we obtain
\begin{align*}
 \textnormal{La}_{\xi}(\theta) = \sqrt{\pi} \theta^{(1-\xi)/2-3/4} e^{2\theta^{3/2}/3}(1+o(1)),
\end{align*}
as $\theta\rightarrow\infty$,  in agreement with \eqref{x21ASuPos}.

\subsection{Asymptotics of $\textnormal{Lb}_{\xi}$ for $\theta\rightarrow \infty$}\seclab{app:Bpinf}
To compute the asymptotics of $\textnormal{Lb}_{\xi}$ in \eqref{xu2solApp} for $\theta\rightarrow \infty$, we replace $z$ by $\sqrt{\theta}z$ so that
\begin{align}
 \textnormal{Lb}_{\xi}(\theta) = \frac{1}{2i}e^{i\xi \pi} \theta^{(1-\xi)/2} \int_{\Upsilon_{2,\nu}} e^{\theta^{3/2}F(z)}e^{-\xi \log_0 z}dz,\eqlab{xu2AppNew}
\end{align}
where 
\begin{align*}
 F(z) = -z^3/3+z.
\end{align*}
In \eqref{xu2AppNew} we have used the fact that the integrand is analytic away from $\text{arg}\,(z)=0$ to replace the path of integration $\Upsilon_{2,\nu}/\sqrt{\theta}$ by $\Upsilon_{2,\nu}$. Then we follow standard arguments used for computing the asymptotics of the Airy function $\text{Ai}$, deforming $\Upsilon_{2,\nu}$ into
\begin{align}
 C=\{z=u+iv\in \mathbb C\vert u=q(v^2)\},\eqlab{Econtour}
\end{align}
where
\begin{align*}
 q(v^2) \equiv-\sqrt{v^2/3+1}.
\end{align*}
Along $C$ we have $\text{Im}\,F(z) = 0$ and 
\begin{align*}
 \text{Re}\,F(z) = -\frac23 \sqrt{v^2/3+1}\left(1+\frac{4}{3}v^2\right)=-\frac23 - v^2(1+\mathcal O(v^2)).
\end{align*}
Using again that $\frac{1}{\sqrt{\pi} a}e^{-1/a^2 s^2}$ is a $\delta$-sequence as $a \equiv \theta^{-3/4} \rightarrow 0$ we obtain
\begin{align*}
 \textnormal{Lb}_{\xi}(\theta) = \frac{1}{2i}e^{i\xi \pi} \theta^{(1-\xi)/2} \int_{C} e^{\theta^{3/2}F(z)}e^{-\xi \log_0 z}dz=\frac{\sqrt{\pi}}{2} \theta^{(1-\xi)/2-3/4} e^{-2\theta^{3/2}/3}(1+o(1)),
\end{align*}
for $\theta\rightarrow \infty$,  in agreement with \eqref{x22ASuPos}.

\subsection{Asymptotics of $\textnormal{Lc}_{\xi}$ for $\theta\rightarrow \infty$}\seclab{app:Cpinf}
To compute the asymptotics of $\textnormal{Lc}_{\xi}$ for $\theta\rightarrow \infty$, we proceed as for $\textnormal{Lb}_{\xi}$. We replace $z$ by $\sqrt{\theta}z$ and then use the analyticity of the integrand away from $\text{arg}(z)=\pi$ to deform $\Upsilon_{3,\nu}$ into a union of $C\cap\{\text{Im}(z)\gtrless \pm \nu\}$ with $C$ as in \eqref{Econtour} and 
\begin{align*}
 \tilde \Upsilon_{3,\nu}^0 &= \{z\in \mathbb C \vert \text{Im}(z) = \pm \nu,\,\text{Re}(z)\in (-q(\nu^2),0)\}\cup \{z\in \mathbb C\vert  \vert z\vert =\nu,\,\text{Re}(z)\ge 0\}.
\end{align*}
The contribution from $C\cap\{\text{Im}(z)\gtrless \pm \nu\}$ is exponentially small $\mathcal O(e^{-2/3 \theta^{3/2}})$ as $\theta\rightarrow \infty$. Therefore
\begin{align*}
 \textnormal{Lc}_{\xi}(\theta) &=  \frac{1}{2i}\theta^{(1-\xi)/2} \int_{\tilde \Upsilon_{3,\nu}^0} e^{\theta^{3/2}(-z^3/3+z)} e^{-\xi \log_\pi(z)} dz+\mathcal O(e^{-2/3 \theta^{3/2}})\\
 &= \sin(\xi\pi) \theta^{(1-\xi)/2} \int_0^1 e^{\theta^{3/2}(\tau^3/3-\tau)} \tau^{-\xi} d\tau+\mathcal O(e^{-2/3 \theta^{3/2}}),
\end{align*}
upon $\nu\rightarrow 0^+$. The function
\begin{align*}
 G(\nu) = \nu^{\xi-1} \int_0^1 e^{\nu^{-1} (\tau^3/3-\tau) }\tau^{-\xi} d\tau = \int_0^{\nu^{-1}} e^{-s +\nu^2 s^3/3} s^{-\xi}ds,
\end{align*}
is continuous at $\nu=0$ with value
\begin{align*}
 G(0)= \int_0^1 e^{-w} w^{-\xi} dw= \Upsilon(1-\xi),
\end{align*}
where $\Upsilon$ on the right hand side is for the $\Upsilon$-function.
Therefore 
\begin{align*}
 \textnormal{Lc}_{\xi}(\theta) = \sin(\xi\pi) \theta^{(1-\xi)/2-3/2}\Upsilon(1-\xi)(1+o(1)),
 \end{align*}
in agreement with \eqref{x23ASuPos}.

\subsection{Asymptotics of $\textnormal{La}_{\xi}$ for $\theta\rightarrow -\infty$}\seclab{app:Aminf}
First we replace $\tau$ by $(-\theta)^{1/2}\tau$ and obtain
\begin{align*}
 \textnormal{La}_{\xi}(\theta) = (-\theta)^{(1-\xi)/2} \int_0^\infty e^{(-\theta)^{3/2} (-\tau^3-\tau)} \tau^{-\xi}d\tau.
\end{align*}
Then we apply the following substitution $$t(\tau)=\tau^3/3+\tau,$$
with inverse
\begin{align*}
 \tau = t(1+m(t^2)),
\end{align*}
with $m$ real analytic. Then 
\begin{align*}
 \textnormal{La}_{\xi}(\theta) =(-\theta)^{(1-\xi)/2} \int_0^\infty e^{-(-\theta)^{3/2}t}t^{-\xi} (1+n(t^2)) dt,
\end{align*}
where $n=(1+m)^{-\xi}\frac{d\tau}{dt}-1$ is real analytic with $n(0)=0$. Finally, setting $s=(-\theta)^{3/2} t$ gives
\begin{align*}
 \textnormal{La}_{\xi}(\theta) =(-\theta)^{-(1-\xi)} E_2((-\theta)^{-3}),
\end{align*}
with 
\begin{align*}
 E_2(w) =\int_0^\infty e^{-s}s^{-\xi} (1+n(w s^2)) ds,
\end{align*}
a real-analytic function, in agreement with \eqref{x21ASuNeg}.

\subsection{Asymptotics of $\textnormal{Lb}_{\xi}$ for $\theta\rightarrow -\infty$}\seclab{app:Bminf}
The asymptotics of $\textnormal{Lb}_{\xi}$ for $\theta\rightarrow -\infty$ is obtained by replacing $z$ by $(-\theta)^{1/2}z$. This gives
\begin{align*}
 \textnormal{Lb}_{\xi}(\theta) = \frac{1}{2i}e^{i\xi \pi} (-\theta)^{(1-\xi)/2} \int_{\Upsilon_{2,\nu}} e^{(-\theta)^{3/2}G(z)} e^{-\xi \log_0(z)} dz,
\end{align*}
where
\begin{align*}
 G(z) = -z^3/3-z.
\end{align*}
We then replace the contour $\Upsilon_{2,\nu}$ with a union of $\Upsilon_{1,\nu}$ and 
\begin{align*}
 \mathcal C^+&=\{z=u+iv\in \mathbb C\vert u = s_{+}(v),\,v>0\},\\
 \mathcal C^-&=\{z=u+iv\in \mathbb C\vert u = s_{-}(v),\,v<0\}
\end{align*}
where
\begin{align*}
 s_{\pm}(v) = \pm ({v\mp 1})\sqrt{\frac{v\pm 2}{3v}}.
\end{align*}
The contours $\mathcal C^\pm$ are also used to study the asymptotics of the Airy function $\text{Ai}(\theta)$ for $\theta\rightarrow -\infty$. Along $\mathcal C^\pm$ we have
$\text{Im}(G)=\mp \frac23$ and 
\begin{align*}
\text{Re}(G)=s_{\pm}(v)\left(-s_{\pm}(v)^2/3+ v^2-1\right)=2(v\mp 1)^2(1+\mathcal O(v\mp 1)).
\end{align*}
In particular, $\text{Re}(G)$ has a global maximum along $\mathcal C^\pm$ at $v=\pm 1$, respectively. This then leads to
\begin{align*}
 \textnormal{Lb}_{\xi}(\theta) &= \frac{1}{2i}e^{i\xi \pi} (-\theta)^{(1-\xi)/2} \int_{\mathcal C^+\cup \mathcal C^-} e^{(-\theta)^{3/2}G(z)} e^{-\xi \log_0(z)} dz+\frac{1}{2i}e^{i\xi \pi} (1-e^{-\xi 2\pi i}) \textnormal{La}_{\xi}(\theta)\\
 &=\sqrt{2\pi}(-\theta)^{(1-\xi)/2-3/4} \sin\left(\frac23 (-\theta)^{3/2}(1+o(1)) +\frac{\pi}{4}-\frac{\xi \pi}{2}\right)(1+o(1))\\
 &+\sin(\xi \pi)\textnormal{La}_{\xi}(\theta),
\end{align*}
in agreement with \eqref{x22ASuNeg}.

\subsection{Asymptotics of $\textnormal{Lc}_{\xi}$ for $\theta\rightarrow -\infty$}\seclab{app:Cminf}
In this case we replace $z$ by $(-\theta)^{1/2}z$ and deform $\Upsilon_{3,\nu}$ into $\mathcal C^+\cup \mathcal C^-$. The calculations are similar to the asymptotics of $\textnormal{Lb}_{\xi}$ as $\theta\rightarrow -\infty$. We obtain
\begin{align*}
\textnormal{Lc}_{\xi}(\theta) &= \sqrt{2\pi}(-\theta)^{(1-\xi)/2-3/4} \sin\left(\frac23 (-\theta)^{3/2}(1+o(1)) +\frac{\pi}{4}+\frac{\xi \pi}{2}\right)(1+o(1)),
\end{align*}
in agreement with \eqref{x23ASuNeg}.
\section{Proof of \propref{SaEpsEst}}\applab{appNF}

From \eqref{eq:entrychart} in Section \secref{sec:kappa1}, we obtain the following equations in terms of the $(\tilde y_0,\tilde w_0)$-variables defined by \eqref{tildeywTransformation}:
\begin{align}
\dot{\tilde y}_0 &=\epsilon_1 \xi\left(1+Q_1^{(0)}(\tilde y_0,\tilde w_0,r_1,\phi_1,\epsilon_1)\right)\tilde y_0+\left(1+Q_2^{(0)}(r_1,\phi_1,\epsilon_1)\right)\tilde w_0+r_1\epsilon_1 Q_0^{(0)}(r_1,\phi_1,\epsilon_1),\eqlab{qpEqn0}\\
\dot{\tilde w}_0 &=\left(-1+P_2^{(0)}(r_1,\phi_1,\epsilon_1)\right)\tilde y_0+\left(-\delta r_1+\frac12 \epsilon_1+P_1^{(0)}(\tilde y_0,\tilde w_0,r_1,\phi_1,\epsilon_1)\right)\tilde w_0+r_1\epsilon_1 P_0^{(0)}(r_1,\phi_1,\epsilon_1),\nonumber\\
\dot r_1 &=-\frac12 r_1\epsilon_1,\nonumber\\
\dot \epsilon_1 &=\frac32 \epsilon_1^2,\nonumber\\
\dot \phi_1 &=\epsilon_1 \left(\left((1-\xi)\phi_1+s\right)\left(1-\frac{\epsilon_1^2}{1-\xi}F_2(\epsilon_1^{2})\right)+V^{(0)}(\tilde y_0,\tilde w_0,r_1,\phi_1,\epsilon_1)\right),\nonumber
\end{align}
the equations defining new smooth functions $Q_i^{(0)}$, $P_{i}^{(0)}$, $V^{(0)}$ satisfying
\begin{align*}
Q_1^{(0)},\,Q_2^{(0)},\,P_1^{(0)},\,P_2^{(0)}=\mathcal O((\epsilon_1+r_1)^2+\varepsilon(\tilde y_0+\tilde w_0)),\quad V^{(0)} = \mathcal O(r_1(\tilde w_0+r_1+\epsilon_1)+\tilde y_0).
 \end{align*}
 Recall $\varepsilon = r_1^3\epsilon_1$. We will now consider this system in detail. 
 
 Specifically, since we cannot apply blowup and desingularization directly due to the fast oscillatory part (recall e.g. \eqref{lambda1pm}$_{r_1=0}$), we first apply normal form transformations in Section \secref{sec.nf} to factor out this oscillatory part. \ed{These normal form transformations are like higher order averaging. But the normal form approach circumvents the singularity associated with the zero amplitude that is known to appear when using averaging in this context.} Then in Section \secref{sec.vanderPol} we apply a van der Pol transformation (moving into a rotating coordinate frame), giving rise to \eqref{zFinal}. In Section \secref{sec.blowupFinal} the transformed normally elliptic line $L_1$ \eqref{eq:lineL1} can be studied using a second blowup and subsequent desingularization. Within chart \eqref{barr11},  we gain hyperbolicity for $r_1=0$ which allow us to extend the slow manifold $S_{a,\varepsilon}$ as a perturbation of $S_a$ up until $\theta=-(\varepsilon \nu^{-1})^{1/2}$ for $\nu>0$ sufficiently small but fixed with respect to $\varepsilon>0$. We extend this further into $\theta = -(\varepsilon \nu^{-1})^{2/3}$ (where $\epsilon_1= \nu$ cf. \eqref{kappa12New}) by applying the forward flow in a subsequent chart \eqref{bareps11} in Section \secref{sec.bareps11}. We find that $S_{a,\varepsilon}$ is $o(1)$-close to the manifold $C_{e,1}=\kappa_{12}(C_{e,2})$ \eqref{Ce2} of non-oscillatory solutions at the section defined by \eqref{thetaEps1_3}; see \propref{M2Est}, which working backwards then implies \propref{SaEpsEst}. 
 
\subsection{Normal form transformation}\seclab{sec.nf}
Let 
 \begin{align*}
\tilde V_1  = \{(\tilde y,r_1,(\tilde w_1,\phi_1,\epsilon_1))\in [-\sigma,\sigma]\times [0,\nu]\times \tilde U_1\},
\end{align*}
with $\tilde U_1 = [-\sigma,\sigma]\times [-\varpi^{-1},-\varpi ]\times [0,\nu]$. 

\begin{proposition}\proplab{NF}
 Fix any $n\in \mathbb N$. Then for $\nu>0$ sufficiently small there exists a smooth mapping
 \begin{align*}
  \Phi_n:\quad \tilde V_1 \ni (\tilde y_0,r_1,\tilde w_0,\phi_1,\epsilon_1)\mapsto (\tilde y_n,r_1,\tilde w_n,\phi_1,\epsilon_1),
 \end{align*}
leaving $r_1,\phi_1,\epsilon_1$ invariant,  that transforms \eqref{qpEqn0} into
 \begin{align}
\dot{\tilde y}_n &=\left(T^{(n)}(I_n,r_1,\phi_1,\epsilon_1)+Q_1^{(n)}(\tilde y_n,\tilde w_n,r_1,\phi_1,\epsilon_1)\right)\tilde y_n\nonumber\\
&+\left(1+\Omega^{(n)}(I_n,r_1,\phi_1,\epsilon_1) + Q_2^{(n)}(\tilde y_n,\tilde w_n,r_1,\phi_1,\epsilon_1)\right)\tilde w_n+r_1\epsilon_1 Q_0^{(n)}(r_1,\phi_1,\epsilon_1),\eqlab{normalForm}\\
\dot{\tilde w}_n &=-\left(1+\Omega^{(n)}(I_n,r_1,\phi_1,\epsilon_1)+P_2^{(n)}(\tilde y_n,\tilde w_n,r_1,\phi_1,\epsilon_1)\right)\tilde y_n\nonumber\\
&+\left(T^{(n)}(I_n,r_1,\phi_1,\epsilon_1)+P_1^{(n)}(\tilde y_n,\tilde w_n,r_1,\phi_1,\epsilon_1)\right)\tilde w_n+r_1\epsilon_1 P_0^{(n)}(r_1,\phi_1,\epsilon_1),\nonumber\\
\dot r_1 &=-\frac12 r_1\epsilon_1,\nonumber\\
\dot \epsilon_1 &=\frac32 \epsilon_1^2,\nonumber\\
\dot \phi_1 &=\epsilon_1 \left(\left((1-\xi)\phi_1+s\right)\left(1-\frac{\epsilon_1^2}{1-\xi}F_2(\epsilon_1^{2})\right)+V^{(n)}(\tilde y_n,\tilde w_n,r_1,\phi_1,\epsilon_1)\right),\nonumber
\end{align}
where 
\begin{align*}
 Q_1^{(n)},\,Q_1^{(n)},\,P_1^{(n)},\, P_2^{(n)}&= \mathcal O_1+\varepsilon \mathcal O_2,\quad 
Q_0^{(n)},\, P_0^{(n)} = \mathcal O_1,\quad  V^{(n)}=\mathcal O_3,
\end{align*}
where
\begin{align*}
 I_n = \vert (\tilde y_n,\tilde w_n)\vert^2=\tilde y_n^2+\tilde w_n^2,
\end{align*}
and
\begin{align*}
 \mathcal O_1 = \mathcal O(\vert (r_1,\epsilon_1)\vert^n),&\quad \mathcal O_2 = \mathcal O((\tilde y_n+\tilde w_n)\vert (r_1,\tilde y_n,\tilde w_n,\epsilon_1)\vert^{n-1}),\\
 \mathcal O_3&=\mathcal O(r_1(\tilde w_n+r_1+\epsilon_1)+\tilde y_n).
\end{align*}
Furthermore $T^{(n)}$ and $\Omega^{(n)}$ are $n$th-degree polynomials of $I_n,r_1,\epsilon_1$, with $\phi_1$-dependent coefficients, satisfying:
\begin{align}
 T^{(n)}(I_n,r_1,\phi_1,\epsilon_1)&=-\frac12 \delta r_1 +\frac12\left(\frac12 +\xi\right)\epsilon_1+\mathcal O((r_1+\epsilon_1)^2+\varepsilon I_n),\eqlab{TnOn}\\
 \Omega^{(n)}(I_n,r_1,\phi_1,\epsilon_1) &= \mathcal O((r_1+\epsilon_1)^2+\epsilon I_n).\nonumber
\end{align}

\end{proposition}
\begin{proof}
The linearization about
\begin{align*}
 (\tilde y_0,\tilde w_0,r_1,\epsilon_1,\phi_1) = (0,0,0,0,-(1-\xi)^{-1}s)
\end{align*}
gives 
\begin{align*}
  L =\begin{pmatrix}
       L_0 & 0_{2\times 3}\\
       0_{3\times 2} & 0_{3\times 3}
      \end{pmatrix},\quad L_0 =\begin{pmatrix} 
      0 & 1\\
      -1 &0
      \end{pmatrix}.
\end{align*}
By normal form theory, see e.g. \cite[Theorem 1.2 and Lemma 1.7]{unknown2011a}, 
the system can be brought into \eqref{normalForm} by successive transformations; the truncated system with $Q_i^{(n)}=P_{i}^{(n)}=0$ being equivariant with respect the action of $e^{tL}$. Simple calculations then give \eqref{TnOn}.
\end{proof}
We shall henceforth drop the subscripts on $r_1$, $\phi_1$ and $\epsilon_1$.

\subsection{Van der Pol transformation}\seclab{sec.vanderPol}
Now we apply the van der Pol transformation
\begin{align}
 \begin{pmatrix}
  \tilde y_n\\
  \tilde w_n
 \end{pmatrix}(t_1) = A(\psi(t_1)) z(t_1),\quad \dot \psi(t_1)&=1,\eqlab{vanderPol}
\end{align}
to \eqref{normalForm}, 
where $\psi\in S^1$, $z\in \mathbb R^2$, and
\begin{align}
 A(\psi) = \begin{pmatrix}
            \cos \psi & \sin \psi\\
            -\sin \psi & \cos \psi
           \end{pmatrix}\in SO(2),\eqlab{rotation}
\end{align}
to give the extended system:
\begin{align}
 \dot z &=\left(\begin{pmatrix}
           T^{(n)}(\vert z\vert^2,r,\phi,\epsilon) & \Omega^{(n)}(\vert z\vert^2,r,\phi,\epsilon)\\
           -\Omega^{(n)}(\vert z\vert^2,r,\phi,\epsilon) &T^{(n)}(\vert z\vert^2,r,\phi,\epsilon)
          \end{pmatrix}+R^{(n)}(z,r,\phi,\epsilon,\psi)\right)z \nonumber\\
          &+ \epsilon r \mathcal R^{(n)}(r,\phi,\epsilon,\psi),\eqlab{zFinal}\\
          \dot r &=-\frac12 r\epsilon,\nonumber\\
\dot \epsilon &=\frac32 \epsilon^2,\nonumber\nonumber\\
\dot \phi &=\epsilon \left(\left((1-\xi)\phi+s\right)\left(1-\frac{\epsilon^2}{1-\xi}F(\epsilon^{2})\right)+V^{(n)}(A(\psi)z,r,\phi,\epsilon)\right),\nonumber\\
\dot \psi&=1,\nonumber
\end{align}
on $(z,r,\phi,\epsilon,\psi)\in \mathbb R^{5}\times S^1$ with 
\begin{align*}
 \mathcal R^{(n)} = \mathcal O(\vert (r,\epsilon)\vert^n+r^3 \epsilon \vert z\vert \vert (r,z,\epsilon)\vert^{n-1}),\, \mathcal R^{(n)}=\mathcal O(\vert (r,\epsilon)\vert^n),\ V^{(n)} = \mathcal O(z+r_1(r_1+\epsilon_1)).
\end{align*}
Recall also \eqref{TnOn}.
We will work with system \eqref{zFinal} henceforth. By construction, we have the following lemma:
\begin{lemma}\lemmalab{Snu}
System \eqref{zFinal} possesses an $S^1$-symmetry:
 \begin{align}
  \mathcal S_\nu:\quad z\mapsto A(\nu)z,\quad \psi\mapsto \psi-\nu,\eqlab{SSymmetry}
 \end{align}
for every $\nu\in S^1$. 
\end{lemma}

\subsection{Subsequent blowup}\seclab{sec.blowupFinal}
Setting $\epsilon=0$ in \eqref{zFinal} gives $z=0$ as a set of equilibria, corresponding to $S_{a,1}$. Therefore it is also non-normally hyperbolic at $r=0$. Indeed $T^{(n)}(0,0,\phi,0)\equiv 0$. Therefore we apply the following polar blowup transformation to \eqref{zFinal}:
\begin{align}
 r = \rho \bar r,\quad \epsilon = \rho\bar \epsilon,\quad \rho\ge 0,(\bar r,\bar \epsilon)\in S^1,\eqlab{blowupFinal}
\end{align}
and desingularize through division of the right hand side by $\rho$. The transformation \eqref{blowupFinal} blows up $r=\epsilon=0$ to a sphere $(\bar r,\bar \epsilon)\in S^1$. We consider two directional charts
\begin{align}
\bar r=1:\quad r=\rho_1,\quad \epsilon =\rho_1 \epsilon_1,\quad \epsilon_1 \in I_1,\eqlab{barr11}
\end{align}
and
\begin{align}
\bar \epsilon=1:\quad r=\rho_2r_2,\quad \epsilon =\rho_2,\quad r_2\in I_2.\eqlab{bareps11}
\end{align}
Here $I_1$ and $I_2$ are sufficiently large open sets that contain $[0,\nu]$ and $[0,\nu^{-1}]$, respectively. In this way the two charts \eqsref{barr11}{bareps11} cover $(\bar r,\bar \epsilon)\in S^1$ with $\bar \epsilon\ge 0$, $\bar r\ge 0$. The coordinate changes are defined by 
\begin{align}
\rho_1 = \rho_2 r_2,\quad \epsilon_1 = r_2^{-1}. \eqlab{cchangehey}
\end{align}
Notice that the conservation $r^3\epsilon=\varepsilon$ in chart $\kappa_1$ becomes
\begin{align}
 \rho_1^4 \epsilon_1 = \varepsilon,\eqlab{rho1eps1}
\end{align}
and
\begin{align}
 \rho_2^4 r_2^3 = \varepsilon,\eqlab{rho2eps2}
\end{align}
in charts \eqsref{barr11}{bareps11}, respectively.

\subsection{Chart \eqref{barr11}}
In this chart we obtain the following set of equations from \eqref{zFinal}:
\begin{align}
 \dot z &=\left(\begin{pmatrix}
           T_1^{(n)}(\vert z\vert^2,\rho_1,\phi,\epsilon_1) & \Omega_1^{(n)}(\vert z\vert^2,\rho_1,\phi,\epsilon_1)\\
           -\Omega_1^{(n)}(\vert z\vert^2,\rho_1,\phi,\epsilon_1) &T_1^{(n)}(\vert z\vert^2,\rho_1,\phi,\epsilon_1)
          \end{pmatrix}+R_1^{(n)}(z,\rho_1,\phi,\epsilon_1,\psi)\right)z \nonumber\\
          &+ \mathcal R_1^{(n)}(\rho_1,\phi,\epsilon_1,\psi),\eqlab{z1Final}\\
          \dot \rho_1 &=-\frac12 \rho_1\epsilon_1,\nonumber\\
\dot \epsilon_1 &=2 \epsilon_1^2,\nonumber\nonumber\\
\dot \phi &=\epsilon_1 \left(\left((1-\xi)\phi+s\right)\left(1-\frac{\rho_1^2\epsilon_1^2}{1-\xi}F(\rho_1\epsilon_1^{2})\right)+V^{(n)}(A(\psi)z,\rho_1,\phi,\rho_1\epsilon_1)\right),\nonumber
\end{align}
and
\begin{align}
\dot \psi&=\rho_1^{-1},\eqlab{phi1Eqn}
\end{align}
after division of the right hand side by $\rho_1$. 
Here
\begin{align}
 T_1^{(n)}=-\frac12 \delta +\mathcal O(\rho_1+\epsilon_1),\quad
 \Omega_1^{(n)} &= \mathcal O(\rho_1+\epsilon_1),\eqlab{T1O1}
\end{align}
and
\begin{align}
 R_1^{(n)}(z,\rho_1,\phi,\epsilon_1,\psi) = \mathcal O(\rho_1^n),\quad  \mathcal R_1^{(n)}(\rho_1,\phi,\epsilon_1,\psi) = \mathcal O(\rho_1^{n+2} \epsilon_1),\, V^{(n)} = \mathcal O(z+\rho_1^2)\eqlab{remainderNew}
\end{align}
The system \eqref{z1Final} is still $\mathcal S_\nu$-symmetric, recall \lemmaref{Snu}.  We consider the following set
 \begin{align*}
  W_1:\quad z\in [-\sigma,\sigma]^2,\,\,\rho_1 \in [0,\nu],\,\epsilon_1 \in [0,\nu],\,\phi_1 \in[-\varpi^{-1},-\varpi],\,\psi \in S^1,
  \end{align*}
after possibly decreasing $\sigma>0$ slightly.

Notice that $\rho_1=0$ is singular in \eqref{phi1Eqn}, but the right hand sides of the $(z,\epsilon_1,\phi)$-equations are well-defined there cf. \eqref{z1Final}, \eqsref{T1O1}{remainderNew}. In particular, any point $(z,\epsilon_1,\phi_1)=(0,0,\phi_1)$ is an equilibrium of this system and the linearization has eigenvalues $0,-\frac12\delta$, both of algebraic multiplicity two. Therefore we have gained hyperbolicity, albeit with the $\psi$-equation \eqref{phi1Eqn} singular at $\rho_1=0$. This allows us to obtain the following:%
\begin{proposition}\proplab{Ma1k1}
 Fix $k\in \mathbb N$ and suppose $n\ge 2$. Then for $\nu>0$ sufficiently small the following holds: There exists an attracting locally invariant manifold $M_{a,1}$ of \eqref{z1Final} within $W_1$ as the following graph:
 \begin{align}
  M_{a,1}:\quad z = \rho_1^{n} \epsilon_1 A(\psi)^T m_1(\rho_1,\phi,\epsilon_1),\eqlab{Ma1k1}
 \end{align}
with $m_1(\cdot,\cdot,\cdot)$ Lipshitz continuous and $A(\psi)\in SO(2)$ (see \eqref{rotation}). Also the first $k$ partial derivatives with respect to $\phi$: $$\partial_{\phi^{i}}m_1(\rho_1,\phi,\epsilon_1), \quad \mbox{with $1\le i\le k$},$$ exist and are Lipshitz continuous. 
\end{proposition}
\begin{proof}
 See \appref{appMa1k1}. 
\end{proof}
\begin{remark}
 In the proof of \propref{Ma1k1} in \appref{appMa1k1}, we actually blowup $\rho_1=\epsilon_1=0,\,z=0$ further by introducing 
 \begin{align}
 z=\rho_1^n\epsilon_1 z_1.\eqlab{newz1}
 \end{align}
  The dynamics of $(z_1,\epsilon_1,\phi)$ is then well-defined for $\rho_1=0$. See \eqref{z1FinalAppNew2}. Recall that in \eqref{z1Final} the $\phi$-equation actually depends upon $\psi$ for $\rho_1=0$. It is therefore tempting to include $z=\bar \rho^n \bar z$ in the blowup \eqref{blowupFinal} (and apply a consecutive blowup of $\epsilon_1=0,z_1=0$ in the proof of \propref{Ma1k1} to finally obtain \eqref{newz1} in chart \eqref{barr11}). This approach might allow for improved estimates of $o(1)$ in \thmref{thm:main}, but we did not find an easy way to deal with the subsequent details in the chart \eqref{bareps11}. 
\end{remark}

The invariant manifold 
\begin{align}
M_{a,1}(\varepsilon)\equiv  M_{a,1}\cap \{\rho_1^4 \epsilon_1 = \varepsilon\},\eqlab{Ma1eps}
\end{align} 
can be viewed as an extension of Fenichel's slow manifold $S_{a,\varepsilon}$ up until $\theta=-(\varepsilon \nu^{-1})^{1/2}$ with $\phi \in [-\varpi^{-1}(\varepsilon\nu^{-1})^{1/2},-\varpi(\varepsilon\nu^{-1})^{1/2}]$ by setting  $\epsilon_1=\nu$ in \eqref{rho1eps1}, together with \eqsref{barr11}{kappa1},  for $\nu$ sufficiently small but fixed with respect to $\varepsilon$. Note that there is a uniform contraction along $M_{a,1}(\varepsilon)$. In terms of $(\tilde y_n,\tilde w_n)$, the invariant manifold $M_{a,1}$ becomes a graph over $(\rho_1,\phi_1,\epsilon_1)$:
\begin{align*}
 (\tilde y_n,\tilde w_n)= \rho_1^n \epsilon_1 m_1(\rho_1,\phi,\epsilon_1),
\end{align*}
(by \eqref{vanderPol} using $AA^T=I$), which is independent of $\psi$ as desired. 

From \eqref{z1Final}, the reduced problem on $M_{a,1}$ becomes
\begin{align*}
 \dot \rho_1 &=-\frac12 \rho_1,\nonumber\\
\dot \epsilon_1 &=2 \epsilon_1,\nonumber\nonumber\\
\dot \phi &=\left((1-\xi)\phi+s\right)\left(1-\frac{\rho_1^2\epsilon_1^2}{1-\xi}F(\rho_1\epsilon_1^{2})\right)+V^{(n)}(\rho_1^{n}\epsilon_1m_1(\rho_1,\phi,\epsilon_1),\rho_1,\phi,\rho_1\epsilon_1),
\end{align*}
after division of the right hand side by $\epsilon_1$. The reduced problem is also independent of $\psi$ as desired. Notice that 
\begin{align}
 p_1:\quad \phi=-\frac{s}{1-\xi},\,\rho_1=0,\epsilon_1=0,\nonumber
\end{align}
is hyperbolic. The invariant line
\begin{align*}
 \phi =-\frac{s}{1-\xi},\,\rho_1=0,\,\epsilon_1\ge 0,
\end{align*}
within $M_{a,1}\cap \{\rho_2=0\}$ 
corresponds to $l_2$, as given in \eqref{linel2}. As in \eqref{reducedCe1}, it is a strong unstable manifold of $p_1$ within $\rho_2=0$. The $1D$ stable manifold, contained within $\{\epsilon_1=0\}$, corresponds to the singular strong canard in this chart.

Setting $\epsilon_1=\nu$ gives $\rho_1=\left(\varepsilon/\nu\right)^{1/4}$ by the conservation \eqref{rho1eps1}. Therefore
\begin{align}
 M_{a,1}(\varepsilon)\cap \{\epsilon_1=\nu\}:\quad z=\left(\varepsilon/\nu\right)^{n/4} \nu A(\psi)^T m_1(\rho_1(\varepsilon),\phi,\nu)= A(\psi)^T \mathcal O(\varepsilon^{n/4}),\eqlab{Ma1k2}
\end{align}
cf. \eqref{Ma1k1}. Henceforth we suppose that $n\ge 4$. 
\subsection{Chart \eqref{bareps11}}\seclab{sec.bareps11}
In this chart we obtain the following equations from \eqref{zFinal}:
\begin{align}
 \dot z &=\left(\begin{pmatrix}
           T_2^{(n)}(\vert z\vert^2,r_2,\phi,\rho_2) & \Omega_2^{(n)}(\vert z\vert^2,r_2,\phi,\rho_2)\\
           -\Omega_2^{(n)}(\vert z\vert^2,r_2,\phi,\rho_2) &T_2^{(n)}(\vert z\vert^2,r_2,\phi,\rho_2)
          \end{pmatrix}+R_2^{(n)}(z,r_2,\phi,\rho_1,\psi)\right)z \nonumber\\
          &+ \mathcal R_2^{(n)}(r_2,\phi,\rho_1,\psi),\eqlab{z2Final}\\
          \dot r_2 &=-2r_2,\nonumber\\
\dot \phi &=\left((1-\xi)\phi+s\right)\left(1-\frac{\rho_2^2}{1-\xi}F(\rho_2^{2})\right)+V^{(n)}(A(\psi)z,\rho_2r_2,\phi,\rho_2),\nonumber\nonumber\\
\dot \rho_2 &=\frac32 \rho_2,\nonumber
\end{align}
and
\begin{align}
\dot \psi&=\rho_2^{-1},\eqlab{phi2Eqn}
\end{align}
after division of the right hand side by $\rho_2$. 
Here
\begin{align*}
 T_2^{(n)}=\frac12 \left(\frac12 +\xi\right)+\mathcal O(r_2+\rho_2),\quad
 \Omega_2^{(n)} &= \mathcal O(r_2+\rho_2),
\end{align*}
and
\begin{align}
 R_2^{(n)}(z,r_2,\phi,\rho_2,\psi) = \mathcal O(\rho_2^n),\quad  \mathcal R_2^{(n)}(r_2,\phi,\rho_2,\psi) = \mathcal O(\rho_2^{n+2} r_2).\eqlab{R2Est}
\end{align}
Also 
\begin{align*}
 V^{(n)}(A(\psi)z,\rho_2r_2,\phi,\rho_2) = \mathcal O(A(\psi) z+r_2(\rho_2+r_2)).
\end{align*}
As above, we notice that $\rho_2=0$ is well-defined for the right hand side of \eqref{z2Final}. But now $z=r_2=0,\,\phi=(1-\xi)^{-1}s$ is a hyperbolic equilibrium, the linearization having the real eigenvalues $\frac12 \left(\frac12 +\xi\right),\,-2,\,(1-\xi)$.


Let
 \begin{align*}
  W_2:\quad z\in [-\sigma,\sigma]^2,\,\,\rho_2 \in [0,\nu],\,r_2 \in [0,\nu],\,\phi_1 \in[-\varpi^{-1},-\varpi],\,\psi \in S^1,
  \end{align*}
\begin{proposition}\proplab{M2Est}
  Fix any $\eta\in (0,1)$, $n\ge 4$, and let $\nu$ be sufficiently small. Then for $0<\varepsilon\ll 1$ the forward flow of $M_{a,2}(\varepsilon)$ intersects the $\{\rho_2=\nu\}$-face of the box $W_2$ in a $C^1$-graph:
 \begin{align*}
  z = A(\psi)^T m_{2,\varepsilon}(\phi),
 \end{align*}
 with
 \begin{align}
  m_{2,\varepsilon}(\phi)=\mathcal O(\varepsilon^{\eta (7-2\xi)/24})=\mathcal O(\varepsilon^{5/24}),\quad m_{2,\varepsilon}'(\phi) = \mathcal O(\varepsilon^{1/12}).\eqlab{m2Est}
 \end{align}
\end{proposition}
\begin{proof}
Consider \eqref{z2Final} with $n\ge 4$. The manifold $M_{a,1}(\epsilon)$ from chart $\bar \rho=1$ enters the chart $\bar \epsilon=1$ \eqref{bareps11} at $r_2=\nu^{-1}$, cf. \eqref{cchangehey}, as a graph \eqref{Ma1k2}. We then apply a finite time flow map to go from $r_2=\nu^{-1}$ to the $\{r_2=\nu\}$-face of the box $W_2$, with $\nu$ small, which we then use as new initial conditions. By \eqref{Ma1k2} we then have $z(0)=A(\psi)^T \mathcal O(\varepsilon^{n/4})$; a $C^1$-graph over $(\phi,\psi)\in [-\varpi^{-1},-\varpi]\times S^1$. Subsequently we work in $W_2$ only and define an exit time $T$ by the condition $\rho_2(T) = \nu$. Solving the $\rho_2$-equation we obtain 
\begin{align}
T=\ln (\varepsilon^{-1/6}\nu^{7/6}), \eqlab{TEqnEst}
\end{align}
using $\rho_2(0)=\varepsilon^{1/4}\nu^{-3/4}$ by \eqref{rho2eps2}. 

Let 
\begin{align*}
\zeta=\frac12 \left(\frac12+\xi\right).
\end{align*}
Then from the $z$-equation we obtain
\begin{align*}
 \left( e^{-\zeta t}\vert z(t)\vert\right) \le \vert z(0)\vert +\int_0^t c_1\left(\nu \left(e^{-\zeta u}\vert z(u)\vert\right) +\varepsilon^{1/3}\right)du,
\end{align*}
 while $\phi\in [-\varpi^{-1},-\varpi]$, $r_2,\,\rho_2\le \nu$. 
 \begin{align*}
  \vert \mathcal R_2^{(n)}\vert \le c_1 \varepsilon^{1/3},
 \end{align*}
for all $\varepsilon\ll 1$. This follows from \eqsref{rho2eps2}{R2Est}. Then by Gronwall's inequality for every $\nu$ and $\varepsilon$ sufficiently small we have  that
\begin{align}
 \vert z(T) \vert &\le e^{(\zeta +c_1\nu)T} \vert z(0)\vert  +c_1e^{(\zeta +c_2\nu)T} \varepsilon^{1/3}\nonumber\\
 &\le c_3 \varepsilon^{1/3-(\zeta/6 +c_3\nu)}\le c_4 \varepsilon^{5/24},\eqlab{zTEst1}
\end{align}
using $n\ge 4$ where $c_2$, $c_3$ and $c_4(\xi)$ are sufficiently large. In the last equality we used the fact that $\zeta<\frac34$ and taken $\nu$ sufficiently small. This proves the first estimate in \eqref{m2Est}.

For the second estimate, we consider the variational equations obtained by differentiating the $(z,\,\phi)$-equations with respect to $\phi(0)=\phi_0$. This gives
\begin{align*}
 \left(e^{-\zeta t}\vert \tilde z(t)\vert \right)&\le \vert \tilde z(0)\vert +\int_0^t c_5\nu \left(\left(e^{-\zeta u}\vert \tilde z(u)\vert\right) +\varepsilon^{1/3-(\zeta/6 +c_3\nu)}e^{-\beta u}\left( e^{-(1-\xi)u} \vert \tilde \phi(u)\vert\right) \right)du,\\
 \left(e^{-(1-\xi) t}\vert \tilde \phi(t)\vert\right) &\le 1 +\int_0^t c_5\nu \left(\left(e^{-(1-\xi) u}\vert \tilde \phi(u)\vert \right)+ e^{\beta u} \left(e^{-\zeta u} \vert \tilde z(u)\vert \right)\right)du
\end{align*}
for $c_5$ sufficiently large, 
where for simplicity we have set $$\beta = \zeta-(1-\xi)=\frac34 (2\xi-1),$$ and introduced the following notation:
\begin{align*}
 \tilde z(t) = \frac{\partial z}{\partial \phi_0}(t),\quad \tilde \phi(t) = \frac{\partial \phi}{\partial \phi_0}(t).
\end{align*}
Notice $\tilde z(0)=\mathcal O(\varepsilon^{(n-2)/4})$ and $\tilde \phi(0)=1$. Then $m_{2,\varepsilon}'(\phi)$ in \eqref{m2Est} becomes $\tilde z(T)\tilde \phi(T)^{-1}$ by the chain rule. 
%
Suppose first that $\xi\le \frac12$ so that $\beta\le 0$. Then  
\begin{align}
 \left(e^{-\zeta t}\vert \tilde z(t)\vert \right)&\le \vert \tilde z(0)\vert +\int_0^t c_6\nu\left( \left(e^{-\zeta u}\vert \tilde z(u)\vert\right) +\varepsilon^{(1+\xi)/6-c_2\nu}\left( e^{-(1-\xi)u} \vert \tilde \phi(u)\vert\right) \right)du,\eqlab{tildezEqnEst}\\
 \left(e^{-(1-\xi) t}\vert \tilde \phi(t)\vert\right) &\le 1 +\int_0^t c_6\nu \left(\left(e^{-(1-\xi) u}\vert \tilde \phi(u)\vert \right)+ \left(e^{-\zeta u} \vert \tilde z(u)\vert \right)\right)du,\nonumber
\end{align}
for $t\in [0,T]$ 
with $c_6$ sufficiently large,
using here that
\begin{align*}
 \varepsilon^{1/3-(\zeta/6 +c_2\nu)}e^{-\beta u}&\le \varepsilon^{1/3-(\zeta/6 +c_2\nu)}e^{-\beta T} \le c_6 \varepsilon^{1/3-(\zeta-\beta)/6-c_2\nu} = c_6  \varepsilon^{(1+\xi)/6-c_2\nu}.
\end{align*}
for every $u\in [0,T]$, and all $\varepsilon$ sufficiently small. 
Therefore by Gronwall's inequality, the following estimate holds true for all $\nu$ sufficiently small
\begin{align}
 \vert \tilde z(t)\vert +\vert \tilde \phi(t)\vert \le c_7 e^{((1-\xi) +c_7\nu)t},\eqlab{initest}
\end{align}
taking $c_7$ sufficiently large and using that $(1-\xi)\ge \zeta$ given that $\beta\le 0$ by assumption. But then by \eqref{tildezEqnEst}
\begin{align*}
\left(e^{-\zeta t}\vert \tilde z(t)\vert \right)\le \vert \tilde z(0)\vert + c_8\nu\left(\int_0^t \left(e^{-\zeta u}\vert \tilde z(u)\vert\right)du + \varepsilon^{(1+\xi)/6-c_8\nu}\right).
\end{align*}
using \eqref{initest} to estimate $e^{-(1-\xi)t}\vert \tilde \phi(t)\vert \le  c_7 e^{c_7\nu}$. 
For $\nu$ sufficiently small we therefore have by Gronwall's inequality that
\begin{align*}
 \vert \tilde z(T)\vert \le c_9 e^{(\zeta+c_9\nu) T} \varepsilon^{(1+\xi)/6-c_9\nu}\le c_9 \varepsilon^{1/8} \le c_9 \varepsilon^{1/12},
\end{align*}
for all $\varepsilon$ sufficiently small. Here we have used \eqref{TEqnEst} and the fact that
\begin{align*}
(1+\xi)/6-\zeta/6> \frac{1}{8}.
\end{align*}
\ed{Now, given $\tilde z(t)$ then the equation for $\tilde \phi$ is a linear, scalar and non-autonomous ODE. Solving this linear equation and using the estimate on $\tilde z$ it is then straightforward to estimate 
$\vert \tilde \phi(T)\vert\ge C_3 \varepsilon^{-(1+\xi)/6+C_3^{-1}\nu}\ge C_3$ uniformly from below for $C_3>0$ and $\nu>0$ sufficiently small and all $0<\varepsilon\ll 1$.}
This allows us to estimate $m_{2,\varepsilon}'(\phi)$ for $\xi\le \frac12$ as follows
\begin{align*}
 \vert m_{2,\varepsilon}'(\phi) \vert \le C_3^{-1} c_9 \varepsilon^{1/8} \le C_3^{-1} c_9 \varepsilon^{1/12}.
\end{align*}

Now suppose that $\xi>\frac12$ so that $\beta> 0$. Then we scale $\tilde z$ as
\begin{align*}
\tilde z(t) = e^{-\beta T} \hat z(t),
\end{align*}
introducing $\hat z(t)$. 
This gives
\begin{align*}
 \left(e^{-\zeta t}\vert \hat z(t)\vert \right)&\le \vert \hat z(0)\vert +\int_0^t c_{10}\nu \left(\left(e^{-\zeta u}\vert \hat z(u)\vert\right) +\varepsilon^{1/3-(\zeta/6 +c_2\nu)}e^{\beta (T-u)}\left( e^{-(1-\xi)u} \vert \tilde \phi(u)\vert\right) \right)du\\
 &\le \vert \hat z(0)\vert +\int_0^t c_{11}\nu \left(\left(e^{-\zeta u}\vert \hat z(u)\vert\right) +\varepsilon^{1/12-c_{12}\nu}\left( e^{-(1-\xi)u} \vert \tilde \phi(u)\vert\right) \right)du\\
 \left(e^{-(1-\xi) t}\vert \tilde \phi(t)\vert\right) &\le 1 +\int_0^t c_{10}\nu \left(\left(e^{-(1-\xi) u}\vert \tilde \phi(u)\vert \right)+ e^{-\beta (T-u)} \left(e^{-\zeta u} \vert \hat z(u)\vert \right)\right)du\\
 &\le 1+ \int_0^t c_{10}\nu \left(\left(e^{-(1-\xi) u}\vert \tilde \phi(u)\vert \right)+ \left(e^{-\zeta u} \vert \hat z(u)\vert \right)\right)du,
\end{align*}
for $t\in [0,T]$. 
Hence
\begin{align}
 \vert \hat z(t)\vert +\vert \tilde \phi(t)\vert \le c_{13} e^{(\zeta +c_{14}\nu)t},\eqlab{initest2}
\end{align}
for $\nu$ sufficiently small. 
But then
\begin{align*}
 \left(e^{-\zeta t}\vert \hat z(t)\vert \right)&
 \le \vert \hat z(0)\vert +c_{15}\nu \left(\int_0^t  \left(e^{-\zeta u}\vert \hat z(u)\vert\right)du  +\varepsilon^{1/3-(\zeta/6 +c_{16}\nu)} e^{\beta T}\right),
\end{align*}
since  $e^{-(1-\xi)t}\vert \tilde \phi(t)\vert \le  c_{13} e^{c_{14}\nu t}$ by \eqref{initest2}.
Now we return to $\tilde z$ by multiplying through by $e^{-\beta T}$. This gives
\begin{align*}
  \left(e^{-\zeta t}\vert \tilde z(t)\vert \right)&
 \le \vert \tilde z(0)\vert +c_{17}\nu \left(\int_0^t  \left(e^{-\zeta u}\vert \tilde z(u)\vert\right) du+\varepsilon^{1/3-(\zeta/6 +c_{18}\nu)}  \right).
\end{align*}
Then by Gronwall's inequality
\begin{align*}
 \vert \tilde z(T)\vert\le c_{19} e^{\zeta T+c_{17}\nu } \varepsilon^{1/3-(\zeta/6 +c_{18}\nu)} \le c_{20} \varepsilon^{1/12},
\end{align*}
using \eqref{TEqnEst} and 
\begin{align*}
1/3 - {\zeta}/{3}> \frac{1}{12}.
\end{align*}
As above, we can easily estimate $\vert \tilde \phi(T)\vert\ge C_3$ uniformly from below. This completes the proof of the estimate for $m_{2,\varepsilon}'(\phi)$ in \eqref{m2Est}.

\end{proof}

\propref{M2Est} implies \propref{SaEpsEst} since $\epsilon_1=\rho_2$. This therefore completes the proof.
\section{Proof of \propref{Ma1k1}}\applab{appMa1k1}

Consider \eqref{z1Final}-\eqref{remainderNew} and set $z = \rho_1^{n}\epsilon_1 z_1$ to get
\begin{align}
 \dot{z_1} &=-\frac12 \delta z_1+\tilde R(z_1,\rho_1,\epsilon_1,\phi,\psi),\eqlab{z1FinalAppNew}\\
          \dot \rho_1 &=-\frac12 \rho_1\epsilon_1,\nonumber\\
\dot \epsilon_1 &=2 \epsilon_1^2, \nonumber\\
\dot \phi &=\epsilon_1 ((1-\xi)\phi+s)+ \tilde V(A(\psi)z_1,\rho_1,\phi,\epsilon_1),\nonumber\nonumber\\
\dot \psi&=\rho_1^{-1},\nonumber
\end{align}
where now, using $n\ge 2$,
\begin{align*}
 \tilde R(z_1,\rho_1,\epsilon_1,\psi) &=\frac12 n\epsilon_1^2 z_1-2\epsilon_1z_1 +\Bigg(\begin{pmatrix}
            T_1^{(n)}(\vert z\vert^2,\rho_1,\phi,\epsilon_1) & \Omega_1^{(n)}(\vert z\vert^2,\rho_1,\phi,\epsilon_1)\\
           -\Omega_1^{(n)}(\vert z\vert^2,\rho_1,\phi,\epsilon_1) &T_1^{(n)}(\vert z\vert^2,\rho_1,\phi,\epsilon_1)
          \end{pmatrix}\\
          &+\frac12 \delta I +R_1^{(n)}(z,\rho_1,\phi,\epsilon_1,\psi)\Bigg)z_1+\rho_1^{-n}\epsilon_1^{-1}\mathcal R_1^{(n)}(\rho_1,\phi,\epsilon_1,\psi)\\
          &=\mathcal O(z_1(\epsilon_1+\rho_1)+\rho_1^2),\\
 \tilde V(A(\psi)z_1,\rho_1,\phi,\epsilon_1)&=\epsilon_1\left(-\left((1-\xi)\phi+s\right)\frac{\rho_1^2\epsilon_1^2}{1-\xi}F(\rho_1\epsilon_1^{2})+V^{(n)}(\rho_1^{n}\epsilon_1 A(\psi)z_1,\rho_1,\phi,\rho_1\epsilon_1)\right),\\
 &=\mathcal O(\epsilon_1(\epsilon_1+\rho_1)),
\end{align*}
are both smooth functions. In particular,  $\tilde R(z_1,0,\epsilon_1,\psi)$ and $\tilde V(A(\psi)z_1,0,\phi,\epsilon_1)$ are both independent of $\psi$. 

%
By modifying the standard proof of the existence of a center manifold using the contraction mapping theorem, we can now prove the existence of a locally invariant manifold $M_{a,1}$. We provide all of the details below. It will be useful to introduce $\omega$ and $\hat \omega$ as $\omega=(\hat \omega,\phi)$ where $\hat \omega=(\rho_1,\epsilon_1)$. Furthermore, let $\Psi:\,\mathbb R\rightarrow [0,1]$ be a $C^\infty$ cut-off function satisfying $\Psi(-x)=\Psi(x)$, $\Psi(x)=1$ for all $x\in [0,1]$, and $\Psi(x)=0$ for all $x \ge 2$. Similarly, we let $\Phi:\,\mathbb R\rightarrow [0,1]$ be a $C^\infty$ function satisfying
\begin{align*}
 \Phi\vert_{[-\varpi^{-1},-\varpi]} = 1,\quad \Phi\vert_{(-\infty,-2\varpi^{-1})\cup (-\varpi^{-1}/2,\infty)}=0.
\end{align*}
Let $\sigma>0$. We then consider the following modified system 
\begin{align}
 \dot{ z_1} &=-\frac12 \delta  z_1+ \tilde R( z_1,\omega,\psi),\eqlab{z1FinalAppNew2}\\
          \dot \rho_1 &=-\frac12 \Psi \left(\frac{\vert \hat \omega\vert}{\sigma}\right) \rho_1\epsilon_1,\nonumber\\
\dot \epsilon_1 &=2 \Psi\left(\frac{\vert\hat \omega\vert}{\sigma}\right) \epsilon_1^2, \nonumber\\
\dot \phi &=\tilde P(z_1,\omega,\psi),\nonumber\nonumber\\
\dot \psi&=\rho_1^{-1},\nonumber
\end{align}
where
\begin{align*}
 \tilde R(z_1,\omega,\psi) &= \Psi \left(\frac{\vert\hat \omega\vert}{\sigma}\right)\Phi\left(\phi\right) R( z_1,\omega,\psi),\\
 \tilde P(z_1,\omega,\psi)&=\Psi \left(\frac{\vert \hat \omega\vert}{\sigma}\right)\Phi \left(\phi\right)\left(\epsilon_1 ((1-\xi)\phi+s)+ \tilde V(A(\psi) z_1,\omega)\right).
\end{align*}
Also $\tilde R$ and $\tilde P$ are $\mathcal S_\nu$-equivariant and $\mathcal S_\nu$-invariant, respectively, recall \eqref{SSymmetry}. Let $B_r(\hat \omega_0)=\{\hat \omega\in \mathbb R^2 \vert \vert \hat \omega-\hat \omega_0\vert <r\}$ denote the open desk centered at $\hat \omega_0$ with radius $r$.  Then 
notice that (a) \eqref{z1FinalAppNew2} coincides with \eqref{z1FinalAppNew} within $\vert \hat \omega\vert\le \sigma$, $\phi \in [-\varpi^{-1},-\varpi]$, cf. the definition of $\Psi$ and $\Phi$, and (b) $\hat \omega_0\in B_{2\sigma}(0)$ implies that $\hat \omega(t)\in B_{2\sigma}(0)$ for all $t$. We therefore consider the following set
\begin{align*}
 \mathcal W = \{ \omega =(\hat \omega,\phi)\in\overline{B_{2\sigma}(0)}\times \mathbb R\}.
\end{align*}
\begin{lemma}\lemmalab{RPEst}
There exists a constant $C_1>0$ so that the following estimates hold 
 \begin{align*}
\vert \tilde R(z_1',\omega',\psi)-\tilde R(z_1,\omega,\psi)\vert &\le C_1\sigma \left(\vert \omega'-w\vert+\vert z_1'-z_1\vert\right),\\
  \vert \tilde P(z_1',\omega',\psi)-\tilde P(z_1,\omega,\psi)\vert &\le C_1\left(\vert \epsilon_1'-\epsilon_1\vert+\sigma\left(\vert \phi'-\phi\vert+\vert \rho_1'-\rho_1\vert+\vert z_1'-z_1\vert\right)\right),  
 \end{align*}
 and
 \begin{align*}
   \vert \tilde R(z_1,\omega,\psi)\vert&\le C_1\sigma^2,
 \end{align*}
 for all $\omega '= (\rho_1',\epsilon_1',\phi'),\omega=(\rho_1,\epsilon_1,\phi)\in \mathcal W$, $\vert z_1\vert,\,\vert z_1'\vert \le \sigma$ and $\psi\in S^1$.

\end{lemma}
\begin{proof}
Straightforward.
\end{proof}

For $p_0>0$ and $p_1>0$ we then define $\mathcal X(p_0,p_1)$ as the set of Lipschitz functions $h:\mathcal W\times S^1 \rightarrow \mathbb R^2$ satisfying:
\begin{align*}
 h(0,\psi)=0,\quad\vert h(\omega,\psi)\vert \le p_0,\quad \vert h(\omega+\upsilon,\psi)-h(\omega,\psi)\vert\le p_1 \vert \upsilon\vert,\quad \forall\,\,\omega,\omega+\upsilon\in \mathcal W.
\end{align*}
With the supremum norm $$\Vert h \Vert=\sup_{(\omega,\psi) \in \mathcal W \times S^1}\vert h(\omega,\psi) \vert,$$ $\mathcal X(p_0,p_1)$ is complete. 

For $h\in \mathcal X(p_0,p_1)$ and $\omega_0=(\rho_{10},\epsilon_{10},\phi_{10})\in \mathcal W$ with $\rho_{10}\ne 0$, we let $\left(\omega(t,\omega_0,\psi_0,h),\,\psi(t,\omega_0,\psi_0,h)\right)$ be the solution of 
\begin{align*}
\dot \rho_1 &=-\frac12 \Psi \left(\frac{\vert \hat \omega\vert}{\sigma}\right) \rho_1\epsilon_1,\nonumber\\
\dot \epsilon_1 &=2 \Psi \left(\frac{\vert\hat \omega\vert}{\sigma}\right) \epsilon_1^2, \nonumber\\
\dot \phi &=\tilde P(h,\omega,\psi) = \Psi \left(\frac{\vert \hat \omega\vert}{\sigma}\right) \Phi \left(\phi\right) \left(\epsilon_1 ((1-\xi)\phi+s)+ \tilde V(A(\psi) h(w),\omega)\right),\nonumber\nonumber\\
\dot \psi&=\rho_1^{-1},
\end{align*}
satisfying:
\begin{align*}
\left(\omega (0,\omega_0,\psi_0,h),\,\psi(0,\omega_0,\psi_0,h)\right)=\left(\omega_0,\,\psi_0\right).
\end{align*}
For $\rho_{10}=0$ we define $\omega(t,\omega_0,\psi_0,h)$ similarly. Here it is cf. \eqref{z1FinalAppNew} simply independent of $\psi_0$. Finally, we set $\psi(t,\omega_0,\psi_0,h)=\psi_0$ when $\rho_{10}=0$ for all $t$. This particular choice is not important. 

\begin{lemma}\lemmalab{WpWEst1}
Let $\omega'=(\rho_1',\epsilon_1',\phi')= \omega(t,\omega_0',\psi_0,h)$ and $\omega =(\rho_1,\epsilon_1,\phi)=\omega(t,\omega_0,\psi_0,h)$ with $h\in \mathcal X(p_0,p_1)$ and $\omega_0'=(\rho_{10}',\epsilon_{10}',\phi_0'),\omega_0=(\rho_{10},\epsilon_{10},\phi_0)\in \mathcal W$. Then there exists a constant $C_2>0$ so that the following estimates hold
 \begin{align*}
  \vert \epsilon_1'-\epsilon_1\vert &\le e^{-C_2\sigma t} \vert \epsilon_{10}'-\epsilon_{10}\vert,\\
  \vert \rho_1'-\rho_1\vert &\le C_2e^{-C_2\sigma t}\left(\vert \rho_{10}'-\rho_{10}\vert+\vert \epsilon_{10}'-\epsilon_{10}\vert\right),\\
  \vert \phi'-\phi\vert&\le C_2(-t)e^{-C_2\sigma t}\left(\vert \rho_{10}'-\rho_{10}\vert+\vert \epsilon_{10}'-\epsilon_{10}\vert+\vert \phi_{0}'-\phi_0\vert\right),
   \end{align*}
for $t\le 0$. 
\end{lemma}
\begin{proof}
 From the $\epsilon_1$-equation we directly obtain
 \begin{align*}
  \vert \epsilon_1'(t)-\epsilon_1(t)\vert \le \vert \epsilon_{10}'-\epsilon_{10}\vert+\int_t^0 c_1\sigma \vert \epsilon_1'(\tau)-\epsilon(\tau)\vert d\tau,
 \end{align*}
 for $c_1>0$ sufficiently large, and therefore by Gronwall's inequality
\begin{align*}
 \vert \epsilon_1'-\epsilon_1\vert \le e^{-c_1\sigma t}\vert \epsilon_{10}'-\epsilon_{10}\vert,\quad t\le 0.
\end{align*}
But then from the $\rho_1$-equation
\begin{align*}
 \vert \rho_1'-\rho_1\vert &\le \vert \rho_{10}'-\rho_{10}\vert+c_2 \sigma \int_t^0 \left(\vert \epsilon_1'(\tau)-\epsilon_1(\tau)\vert+\vert \rho_{1}'(\tau)-\rho_{1}(\tau)\vert\right)d\tau\\
 &\le \vert \rho_{10}'-\rho_{10}\vert+c_3 e^{-c_1 \sigma t} \vert \epsilon_{10}'-\epsilon_{10}\vert+c_2 \sigma \int_t^0\vert \rho_{1}'(\tau)-\rho_{1}(\tau)\vert d\tau
\end{align*}
for $c_3>c_2>0$ sufficiently large. Then by Gronwall's inequality 
\begin{align*}
 \vert \rho_1'-\rho_1\vert \le c_3 e^{-c_4\sigma t}\left(\vert \rho_{10}'-\rho_{10}\vert+\vert \epsilon_{10}'-\epsilon_{10}\vert\right)
\end{align*}
for $c_4>0$ sufficiently large. Finally, from the $\phi$-equation:
\begin{align*}
 \vert \phi'(t)-\phi(t)\vert &\le \vert \phi'_0-\phi_0\vert + c_5\int_t^0 \big(\vert \epsilon_1'(\tau)-\epsilon_1(\tau)\vert+\sigma\big(\vert \phi'(\tau)-\phi(\tau)\vert\\
 &+\vert \rho_1'(\tau)-\rho_1(\tau)\vert\big)\big)d\tau\\
 &\le \vert \phi'_0-\phi_0\vert + c_6(-t)e^{-c_7\sigma t} \left(\vert \epsilon_{10}'-\epsilon_{10}\vert+\vert \rho_{10}'-\rho_{10}\vert\right)\\
 &+c_5\sigma \int_t^0 \vert \phi'(\tau)-\phi(\tau)\vert d\tau,
\end{align*}
using \lemmaref{RPEst} and that $h\in \mathcal X(p_0,p_1)$. Therefore
\begin{align*}
 \vert \phi'(t)-\phi(t)\vert\le c_8(-t)e^{-c_9\sigma t} \left(\vert \epsilon_{10}'-\epsilon_{10}\vert+\vert \rho_{10}'-\rho_{10}\vert+\vert \phi'_0-\phi_0\vert\right),
\end{align*}
for $c_8>0$ and $c_9>0$ sufficiently large. This gives the desired result. 
\end{proof}
\begin{lemma}\lemmalab{WpWEst2}
Let $\omega'=(\rho_1',\epsilon_1',\phi')\equiv \omega(t,\omega_0,\psi_0,h')$ and $\omega =(\rho_1,\epsilon_1,\phi)=\omega(t,\omega_0,\psi_0,h)$ with $h',\,h\in \mathcal X(p_0,p_1)$ and $\omega_0=(\rho_{10},\epsilon_{10},\phi_0)\in \mathcal W$. Then $\epsilon_1'=\epsilon_1$, $\rho_1'=\rho_1$ and there exists a constant $C_3>0$ so that the following estimate holds
 \begin{align*}
  \vert \phi'-\phi\vert&\le C_3\sigma  (-t) e^{-C_3\sigma t} \Vert h'-h\Vert,
   \end{align*}
for $t\le 0$. 
\end{lemma}
\begin{proof}
 The $\epsilon_1$ and $\rho_1$-equations are independent of $h$. Therefore by the $\phi$-equation
 \begin{align*}
  \vert \phi'(t)-\phi(t)\vert& \le c_1 \sigma \int_t^0 \left(\vert \phi'(\tau)-\phi(\tau)\vert + \Vert h'-h\Vert \right)d\tau,
 \end{align*}
 using \lemmaref{RPEst}, 
and then by Gronwall's inequality
\begin{align*}
 \vert \phi'-\phi\vert \le c_1\sigma  (-t) e^{-c_1\sigma t}\Vert h'-h\Vert.
\end{align*}
\end{proof}
Finally, we define $\mathcal T:\mathcal X(p_0,p_1)\rightarrow \mathcal X(p_0,p_1)$ as
\begin{align*}
 (\mathcal T h)(\omega_0,\psi_0) = \int_{-\infty}^0 e^{\frac12 \delta t} \tilde R(h(\omega,\psi),\omega,\psi) dt,
\end{align*}
where for simplicity
\begin{align*}
 \omega =\omega(t,\omega_0,\psi_0,h),\quad \psi = \psi(t,\omega_0,\psi_0,h).
\end{align*}

\begin{proposition}
 For $p_0$ and $\sigma$ sufficiently small, $\mathcal T$ is a contraction on $\mathcal X(p_0,p_1)$.  
\end{proposition}
\begin{proof}
We set
\begin{align*}
 p_0=\sigma.
\end{align*}
Then we show that $\mathcal T:\mathcal X(p_0,p_1)\rightarrow \mathcal X(p_0,p_1)$ is well-defined. Using \lemmaref{RPEst}, we obtain
 \begin{align*}
  \vert (\mathcal T h)(\omega_0,\psi_0)\vert \le c_1 \sigma^2,
 \end{align*}
for any $\sigma>0$, with $c_1>0$ sufficiently large.  Thus $\vert (\mathcal T h)(\omega_0,\psi_0)\vert\le p_0=\sigma$ for $\sigma$ sufficiently small. Next, we have
\begin{align*}
 \vert (\mathcal T h)(\omega_0',\psi_0)-(\mathcal T h)(\omega_0,\psi_0)\vert \le \int_{-\infty}^0 e^{\frac12 \delta t} C_1 (1+p_1)\sigma\vert \omega'(t)-\omega(t)\vert dt,
 \end{align*}
 using \lemmaref{RPEst}. 
 Therefore by \lemmaref{WpWEst1}
 \begin{align*}
  \vert (\mathcal T h)(\omega_0',\psi_0)-(\mathcal T h)(\omega_0,\psi_0)\vert&\le C_1(1+p_1) \sigma \int_{-\infty}^0 (-t)e^{\frac12 \delta t-C_2\sigma t} dt \vert \omega_0'-\omega_0\vert\\
  &\le c_2(1+p_1)\sigma \vert \omega_0'-\omega_0\vert,
 \end{align*}
 with $c_2>0$,
for all $\sigma$ sufficiently small. Therefore $\mathcal T$ is well-defined. Finally, 
\begin{align*}
 \vert (\mathcal T h')(\omega_0,\psi_0)-(\mathcal T h)(\omega_0,\psi_0)\vert &\le C_1 \sigma \int_{-\infty}^0 e^{\frac12 \delta t} \left(\vert \phi'(t)-\phi(t)\vert+\Vert h'-h\Vert\right)dt \\
 &\le c_3 \sigma \int_{-\infty}^0 (-t) e^{\frac12 \delta t-C_3\sigma t}  dt \Vert h'-h\Vert \le c_4\sigma \Vert h'-h\Vert,
\end{align*}
by \lemmaref{WpWEst2}, for $c_4,\,c_3>0$ sufficiently large, and all $\sigma$ sufficiently small. The result then follows. 
\end{proof}
The contraction mapping theorem guarantees the existence of a unique fixed point $h_*\in \mathcal X(p_0,p_1)$ of $\mathcal T$. The graph of $h_*$ is our center manifold. The function $h_*$ is $C^k$-smooth in $\phi$. The key observation here is that $\psi$ only depends upon $\hat \omega$; it is independent of $\phi$.  The result is therefore standard, following almost identical arguments to those used above. We skip the details. The smoothness in $\rho_1,\epsilon_1$ is more delicate, but we do not need it for our purposes. 

The following lemma completes the proof of \propref{Ma1k1}.
\begin{lemma}
 The fixed point $h_*(\omega,\psi)$ of $\mathcal T$ on $\mathcal X(p_0,p_1)$ satisfies:
 \begin{align*}
  h_*(\omega,\psi) = A(\psi)^T m_1(\omega).
  \end{align*}

\end{lemma}
\begin{proof}
 The modified system \eqref{z1FinalAppNew2} is $\mathcal S_\nu$-equivariant, recall \lemmaref{Snu}. This implies, by the uniqueness of $h_*$, that 
 \begin{align*}
  z_1 = A(\nu)^T h_*(\omega,\psi-\nu) = h_*(\omega,\psi),
 \end{align*}
for all $\nu\in S^1$. Setting $\nu=\psi$ and $m_1(\omega)=h_*(\omega,0)$ gives the desired result.
\end{proof}

 \end{document}